\documentclass[12pt]{article}
\usepackage{amsfonts}
\usepackage{amssymb}
\usepackage{amsmath}
\usepackage{amsthm}
\usepackage{color}
\usepackage{caption}
\usepackage{graphics}
\usepackage{enumerate}
\usepackage{epsfig,amssymb,latexsym,verbatim}
\usepackage{graphicx}
\usepackage{multirow}
\usepackage{multicol}
\usepackage{float}
\usepackage{hyperref}
\usepackage{graphicx}
\usepackage{amsfonts}
\usepackage{epsfig}
\usepackage[authoryear,sort]{natbib}
\usepackage{booktabs}
\usepackage{bm}
\usepackage{subcaption}
\allowdisplaybreaks

\numberwithin{equation}{section}

\newtheorem{lem}{Lemma}

\newtheorem{rem}{Remark}

\definecolor{pink}{rgb}{1,0.75,0.8}
\definecolor{grey}{rgb}{0.745,0.745,0.745}

\def\delete#1{\iffalse #1 \fi}

\newtheorem{theorem}{\bf Theorem}%[section]
\newtheorem{corollary}{\bf Corollary}%[section]
\newtheorem{lemma}{\bf Lemma}

\def\bse{\begin{eqnarray*}}
\def\ese{\end{eqnarray*}}
\def\bee{\begin{enumerate}}
\def\eee{\end{enumerate}}
\def\bqe{\begin{eqnarray}}
\def\eqe{\end{eqnarray}}
\def\bed{\begin{description}}
\def\eed{\end{description}}
\def\bei{\begin{itemize}}
\def\eei{\end{itemize}}

\def\pmb#1{\setbox0=\hbox{#1}%
    \kern-.025em\copy0\kern-\wd0
    \kern.05em\copy0\kern-\wd0
    \kern-.025em\raise.0433em\box0 }
\def\pmbh#1#2{\setbox0=\hbox{#1}%
    \setbox1=\hbox{#2}%
    \kern-.025em\copy0\kern-\wd0
    \kern.05em\copy1\kern-\wd0
    \kern-.025em\raise.0433em\box0 }

\def\frac#1#2{{#1\over#2}}

\def\boxit#1{\vbox{\hrule\hbox{\vrule\kern6pt
   \vbox{\kern6pt#1\kern6pt}\kern6pt\vrule}\hrule}}

\def\listing#1{\vskip 4mm\begin{verbatim}\input#1 \vskip 4mm}
\def\thick#1{\hbox{\rlap{$#1$}\kern0.25pt\rlap{$#1$}\kern0.25pt$#1$}}

\def\pmbh{{\pmb h}}

\renewcommand\today{\ifcase\month\or
   Jan\or Feb\or Mar\or Apr\or May\or
   Jun\or Jul\or Aug\or Sep\or Oct\or Nov\or
   Dec\fi
   \space\number\day, \number\year}

\let\hat\widehat
\let\tilde\widetilde
\def\@evenhead{\vbox{\hbox to\textwidth{\footnotesize \hfill
Liang.. \hfill } }}
\def\@oddhead{\vbox{\hbox to \textwidth{\footnotesize\hfill
SCFM \hfill } }}

\def\tit.arg{{\bf From Model Selection to Model Averaging: A Comparison for Nested Linear Models}
}

\def\key.arg{
}

\def\author.arg{
Wenchao Xu and Xinyu Zhang \\
Academy of Mathematics and Systems Science, Chinese Academy of Sciences, Beijing 100190, China\\
{\em xinyu@amss.ac.cn}
}

\def\abst.arg{

}

\begin{document}
\baselineskip=18pt

%%%%%%%%%%%%%%%%%%%%%%%%%%%%%%%%%%%%%%%%%%
\thispagestyle{plain}
\thispagestyle{empty} %{plain}

%%%%%%%%%%%%%%%%%%%%%%%%%%%%%%%%%%%%%%%%%%

\begin{center}
{\Large \tit.arg}

%\vskip 6mm
%\today
\vskip 6mm

\author.arg
\end{center}

\vskip 3mm
\centerline{\small ABSTRACT} 
Model selection (MS) and model averaging (MA) are two popular approaches when having many candidate models. Theoretically, the estimation risk of an oracle MA is not larger than that of an oracle MS because the former one is more flexible, but a foundational issue is: does MA offer a {\it substantial} improvement over MS? Recently, a seminal work: \citet{peng2021}, has answered this question under nested models with linear orthonormal series expansion. In the current paper, we further reply this question under linear nested regression models. Especially,  a more general nested framework, heteroscedastic and autocorrelated random errors, and sparse coefficients are allowed in the current paper, which is more common in practice.  In addition, we further compare MAs with different weight sets. Simulation studies support the theoretical findings in a variety of settings.
\abst.arg

\vskip 3mm\noindent
{\it Some key words:} 
Asymptotic efficiency,  model averaging, model selection,
risk comparison, weight set.\key.arg

\setcounter{page}{1}

\section{Introduction}\label{sec:intro}
\baselineskip=23pt

In the past two decades, model selection (MS) has received growing attention in statistics and econometrics. When a list of candidate models is considered, MS attempts to select a single best model. A large number of MS criteria have been proposed in the literature, including Akaike information criterion \citep[AIC;][]{akaike1973}, Mallows' $C_p$ \citep{mallows1973}, Bayesian information criterion \citep[BIC;][]{schwarz1978},  cross-validation \citep[CV;][]{allen1974,stone1974}, and so on. Model averaging (MA) is an alternative to MS by taking a weighted average of the estimators/predictions from candidate models, and thus is a smoothed extension of MS, and can potentially reduce risk relative to MS 
 \citep{yuan2005,magnus2010comparison}.
 
Asymptotic efficiency (or asymptotic optimality) is a key theoretical goal pursued in both MS and MA literature. It states that the risk (or loss) of MS (or MA) is equivalent to that of the infeasible oracle candidate model (or averaging estimator/prediction).
 %Methods that asymptotically achieve the optimal risk (or loss) are said to be asymptotically efficient or optimal. %\citet{peng2021} provided a brief review about the asymptotic efficiency of MS and MA.
%  
 For MS methods, AIC is asymptotic efficient but BIC is not in the nonparametric framework \citep{shibata1983,shao1997}; \citet{li1987} established the asymptotic efficiency of $C_p$ and leave-one-out CV (LOO-CV) for the homoscedastic nonparametric regression; \citet{andrews1991} extended the results of \citet{li1987} to the case of heteroscedastic errors. 
 See \citet{ding2018} for a recent review on the properties of MS methods. 
 
 For asymptotic optimality of MA, \citet{hansen2007mma} established the asymptotic optimality for Mallows Model Averaging (MMA) when the candidate models are nested and the weights are restricted to a discrete set. \citet{wan2010} and \citet{zhang2021} extended the result of \citet{hansen2007mma} to a non-nested model setting with continuous weights. \citet{hansen2012} established the asymptotic optimality of Jackknife Model Averaging (JMA) for heteroscedastic errors with the weight contained in a discrete set. \citet{zhang2013JMA} broadened \citet{hansen2012}'s scope of analysis to dependent data and a continuous weight set. 
 \citet{liu2013ma} proposed a heteroscedasticity-robust
  model averaging with asymptotic optimality. 
 \citet{ando2014,ando2017} removed the conventional restriction that the sum of weights equals one in MA and established the asymptotic optimality of JMA for high-dimensional linear models and generalized linear models. \citet{zhao2016} broadened \citet{ando2014}'s scope of analysis to dependent data.
 \citet{zhang.wang:2019}, \citet{fang.lan.ea:2019}, and \citet{feng.liu.ea:2021}
 established  
  asymptotic optimality for MA in nonparametric, missing data, and nonlinear models, respectively.  

However, most existing literature mainly focuses on optimal properties (e.g., asymptotic efficiency) of MS or MA in their own terms. Although many successful empirical advancements in MA have been demonstrated \citep[see, e.g.,][]{moral2015model,magnus2016,steel2020model}, theoretical investigations comparing MS and MA are still lacking. Recently, \citet{peng2021} made a seminal contribution to this question. They studied the foundational matter that compares the oracle/optimal MS with MA procedures in a nested model setting with series expansion and have
some remarkable findings. However, a limit of \citet{peng2021} is that their study was built with three restrictions: orthonormal design, homoscedastic and  independent random errors, and non-sparse coefficients, which could make the study is not applicable in applications. The goal of the current paper is to broaden the scope of analysis of \citet{peng2021} to a general model setting and to answer more important questions. Specifically, our main contributions are as follows. 
\begin{description}
	\item[(i)] 
	Without the three restrictions aforementioned,  
	%in	\citet{peng2021},
	we answer the questions:
	 does 
	MA offer a {\it significant} improvement over MS? If it can happen, when?  
	 %of \citet{peng2021} when these three restrictions are removed.
	 % 
	  Moreover, we partition the predictor variables into groups and the group size is allowed to be larger than 1, which leads to a more general nested model framework than
	  that in \citet{peng2021}. % for MS and MA. 
	  %As a result, 
	  We define a sequence of indices for grouped variables and find that its decaying order determines when MA is substantially better than MS. As a result, the analysis of \citet{peng2021} becomes a special case of ours. In finite simulation studies, we compare the performance of MMA or JMA with several MS methods, including AIC, BIC, and LOO-CV. The simulation results support our theoretical findings.
	\item[(ii)] In MA literature, the MA weights could be chosen from three different weight sets, e.g., the unit simplex \citep{wan2010,zhang2013JMA}, the unit hypercube \citep{ando2014,ando2017}, and the discrete weight set \citep{hansen2007mma,hansen2012}. We further broaden the scope of our work to  compare MAs with the weights belonged to these three weight sets.
\end{description}

The rest of the paper is organized as follows. %Section \ref{sec:asyeff} provides a review about the asymptotic efficiency of MS and MA. 
Section \ref{sec:problem} provides model setting and four important questions to be answered. 
% sets up the problem in the more general model setting than \citet{peng2021}. 
Section \ref{sec:main} presents the main results about the comparison of MS and MA. Section \ref{sec:relax} considers the comparison of MAs with three different weight sets. Section \ref{sec:twoexm} provides two examples to verify the theoretical results. Section \ref{sec:simu} presents the results of finite sample simulations. Section \ref{sec:concl} concludes. The proofs of theorems, corollaries, and a lemma are relegated to the Appendix. Supplementary materials contain some additional theoretical results and simulation studies.

\section{Model Setting and Questions}\label{sec:problem}

Consider the model
\begin{equation}\label{eq:model}
	y_i=\mu_i+\varepsilon_i=\sum_{j=1}^{p_n} x_{ij}\beta_j+\varepsilon_i,\quad i=1,\ldots,n,
\end{equation}
where $\varepsilon_1,\ldots,\varepsilon_n$ are random errors, $x_{i1},\ldots,x_{ip_n}$ are nonstochastic predictor variables, and $p_n$ $(p_n<n)$ is the number of the predictors. In matrix notation, \eqref{eq:model} can be written as $\bm{y}=\bm{\mu}+\bm{\varepsilon}$, where $\bm{y}=(y_1,\ldots,y_n)^{\top}$, $\bm{\mu}=(\mu_1,\ldots,\mu_n)^{\top}$, and $\bm{\varepsilon}=(\varepsilon_1,\ldots,\varepsilon_n)^{\top}$. We assume that $E(\bm{\varepsilon})=\bm{0}$ and $\text{Cov}(\bm{\varepsilon})=\bm{\Omega}$, where $\bm{\Omega}$ is a positive definite matrix. In the current paper, we use bold forms to denote vectors or matrices. 

Following \citet{hansen2007mma,hansen2014}, \citet{feng2020nested}, and \citet{zhang2020PMA}, we consider the nested models, where the $m$th model using the first $\nu_m$ predictor variables, such that $0=\nu_0<\nu_1<\nu_2<\cdots<\nu_{q_n-1}<\nu_{q_n}= p_n$, and $q_n$ is a positive integer. 
This nested framework essentially  
 requires that the predictor variables are partitioned into $q_n$ groups and the grouped predictors are ordered, where the $m$th group of predictors is $\{x_{i,\nu_{m-1}+1},\ldots,x_{i,\nu_m}\}$ and its size is $\nu_m-\nu_{m-1}$ for $m=1,\ldots,q_n$. 
 
 Because some model screening procedures may be implemented prior to MS or MA, without loss of generality, we use the first $M_n$ ($2\leq M_n\leq q_n$) nested candidate models for MS and MA. 
Let $ \mathbf{X}_m$ be the $n\times \nu_m$ design matrix of the $m$th candidate model. We assume $ \mathbf{X}_m$ is full of rank for any  $m\in \{1,\ldots,M_n\}$. Then, under the $m$th model, the estimator of $\bm{\mu}$ is
\begin{equation*}
	\hat{\bm{\mu}}_{m}=  \mathbf{X}_m(  \mathbf{X}_m^{\top}  \mathbf{X}_m)^{-1}  \mathbf{X}_m^{\top}\bm{y}\equiv  \mathbf{P}_m\bm{y},
\end{equation*}
where $ \mathbf{P}_m=  \mathbf{X}_m(  \mathbf{X}_m^{\top}  \mathbf{X}_m)^{-1}  \mathbf{X}_m^{\top}$ is the hat matrix. A MS method selects an index (say $m^\star$) from the index set $\mathcal{H}_n=\{1,\ldots,M_n\}$ and estimate $\bm{\mu}$ by $\hat{\bm{\mu}}_{m^\star}$ using the selected model. 
%
%Next, we consider MA estimator of $\bm{\mu}$. 
Let $\bm{w}=(w_1,\ldots,w_{M_n})^{\top}$ be a weight vector belonging to the unit simplex of $\mathbb{R}^{M_n}$:
$$
\mathcal{W}_n=\left\{\bm{w}\in [0,1]^{M_n}: \sum_{m=1}^{M_n} w_m=1\right\}.
$$
Then, the MA estimator of $\bm{\mu}$ with weight vector $\bm{w}$ is
\begin{equation*}
	\hat{\bm{\mu}}(\bm{w})=\sum_{m=1}^{M_n} w_m \hat{\bm{\mu}}_m=\sum_{m=1}^{M_n} w_m  \mathbf{P}_m\bm{y}\equiv  \mathbf{P}(\bm{w})\bm{y},
\end{equation*}
where $ \mathbf{P}(\bm{w})=\sum_{m=1}^{M_n} w_m  \mathbf{P}_m$. The measurement of estimation accuracy is the squared prediction risk, which is defined as $R_n(m)=E\{L_n(m)\}$ and $R_n(\bm{w})=E\{L_n(\bm{w})\}$  for MS and MA, respectively, where
\begin{equation*}
	L_n(m)=\|\hat{\bm{\mu}}_{m}-\bm{\mu}\|^2\quad\text{and}\quad L_n(\bm{w})=\|\hat{\bm{\mu}}(\bm{w})-\bm{\mu}\|^2
\end{equation*}
are the squared prediction loss for MS and MA, respectively, and $\|\cdot\|^2$ is the squared Euclidean norm.

Let $m_n^*$ be the oracle optimal model index that minimizes $R_n(m)$ in $\mathcal{H}_n$ and let $\bm{w}_n^*$ be the oracle optimal weights that minimizes $R_n(\bm{w})$ in $\mathcal{W}_n$. We assume that $m_n^*$ and $\bm{w}_n^*$ are unique. Note that we can not apply $m_n^*$ and $\bm{w}_n^*$ in practice, since they are unknown. But using the asymptotically optimal MS and MA procedures mentioned in Section \ref{sec:intro}, we can select
a model index $\hat{m}_n$ and weights $\hat{\bm{w}}_n$, respectively, in the sense that
\begin{equation}\label{eq:asyopt}
	\frac{R_n(\hat{m}_n)}{R_n(m_n^*)}\xrightarrow{p} 1\quad \text{and}\quad \frac{R_n(\hat{\bm{w}}_n)}{R_n(\bm{w}_n^*)}\xrightarrow{p} 1,
\end{equation}
where $\xrightarrow{p}$ denotes convergence in probability. 
All limiting processes are studied with respect to $n\to \infty$. 
Note that in \eqref{eq:asyopt}, $\hat{m}_n$ and $\hat{\bm{w}}_n$ are directly plugged in the expressions $R_n(m)$ and $R_n(\bm{w})$. \citet{zhang2020PMA} provided a new type of asymptotical optimality as follows:
\begin{equation}\label{eq:asyoptnew}
	\frac{E\{L_n(\hat{m}_n)\}}{R_n(m_n^*)}\xrightarrow{} 1\quad \text{and}\quad \frac{E\{L_n(\hat{\bm{w}}_n)\}}{R_n(\bm{w}_n^*)}\xrightarrow{} 1,
\end{equation}
where the randomness of $\hat{m}_n$ and $\hat{\bm{w}}_n$ are taken into account.

Since MS is a special case of MA with weights concentrating on a single model, it is obvious that $R_n(m_n^*)\geq R_n(\bm{w}_n^*)$. The first task of this paper is essentially on improvability of the oracle regression model $m_n^*$ by the oracle MA. 
Let $\Delta_n=R_n(m_n^*)-R_n(\bm{w}_n^*)$ denote the potential risk reduction of the oracle MA compared to the oracle MS. We first the following two key questions: 
\begin{itemize}
	\item[Q1.] Can $R_n(\bm{w}_n^*)$ bring in 
	a smaller order
	than $R_n(m_n^*)$, i.e. can 
	$R_n(\bm{w}_n^*)/R_n(m_n^*)=o(1)$ happen?
	\item[Q2.] Is $\Delta_n$ a substantial reduction relative to $R_n(m_n^*)$ or actually negligible? If both can happen, when is the MA substantially better than MS?
\end{itemize}

\begin{rem}
Write $ \bm{x}_j=(x_{1j},\ldots,x_{nj})^{\top}$, $j=1,\ldots,p_n$. The vectors $ \bm{x}_{1},\ldots, \bm{x}_{p_n}$ are said to be orthonormal if
\begin{equation}\label{eq:orthcond}
	\frac{1}{n}\sum_{i=1}^n x_{ij}^2=1\quad \text{and}\quad \sum_{i=1}^n x_{ij}x_{ik}=0,\quad 1\leq j\neq k\leq p_n.
\end{equation}
Note that \eqref{eq:orthcond} is satisfied for a nonparametric regression with orthonormal series expansion. \citet{peng2021} has answered Questions Q1 and Q2 under the assumption that the predictors are orthonormal (i.e., \eqref{eq:orthcond} holds), and the error $\varepsilon_i$'s are homoscedastic and independent, which can be 
restrictive for some applications. In the current paper, we answer the same questions in a more general model setting without the above assumptions. Moreover, \citet{peng2021} used the typical nested framework $\nu_m=m$, which is also relaxed in the current paper.
\end{rem}

In addition to the weight set $\mathcal{W}_n$, there are other two popular weight sets in MA literature: the unit hypercube
$$\mathcal{Q}_n=\left\{\bm{w}\in [0,1]^{M_n}: 0\leq w_m\leq 1\right\}$$
and the discrete weight set
\begin{equation*}
	\mathcal{W}_n(N)=\left\{\bm{w}:w_m\in \left\{0,\frac{1}{N},\frac{2}{N},\ldots,1\right\},\sum_{m=1}^M w_m=1\right\}
\end{equation*}
for a fixed positive integer $N$. The weight set $\mathcal{Q}_n$ removes the restriction of  weights from adding up to 1 in $\mathcal{W}_n$. Among MA literature, \citet{ando2014} 
used this weight set to study the optimal MA for the first time in a high-dimensional linear regression setting.  Recently, this weight set was further carried to other regression models, e.g., high-dimensional generalized linear model \citep{ando2017}, high-dimensional quantile regression \citep{wang.QR}, and high-dimensional survival analysis \citep{yan2021,he2020}. In MA literature, the discrete set $\mathcal{W}_n(N)$ is often only considered to establish some asymptotic theories of MA for technical convenience; see, e.g.,  \citet{hansen2007mma}, \citet{hansen2012}, and \citet{fang2020}. In practice, different weight sets can produce different results, hence the the comparisons of MA with the weight sets are very important. In literature, they are compared by numerical examples; see, for example, \citet{ando2014} and \citet{wang.QR}.   
The next task of the current paper is on the theoretical comparison of MA with the weights belonging to different weight sets.   

Let $\tilde{\bm{w}}_n^*$ and $\bm{w}_{n,N}^*$ be the oracle optimal weights that minimizes $R_n(\bm{w})$ in $\mathcal{Q}_n$ and $\mathcal{W}_n(N)$, respectively. We assume that $\tilde{\bm{w}}_n^*$ and $\bm{w}_{n,N}^*$ are unique. Obviously, $R_n(\tilde{\bm{w}}_n^*)\leq R_n(\bm{w}_n^*)\leq R_n(\bm{w}_{n,N}^*)$ since $\mathcal{W}_n(N)\subset \mathcal{W}_n\subset \mathcal{Q}_n$. It implies that the weight relaxation could possibly bring in a smaller optimal risk of MA and the restriction of the weight set $\mathcal{W}_n$ to $\mathcal{W}_n(N)$ could lead to a larger optimal risk of MA. However, it is unclear that whether 
the risk reduction of optimal  MA 
by relaxing the 
weight set $\mathcal{W}_n$ to $\mathcal{Q}_n$ 
 is substantial and whether the risk increment of optimal  MA by restricting the weight set $\mathcal{W}_n$ to $\mathcal{W}_n(N)$ is substantial. Since $\mathcal{W}_n$ is widely  used, we use $R_n(\bm{w}_n^*)$ as a benchmark for the comparisons. Therefore, we consider the other two key issues as follows:

\begin{itemize}
	\item[Q3.] Is $R_n(\bm{w}_n^*)-R_n(\tilde{\bm{w}}_n^*)$ a substantial reduction relative to $R_n(\bm{w}_n^*)$ or actually negligible? If both can happen, when is $\tilde{\bm{w}}_n^*$ substantially better than $\bm{w}_{n}^*$?
	\item[Q4.] Is the risk increment  $R_n(\bm{w}_{n,N}^*)-R_n(\bm{w}_n^*)$ substantial relative to $R_n(\bm{w}_n^*)$ or actually negligible? If both can happen, when is $\bm{w}_n^*$ substantially better than $\bm{w}_{n,N}^*$?
\end{itemize}

The answers to Questions Q1 and Q2 broaden the scope of \citet{peng2021}'s work on advantage of MA over MS. The answers to Questions Q3 and
Q4 provide a previously unavailable insight on relative strengths of MA with these three weight sets.

Throughout this paper, we will use the following symbols. For two nonstochastic positive sequences $a_n$ and $b_n$, $a_n\succeq b_n$ means $b_n=O(a_n)$, and $a_n\asymp b_n$ means both $a_n\succeq b_n$ and $b_n\succeq a_n$. Also, $a_n\sim b_n$ means that $a_n/b_n\to 1$.
% as $n\to \infty$. 
For two stochastic sequences $a_n$ and $b_n$, $a_n\asymp_p b_n$ means that $a_n=O(b_n)\{1+o_p(1)\}$ and $b_n=O(a_n)\{1+o_p(1)\}$; $a_n\sim_p b_n$ means that $a_n/b_n\to_p 1$. % as $n\to \infty$. 
Let $\lfloor a\rfloor$ be the greatest integer less than or equal to $a$. Let $\lambda_{\min}(\bm{\Omega})$ and $\lambda_{\max}(\bm{\Omega})$ be the minimum and maximum eigenvalues of $\bm{\Omega}$, respectively.

\section{Comparisons of MS and MA Procedures}\label{sec:main}

In this section, the goal is to answer Questions Q1 and Q2. 
We first introduce some important notations and assumptions. Then, 
%following \citet{peng2021}, 
we theoretically investigate the comparison of the oracle optimal model $m_n^*$ and the oracle MA in Subsection \ref{subsec:comp}. Last, we extend the obtained results to compare two specific asymptotically optimal MS and MA procedures in Subsection \ref{subsec:comp2}. 

\subsection{Grouped Variable Importance}

In this subsection, we introduce a notation to measure the importance of each grouped predictors, whose decaying order determines when MA is substantially better than MS that will be shown in Subsection \ref{subsec:comp}. Let $\mathbf{X}_m^c$ be the $n\times (p_n-\nu_m)$ design matrix that consists of the predictors excluded from the $m$th model. Let $\bm{\beta}_m=(\beta_1,\ldots,\beta_{\nu_m})^{\top}$ and $\bm{\beta}_m^c=(\beta_{\nu_m+1},\ldots,\beta_{p_n})^{\top}$. Then, $\bm{\mu}=\mathbf{X}_m \bm{\beta}_m+\mathbf{X}_m^c \bm{\beta}_m^c$. For convenience, we further assume $ \mathbf{X}_m$ is full of rank for any $m\in \{M_n+1,\ldots,q_n\}$. 
For the $m$th model ($m=1,\ldots,q_n$), define
\begin{equation}\label{eq:index2}
	\theta_{n,m}=\frac{\bm{\beta}_{m-1}^{c\top}\mathbf{X}_{m-1}^{c\top}(\mathbf{I}_n-\mathbf{P}_{m-1})\mathbf{X}_{m-1}^c\bm{\beta}_{m-1}^c-\bm{\beta}_{m}^{c\top}\mathbf{X}_{m}^{c\top}(\mathbf{I}_n-\mathbf{P}_{m})\mathbf{X}_{m}^c\bm{\beta}_{m}^c}{n\mathrm{tr}\{( \mathbf{P}_m- \mathbf{P}_{m-1})\bm{\Omega}\}},
\end{equation}
where $ \mathbf{P}_0=\bm{0}$. Since the two terms in the numerator of \eqref{eq:index2} measure the importance of the remaining predictors after excluding from the $(m-1)$th model and the $m$th model, respectively, $\theta_{n,m}$ can be regarded as the importance of the $m$th group of variable in some sense. Therefore, we term 
$\theta_{n,m}$ as the grouped variable importance (GVI). 
%Note that $n \theta_{n,m}$ is also referred to as the signal over noise ratio by \citet{ando2017}. 

Next, we impose additional assumptions for the model \eqref{eq:model} to obtain simple forms of GVI as follows.\\
 \textbf{Case 1 (Orthonormal design)} If the orthonormal design assumption \eqref{eq:orthcond} is satisfied, it is easy to show that $\bm{\beta}_{m}^{c\top}\mathbf{X}_{m}^{c\top}(\mathbf{I}_n-\mathbf{P}_{m})\mathbf{X}_{m}^c\bm{\beta}_{m}^c=n\|\bm{\beta}_{m}^c\|^2$. Then, $\theta_{n,m}$ reduces to
\begin{equation*}
	\theta_{n,m}=\frac{\sum_{j=\nu_{m-1}+1}^{\nu_m} \beta_j^2}{\mathrm{tr}\{( \mathbf{P}_m- \mathbf{P}_{m-1})\bm{\Omega}\}}.
\end{equation*}
 \textbf{Case 2 (Homoscedastic and uncorrelated errors)} If the error terms are homoscedastic and uncorrelated with variance $\sigma^2>0$, it is easy to see that $\mathrm{tr}\{( \mathbf{P}_m- \mathbf{P}_{m-1})\bm{\Omega}\}=\sigma^2(\nu_m-\nu_{m-1})$. Then, $\theta_{n,m}$ reduces to
\begin{equation*}
		\theta_{n,m}=\frac{\bm{\beta}_{m-1}^{c\top}\mathbf{X}_{m-1}^{c\top}(\mathbf{I}_n-\mathbf{P}_{m-1})\mathbf{X}_{m-1}^c\bm{\beta}_{m-1}^c-\bm{\beta}_{m}^{c\top}\mathbf{X}_{m}^{c\top}(\mathbf{I}_n-\mathbf{P}_{m})\mathbf{X}_{m}^c\bm{\beta}_{m}^c}{n\sigma^2(\nu_m-\nu_{m-1})}.
\end{equation*}
 \textbf{Case 3 (Orthonormal design and homoscedastic and uncorrelated errors)} If the orthonormal design assumption \eqref{eq:orthcond} is satisfied and the error terms are homoscedastic and uncorrelated with variance $\sigma^2>0$, then $\theta_{n,m}$ has a simple form
\begin{equation*}
	\theta_{n,m}=\frac{\sum_{j=\nu_{m-1}+1}^{\nu_m} \beta_j^2}{\sigma^2(\nu_m-\nu_{m-1})}.
\end{equation*}
\textbf{Case 4 (Model setting of \citet{peng2021})} In the model setting of \citet{peng2021}, i.e., under the  assumptions of the orthonormal design \eqref{eq:orthcond}, the homoscedastic and uncorrelated errors with variance $\sigma^2>0$, and the typical nested framework $\nu_m=m$, $\theta_{n,m}$ reduces to
\begin{equation*}
	\theta_{n,m}=\beta_m^2/\sigma^2.
\end{equation*}

From these cases (especially, Cases 1, 3, and 4), the numerator of $\theta_{n,m}$ is determined 
by the coefficients of the variables in the $m$th group, which further means that $\theta_{n,m}$ is a measure of the grouped variable importance. In the following remark, we provide another explanation for $\theta_{n,m}$. 
\begin{rem}[\textbf{Another explanation for $\theta_{n,m}$}]
	By some simple calculations, we can get a more simple form for $\theta_{n,m}$ as follows
	\begin{equation}\label{eq:index}
	\theta_{n,m}=\frac{n^{-1}\bm{\mu}^{\top}( \mathbf{P}_m- \mathbf{P}_{m-1})\bm{\mu}}{\mathrm{tr}\{( \mathbf{P}_m- \mathbf{P}_{m-1})\bm{\Omega}\}}.
%	\label{muP}
	\end{equation}
	In the first formula of Appendix \ref{sec:thm1}, 
	we show that the risk of
		 $\hat{\bm{\mu}}_m$ is
	\begin{eqnarray*}
		R_n(m)\hspace{-6mm}&&=E\|\hat{\bm{\mu}}_m-\bm{\mu}\|^2\nonumber\\
		&&=\mathrm{tr}\left\{ (E\hat{\bm{\mu}}_m-\bm{\mu})(E\hat{\bm{\mu}}_m-\bm{\mu})^\top\right\} + \mathrm{tr}\left\{\mathrm{var}(\hat{\bm{\mu}}_m)\right\}\nonumber\\
		&&=\bm{\mu}^{\top}(\mathbf{I}_n-\mathbf{P}_m)\bm{\mu}+\mathrm{tr}(\mathbf{P}_m\bm{\Omega}),
	\end{eqnarray*} 
	where $\bm{\mu}^{\top}(\mathbf{I}_n-\mathbf{P}_m)\bm{\mu}$ is the trace of the squared bias term of $\hat{\bm{\mu}}_m$ and $\mathrm{tr}(\mathbf{P}_m\bm{\Omega})$ is the trace of the variance term of $\hat{\bm{\mu}}_m$, which are non-increasing and increasing in $m$, respectively, because of the nested framework of candidate models.  Therefore, by (\ref{eq:index}), the numerator and denominator of $\theta_{n,m}$ are  the decrement of the squared bias scared by $n$ and the increment of the variance of risks, respectively, when adding the $m$th group of predictors to the $(m-1)$th model. When $\bm{\Omega}=\sigma^2 \mathbf{I}_n$ and the size of groups is fixed (say $\nu^\ast$), the increment of the variance of risks is fixed to be $\nu^\ast\sigma^2$.
\end{rem}

We now impose the following two assumptions for the model \eqref{eq:model}, which are commonly used in the MA literature.
\begin{description}
	\item[Assumption 1.] $\|\bm{\mu}\|^2/n=O(1)$.
	\item[Assumption 2.] There are constants $0<c_1\leq c_2<\infty$ such that $c_1< \lambda_{\min}(\bm{\Omega})\leq \lambda_{\max}(\bm{\Omega})< c_2$.
\end{description}
%{\blue Wenchao, add discussions on the above assumptions.}

{%\colorchange 
	Assumption 1 requires the average of $\mu_i^2$ is bounded. Assumption 2 excludes the degeneracy and divergence of the error terms. Similar assumptions can be found in \citet{wan2010}, \citet{zhang2013JMA}, and \citet{liu2016}.}
Since $ \mathbf{P}_m  \mathbf{P}_l= \mathbf{P}_{\min(m,l)}$ in the nested model setting, it can be easily verified that $ \mathbf{P}_m- \mathbf{P}_{m-1}$ is a symmetric idempotent  matrix, which, together with Assumption 1 yields that $0\leq n^{-1}\bm{\mu}^{\top}( \mathbf{P}_m- \mathbf{P}_{m-1})\bm{\mu}\leq \|\bm{\mu}\|^2/n<\infty$. Using Assumption 2,  we have
$$
0<c_1(\nu_m-\nu_{m-1})< \mathrm{tr}\{( \mathbf{P}_m- \mathbf{P}_{m-1})\bm{\Omega}\}< c_2(\nu_m-\nu_{m-1}).
$$
By (\ref{eq:index}) and the truth that
$ \mathbf{P}_m- \mathbf{P}_{m-1}$ is a symmetric idempotent  matrix, we know that 
\begin{equation*}%\label{eq:pos}
	0\leq \theta_{n,m}< \infty \text{~for any $n$ and $m=1,\ldots,q_n$}.
\end{equation*}

We further impose an  assumption on $\theta_{n,m}$ as follows.
\begin{description}
	\item[Assumption 3.] For each sufficiently large $n$, $\{\theta_{n,m}: m=1,\ldots,q_n\}$ is a non-increasing sequence.
\end{description}

Assumption 3 is an extension of Assumption 1 of \citet{peng2021}. %The positive condition of $\theta_{n,m}$ excludes the case where $\theta_{n,m}=0$ in \eqref{eq:pos}. 
The non-increasing ordering of $\{\theta_{n,m}\}_{m=1}^{q_n}$ can give us a convenience to characterize the unknown optimal model index $m_n^*$ and weights $\bm{w}_n^*$, $\tilde{\bm{w}}_n^*$, and $\bm{w}_{n,N}^*$, and essentially requires that the predictors are groupwise ordered from the most important to the
least important, i.e, the ordering of grouped variables is ``correct". 
However, different from \citet{peng2021} where the sequence $\{\theta_{n,m}: m=1,\ldots,q_n\}$ required to be positive, the current paper allows some components in the sequence to be equal to zeros, i.e., we allow some totally unimportant variables. For the aforementioned Cases 3-4, it means we allow some kind of sparsity of coefficients, which is an important  property especially for high-dimensional problems.
Under Assumption 3, let $d_n=\max\{m\in\{1,\ldots,q_n\}:\theta_{n,m}>0\}$ be the number of important groups of predictors. If $d_n<q_n$, the $m$th group of predictors is not important for $m=d_n+1,\ldots,q_n$. Next, we make the following assumptions.

\begin{figure}[tpb]
	\centering
	\includegraphics[scale=0.7]{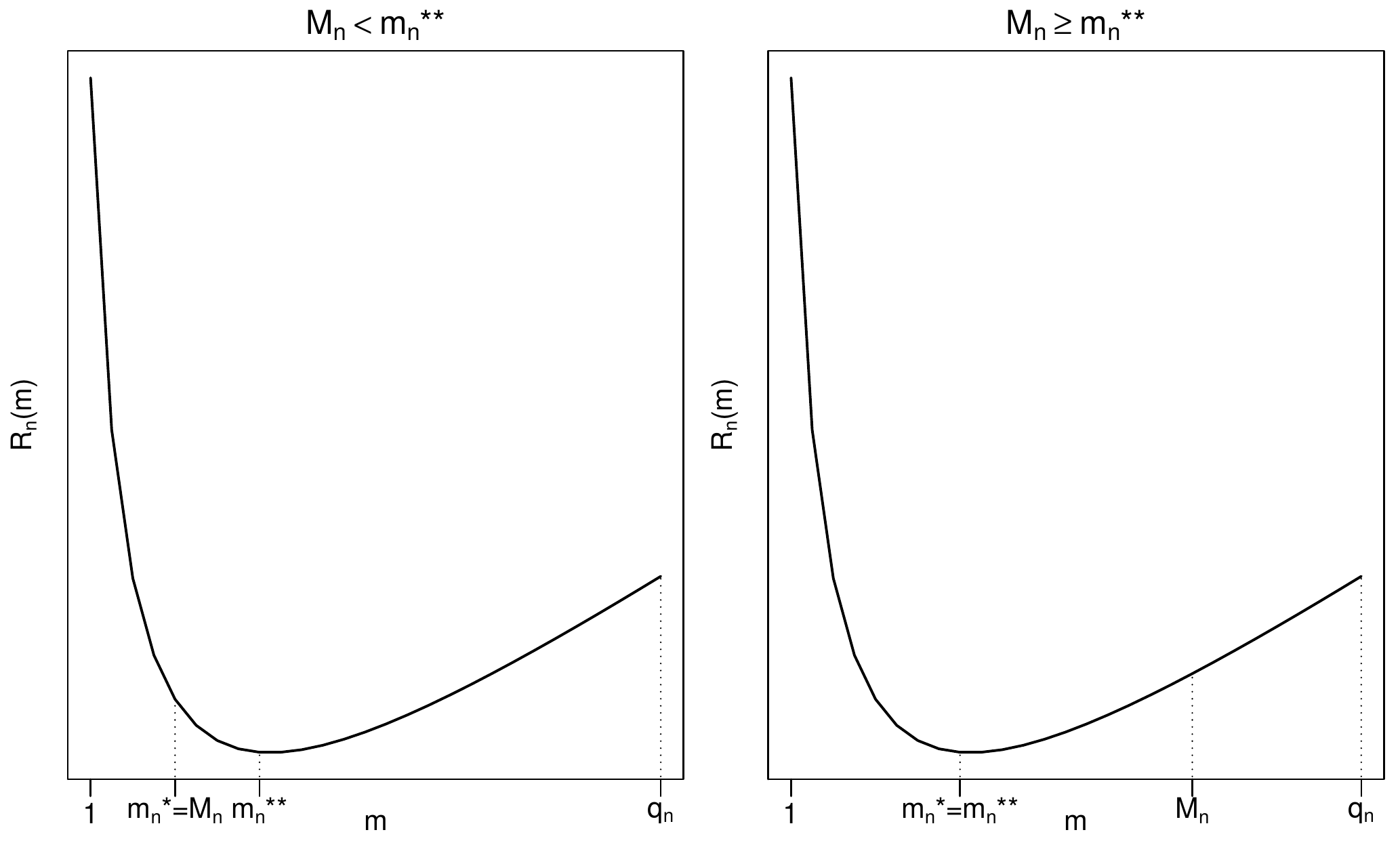}
	\caption{Two typical situations of the squared prediction risk $R_n(m)$ and the relationship between $M_n$ and $m_n^{**}$  
		under Assumption 6. In left panel: $M_n<m_n^{**}$. In right penal: $M_n\geq m_n^{**}$.}
	\label{fig:riskRm}
\end{figure}

\begin{description}
	\item[Assumption 4.] There exists a constant $V\geq 1$ which does not depend on $m$ and $n$, such that $\max_{1\leq m\leq d_n}(\nu_{m}-\nu_{m-1})\leq V$ holds uniformly in large enough $n$.
	\item[Assumption 5.] For any fixed $m\in \{1,\ldots,d_n\}$, there exist constants $\bar{\theta}_m>0$ and $K_m>0$ such that $\theta_{n,m}\geq \bar{\theta}_m$ for any $n\geq K_m$.
	\item[Assumption 6.] For each sufficiently large $n$, $R_n(m)$ is first decreasing and then increasing when $m$ varies from 1 to $d_n$.
\end{description}

Assumption 4 means that when the predictors are partitioned into groups described in Section \ref{sec:problem}, the sizes of all important groups do not grow to infinity as $n$ increases. \citet{hansen2014} and \citet{zhang2016ols} also made assumptions on group sizes for the comparison of the risks of estimators by the full model and the MMA.
%, which is different from ours. 
Assumption 5 basically eliminates the case that $\theta_{n,m}$ goes to zero as $n$ increases for any fixed $m\in \{1,\ldots,d_n\}$. Thus, it excludes the local-to-zero asymptotic framework with fixed dimensions considered by \citet{hjort2003} and \citet{liu2015joe}. Note that in the model setting of \citet{peng2021}, Assumptions 4--5 are obviously satisfied. %Overall, Assumptions 4--5 are quite mild in our context. 
Assumption 6 essentially prevents $d_n$ from being too small. For example, when $d_n\equiv d_0$ is fixed as $n\to \infty$ and Assumptions 2 and 5 hold, by \eqref{eq:diffrisk} in the Appendix, 
\begin{eqnarray*}
	R_n(m)-R_n(m-1)\hspace{-6mm}&&=n\mathrm{tr}\{( \mathbf{P}_m- \mathbf{P}_{m-1})\bm{\Omega}\}\biggl(\frac{1}{n}-\theta_{n,m}\biggr) \\
	&&\leq n\mathrm{tr}\{( \mathbf{P}_m- \mathbf{P}_{m-1})\bm{\Omega}\}\biggl(\frac{1}{n}-\bar{\theta}_{m}\biggr)<0 
\end{eqnarray*}
for $m=2,\ldots,d_0$ and each sufficiently large $n$, which implies that $R_n(m)$ is decreasing on $\{1,\ldots,d_0\}$ for each sufficiently large $n$ and thus Assumption 6 is not satisfied. Under Assumptions 3 and 5, Assumption 6 is satisfied when $\theta_{n,d_n}<{1}/{n}$ for each sufficiently large $n$, which is derived as follows. For a sufficiently large $n$, from Assumptions 3 and 5 and $\theta_{n,d_n}<{1}/{n}$, there exists a $m_n'\in \{2,\ldots,d_n-1\}$ such that $\theta_{n,m_n'}>1/n\geq \theta_{n,m_n'+1}$. Therefore, when $m=2,\ldots,m_n'$, 
\begin{equation*}
	R_n(m)-R_n(m-1)\leq n\mathrm{tr}\{( \mathbf{P}_m- \mathbf{P}_{m-1})\bm{\Omega}\}\biggl(\frac{1}{n}-\theta_{n,m_n'}\biggr)<0;
\end{equation*}
when $m=m_n'+1,\ldots,d_n$,
\begin{equation*}
	R_n(m)-R_n(m-1)\geq n\mathrm{tr}\{( \mathbf{P}_m- \mathbf{P}_{m-1})\bm{\Omega}\}\biggl(\frac{1}{n}-\theta_{n,m_n'+1}\biggr)\geq 0;
\end{equation*}
and $R_n(d_n)-R_n(d_n-1)>0$, which together imply that $R_n(m)$ is decreasing in $m\in\{1,\ldots,m_n'\}$ and increasing in $m\in \{m_n',\ldots,d_n\}$. Assumption 6 also implies that $R_n(m)$ is first decreasing and then increasing when $m$
increases from $1$ to $q_n$ for each sufficiently large $n$. 
In the derivation of (A.2) in \citet{peng2021}, they also used this assumption. 

Let $m_n^{**}=\arg\min_{m\in \{1,\ldots,q_n\}} R_n(m)$ be the global optimal model index. We assume that $m_n^{**}$ is unique. Note that $m_n^{**}$ may not equal to $m_n^*$ since the number of candidate models $M_n$ may be too small to include the $m_n^{**}$th model. In fact, under Assumption 6, 
\begin{equation*}
	m_n^*=\min\{M_n,m_n^{**}\}=\begin{cases}
		M_n, & \mathrm{if}~M_n<m_n^{**} \\
		m_n^{**}, & \mathrm{if}~M_n\geq m_n^{**}
	\end{cases};
\end{equation*}
see Figure \ref{fig:riskRm} which shows 
two typical situations for $m_n^\ast$.  \citet{peng2021} compared MS and MA under the assumption that $M_n$ is large enough to include $m_n^{**}$, i.e., $M_n\geq m_n^{**}$. However, there is a lack of considering on the comparison of MS and MA when $M_n<m_n^{**}$. In this paper, we also relax this assumption and investigate the impact of $M_n$ on the comparison of MS and MA.

In Subsections \ref{subsec:comp}--\ref{subsec:comp2}, we shall show that the number of candidate models $M_n$ and the decaying order of $\{\theta_{n,m}\}_{m=1}^{d_n}$ determine when MA is substantially better than MS.

\subsection{A Comparison of Oracle Optimal MS and MA}\label{subsec:comp}

In this subsection, we answer Questions Q1 and Q2, that is we compare the risks of optimal MS and MA estimators. We begin with the following theorem on the order relationship of $R_n(m_n^*)$ and $R_n(\bm{w}_n^*)$, which provides an answer to Question Q1.

\begin{theorem}[\textbf{Answer to Question Q1}]\label{thm:mw}
	Suppose that Assumptions 1-3 and 6 hold. Then, for any sufficiently large $n$,
	\begin{equation}\label{eq:wmlb}
		R_n(\bm{w}_n^*)\geq \frac{1}{2} R_n(m_n^*).
	\end{equation}
	Moreover, the risks using $m_n^*$ and $\bm{w}_n^*$ have the same order, i.e.,
	\begin{equation}\label{eq:eqiv}
		R_n(m_n^*)\asymp R_n(\bm{w}_n^*).
	\end{equation}
\end{theorem}

Note that Theorem \ref{thm:mw} needs not any restriction on the number of candidate models $M_n$, while \citet{peng2021} restricted $M_n\geq m_n^{**}$. Inequality \eqref{eq:wmlb} leads to $\Delta_n\leq \frac{1}{2}R_n(m_n^*)$ for sufficiently large $n$, which implies that the potential risk reduction of MA compared to MS does not exceed half of optimal risk of MS. Equation \eqref{eq:eqiv} indicates that while MA has a smaller optimal risk than MS, MA actually cannot reduce the increasing rate (or improve the decreasing rate) of risk by the optimal MS. 
Thus, even if the oracle model based on MS can be improved by MA, the potential advantage of MA in risk reduction is limited in terms of the increasing/deceasing rate.

We now turn our attention to Question Q2. We first present the following theorem on some elementary properties of the global optimal model index $m_n^{**}$, $R_n(m_n^*)$, and $R_n(\bm{w}_n^*)$, which are important for the subsequent analysis.

\begin{theorem}\label{thm:mn}
	Suppose that Assumptions 1--6 are satisfied. Then, 
	\begin{itemize}
		\item[(i)] $m_n^{**}\to \infty$ as $n\to \infty$;
		\item[(ii)] $R_n(m_n^*)\to \infty$ and $R_n(\bm{w}_n^*)\to \infty$ as $n\to \infty$.
	\end{itemize}
	
\end{theorem}
Theorem \ref{thm:mn}(i) indicates that the index of the global optimal model diverges to infinity as $n\to \infty$ under some mild assumptions. For applications, this result tells us that 
a diverging dimension should be utilized to achieve a promising MS performance.
Theorem \ref{thm:mn}(ii) means that the smallest risks of MS and MA grow to infinity as the sample size $n$ increases. Note that Theorem \ref{thm:mn}(ii) does not require the restriction $M_n\geq m_n^{**}$, but this restriction is used in  \citet{peng2021}.

To investigate the impact of the number of candidate models $M_n$ on the comparison of MS and MA, we consider the following two conditions for $M_n$.

\begin{description}
	\item[Condition M1] (Too Small $M_n$). $\lim_{n\to \infty} M_n/m_n^{**}=0$.
	
	\item[Condition M2] (Large Enough $M_n$). There exists a constant $\underline{c}>0$ such that $M_n/m_n^{**}\geq \underline{c}$ when $n$ is large enough.
\end{description}

By Theorem \ref{thm:mn}(i), under Assumptions 1-6, Condition M1 holds when $M_n$ is fixed or diverges to infinity more slowly than $m_n^{**}$ as $n\to \infty$. Condition M2 holds when $M_n\geq m_n^{**}$ considered by \citet{peng2021} or $M_n<m_n^{**}$ but $M_n$ has the same order with $m_n^{**}$. Next, we  successively  explore the degree of improvement $\Delta_n/R_n(m_n^*)$ by the following theorems under Conditions M1 and M2.

\begin{theorem}[\textbf{Answer to Question Q2 under Condition M1}]\label{thm:delta2}
	Suppose that Assumptions 1-6 hold. Under Condition M1,
	$$\Delta_n=o\{R_n(m_n^*)\}.$$
\end{theorem}

Theorem \ref{thm:delta2} implies that when the number of candidate models $M_n$ is fixed or diverges to infinity more slowly than $m_n^{**}$ as $n\to \infty$, MA has no essential advantage over MS, which again indicates that in applications, an enough large number of candidate models should be utilized. 
%b
Note that Theorem \ref{thm:delta2} does not require any assumption of the decaying order of $\{\theta_{n,m}\}_{m=1}^{d_n}$. 

Furthermore, when Condition M2 holds, we explore the degree of improvement $\Delta_n/R_n(m_n^*)$ by the following theorem under sensible conditions on $\theta_{n,m}$, which provides an answer to Question Q2 under Condition M2. The answer depends on the decaying order of $\{\theta_{n,m}\}_{m=1}^{d_n}$. 
\begin{description}
	\item[Condition A1] (Slowly Decaying $\{\theta_{n,m}\}_{m=1}^{d_n}$). There exist constants $k>1$, $0<\delta\leq \eta<1$ with $k\eta<1$, and $K>0$ such that for every integer sequence $\{l_n\}$ satisfying $\lim_{n\to \infty} l_n=\infty$, $$\delta\leq \theta_{n, \lfloor k l_n\rfloor}/\theta_{n, l_n}\leq \eta$$ 
	for any $n\geq K$.

\item[Condition A2] (Fast Decaying $\{\theta_{n,m}\}_{m=1}^{d_n}$). For every constant $k>1$ and every integer sequence $\{l_n\}$ satisfying $\lim_{n\to \infty} l_n=\infty$, $$\lim_{n\to \infty} \theta_{n, \lfloor kl_n\rfloor}/\theta_{n, l_n}=0.$$
\end{description}
In the model setting of \citet{peng2021}, since $\theta_{n,m}=\beta_m^2/\sigma^2$, Conditions A1--A2 are equivalent to the conditions 1--2 of \citet{peng2021}, respectively.

\begin{theorem}[\textbf{Answer to Question Q2 under Condition M2}]\label{thm:delta}
	Suppose that Assumptions 1-6 and Condition M2 holds. Under Condition A1, we have
	$$\Delta_n\asymp R_n(m_n^*).$$
	Under Condition A2, we have
	$$\Delta_n=o\{R_n(m_n^*)\}.$$
\end{theorem}
From Theorems \ref{thm:mw} and \ref{thm:delta}, under Condition A1,
\begin{equation*}
	\frac{1}{2}\leq \liminf_{n\to \infty} \frac{R_n(\bm{w}_n^*)}{R_n(m_n^*)}\leq \limsup_{n\to \infty} \frac{R_n(\bm{w}_n^*)}{R_n(m_n^*)}\leq 1-c^*
\end{equation*}
for some $c^*\in (0, 1/2]$; and under Condition A2, $R_n(\bm{w}_n^*)\sim R_n(m_n^*)$. Therefore, when the number of candidate models $M_n$ is large enough, there is a phase transition for the advantage of MA over MS. When $\theta_{n,m}$ decays slowly in $m$, the oracle MA can reduce the optimal risk of MS by a substantial fraction; when $\theta_{n,m}$ decays fast in $m$, MA has no real advantage over MS.

To gain a better understanding of Conditions A1--A2, we make the following more simple conditions (i.e., Assumption 7 and Conditions B1--B2), which imply Conditions A1--A2 by Lemma \ref{lem:cond} below.
\begin{description}
	\item[Assumption 7.] There exists a sequence $\theta_m^{*}>0$ for $m=1,\ldots,d_n$ such that
	$$\sup_{1\leq m\leq d_n}\left|\frac{\theta_{n,m}}{\theta_{m}^*}-1\right|\to 0  \mathrm{~as~}n\to \infty.$$
\end{description}
\begin{description}
	\item[Condition B1] (Slowly Decaying $\theta_m^*$). There exist constants $k>1$ and $0<\delta^* \leq \eta^*<1$ with $k\eta^*<1$ such that $\delta^*\leq \theta_{\lfloor km \rfloor}^{*}/\theta_m^{*}\leq \eta^*$ when $m$ is large enough.
	\item[Condition B2] (Fast Decaying $\theta_m^*$). For every constant $k>1$, $\lim_{m\to \infty} \theta_{\lfloor km \rfloor}^{*}/\theta_m^{*}=0$.
\end{description}

\begin{lem}\label{lem:cond}
	Suppose that Assumption 7 holds. Then, Conditions B1--B2 imply Conditions A1--A2, respectively.
\end{lem}

Assumption 7 implies that $\lim_{n\to \infty} \theta_{n,m}=\theta_m^*$ for any fixed $m$. Thus, Assumption 7 can lead to Assumption 5 by taking $\bar{\theta}_m=\theta_m^*/2$. Condition B1 is satisfied for $\theta_m^*\sim m^{-2\alpha}$ or slightly more generally for $\theta_m^*\sim m^{-2\alpha}(\log m)^{\beta}$ with constants $\alpha>1/2$ and $\beta\in \mathbb{R}$. Condition B2 is satisfied for the exponential-decay case, that is, $\theta_m^*\sim \exp(-c m)$ for some $c>0$. {%\colorchange 
	These two types of decaying rates are commonly used in the literature. For example, in the research of infinite-order autoregressive (AR) models, \citet{ing2005}, \citet{ing2007}, and \citet{liao2021ma} considered the exponential-decay and algebraic-decay cases for the AR coefficients, which are described in our context as follows:
\begin{itemize}
	\item[(i)] Exponential-decay case: $C_1m^{-\tau_1}e^{-c m}\leq \theta_m^*\leq C_2m^{\tau_1}e^{-c m}$, where $C_1$, $C_2$, $\tau_1$, and $c$ are some constant with $C_2\geq C_1>0$, $\tau_1\geq 0$, and $c>0$.
	\item[(ii)] Algebraic-decay case: $(C_3-C_4m^{-\tau_2})m^{-\bar{\alpha}}\leq \theta_m^*\leq (C_3+C_4m^{-\tau_2})m^{-\bar{\alpha}}$, where $C_3$, $C_4$, $\tau_2$, and $\bar{\alpha}>1$ are some positive constants.
\end{itemize}
It can be easily verified that the exponential-decay case (i) and the algebraic-decay case (ii) satisfy Condition B2 and Condition B1, respectively.}

%{\blue Wenchao, add more discussions on the above decaying rate.}

\subsection{A Comparison of Two Specific MS and MA Procedures}\label{subsec:comp2}

Up to now, the theoretical results of Theorems \ref{thm:mw} and \ref{thm:delta2}--\ref{thm:delta} mainly focus on the comparison of oracle optimal MS and MA, not directly on the comparison of two specific MS and MA procedures. 
Fortunately, by using 
\eqref{eq:asyopt} and \eqref{eq:asyoptnew}, we can do the latter comparison by connecting the feasible risks (when using a selected model index or weights from some methods) and infeasible risks (when using the oracle model index or weights).
%  with asymptotically efficient (or optimal) MS and MA procedures, respectively. 
% 
%For MS methods, the asymptotic efficiency of AIC has been derived by \cite{shibata1983}; \citet{li1987} and \citet{andrews1991} showed the asymptotic efficiency of LOO-CV for the homoskedastic and heteroskedastic errors, respectively. For MA methods, \citet{hansen2007mma} and \citet{wan2010} established the asymptotic optimality of MMA for homoskedastic errors; \citet{hansen2012} and \citet{zhang2013JMA} established the asymptotic optimalities of JMA for heteroskedastic errors and dependent data, respectively. Note that these asymptotic efficiencies are of type \eqref{eq:asyopt}. \citet{zhang2020PMA} established the types of asymptotic optimality \eqref{eq:asyoptnew} for the parsimonious model averaging.
%
%In nested model setting, 
In literature, the proof of asymptotic efficiency (or optimality) of MS and MA requires the smallest risks of MS and MA (i.e., $R_n(m_n^*)$ and $R_n(\bm{w}_n^*)$ in our notations) to grow to infinity as the sample size increases, respectively. Both have been verified in Theorem \ref{thm:mn}(ii) under Assumptions 1--6.

Let $\hat{m}_n$ and $\hat{\bm{w}}_n$ be the selected model index and chosen weights based on asymptotically optimal MS and MA methods, respectively. Then, we have the following two consequences.
\begin{corollary}\label{cor:main}
	Suppose that Assumptions 1-6 hold, and $\hat{m}_n$ and $\hat{\bm{w}}_n$ are asymptotically optimal in the sense of \eqref{eq:asyopt}, i.e., $R_n(\hat{m}_n)/R_n(m_n^*)\to_p 1$ and $R_n(\hat{\bm{w}}_n)/R_n(\bm{w}_n^*)\to_p 1$. Then, 
	\begin{itemize}
		\item[(i)] the risks using $\hat{m}_n$ and $\hat{\bm{w}}_n$ have the same order, i.e., $R_n(\hat{m}_n)\asymp_p R_n(\hat{\bm{w}}_n)$; 
	\item[(ii)] under Conditions M2 and A1, the MA using $\hat{\bm{w}}_n$ essentially improves over the MS using $\hat{m}_n$, i.e., $R_n(\hat{m}_n)-R_n(\hat{\bm{w}}_n)\asymp_p R_n(\hat{m}_n)$; 
	\item[(iii)] under Condition M1 or Conditions M2 and A2, $\hat{m}_n$ and $\hat{\bm{w}}_n$ are asymptotically equivalent in risk, i.e., $R_n(\hat{m}_n)\sim_p R_n(\hat{\bm{w}}_n)$.
\end{itemize}
\end{corollary}

\begin{corollary}\label{cor:mainnew}
	Suppose that Assumptions 1-6 hold, and $\hat{m}_n$ and $\hat{\bm{w}}_n$ are asymptotically optimal in the sense of  \eqref{eq:asyoptnew}, i.e.,  $E\{L_n(\hat{m}_n)\}/R_n(m_n^*)\to 1$ and $E\{L_n(\hat{\bm{w}}_n)\}/R_n(\bm{w}_n^*)\to 1$. Then, the results of Corollary \ref{cor:main} hold when $R_n(\hat{m}_n)$, $R_n(\hat{\bm{w}}_n)$, $\asymp_p$, and $\sim_p$ are replaced by $E\{L_n(\hat{m}_n)\}$, $E\{L_n(\hat{\bm{w}}_n)\}$, $\asymp$, and $\sim$, respectively.
\end{corollary}

Since the proofs of Corollaries \ref{cor:main} and \ref{cor:mainnew} are similar, we only give the proof of Corollary \ref{cor:main} in Appendix \ref{subsec:appcor}.  
%Corollary \ref{cor:main} focuses on the comparison of $R_n(\hat{m}_n)$ and $R_n(\hat{\bm{w}}_n)$, and Corollary \ref{cor:mainnew} focuses on the comparison of $E\{L_n(\hat{m}_n)\}$ and $E\{L_n(\hat{\bm{w}}_n)\}$.

\section{Comparisons of MAs with Different Weight Sets}\label{sec:relax}

In this section, we shall compare the optimal risks of MAs when the weights belong to three weight sets: $\mathcal{W}_n$, $\mathcal{Q}_n$, and $\mathcal{W}_n(N)$, which provide answers to Questions Q3--Q4.

\subsection{A Comparison of MAs with Weight Sets $\mathcal{W}_n$ and $\mathcal{Q}_n$}

In this subsection, we focus on the comparison of the risks of MA estimators when the weights comes from $\mathcal{W}_n$ and $\mathcal{Q}_n$, respectively. We first present the following theorem, which provides an answer to Question Q3.

\begin{theorem}[\textbf{Answer to Question Q3}]\label{thm:relax}
	Suppose that Assumptions 1-6 hold. Then, 
\begin{equation}\label{eq:relaxasy}
R_n(\bm{w}_n^*)- R_n(\tilde{\bm{w}}_n^*)=o\{R_n(\bm{w}_n^*)\},
\end{equation}
 i.e., $\bm{w}_n^*$ and $\tilde{\bm{w}}_n^*$ are asymptotically equivalent in risk.
\end{theorem}
Equation \eqref{eq:relaxasy} indicates that while the weight relaxation could lead to a smaller optimal risk of MA, asymptotically it does not provide any substantial benefit. Note that Theorem \ref{thm:relax} does not require any assumptions of the number of candidate models $M_n$ and the decaying order of $\{\theta_{n,m}\}_{m=1}^{d_n}$.

Furthermore, we compare two specific asymptotically optimal MS and MA procedures, where MA weights are chosen from the weight set $\mathcal{Q}_n$. Let $\hat{\bm{w}}_n^Q$ be chosen weights based on a specific MA method 
%without the total weight constraint that
 satisfying the asymptotic optimality \eqref{eq:asyopt} or \eqref{eq:asyoptnew}
 when the total weight is not constrained
 (i.e., the total weight constraint $\sum_{m=1}^{M_n}w_m=1$ is not used).  %$\bm{w}\in\mathcal{Q}_n$. 
 For example, the asymptotic optimality \eqref{eq:asyopt} of JMA without the total weight constraint 
%with $\bm{w}\in\mathcal{Q}_n$
 has been established by \citet{ando2014} and \citet{zhao2016} for independent data and dependent data, respectively. Using Theorem \ref{thm:relax}, it is easy to conclude the following corollary on a comparison of MA without the total weight constraint and MS.
\begin{corollary}\label{cor:main2}
	{%\colorchange 
	Suppose that Assumptions 1-6 hold.
	\begin{itemize}
		\item[(i)] Assume that $\hat{m}_n$ and $\hat{\bm{w}}_n^Q$ satisfy $R_n(\hat{m}_n)/R_n(m_n^*)\to_p 1$ and $R_n(\hat{\bm{w}}_n^Q)/R_n(\bm{w}_n^*)\to_p 1$, then the results of Corollary \ref{cor:main} hold when $\hat{\bm{w}}_n$ is replaced by $\hat{\bm{w}}_n^Q$.
		\item[(ii)] Assume that $\hat{m}_n$ and $\hat{\bm{w}}_n^Q$ satisfy $E\{L_n(\hat{m}_n)\}/R_n(m_n^*)\to 1$ and $E\{L_n(\hat{\bm{w}}_n^Q)\}/R_n(\bm{w}_n^*)\to 1$, then the results of Corollary \ref{cor:mainnew} hold when $\hat{\bm{w}}_n$ is replaced by $\hat{\bm{w}}_n^Q$.
	\end{itemize}}
%Theorems \ref{thm:mw} and \ref{thm:delta2}--\ref{thm:delta} and Corollaries \ref{cor:main}--\ref{cor:mainnew} with $\bm{w}_n^*$ and $\hat{\bm{w}}$ replacing by $\tilde{\bm{w}}_n^*$ and $\tilde{\bm{w}}$, respectively, still hold.
\end{corollary}
%{\blue  Wenchao, reorganize the above corollary.}

\subsection{A Comparison of MAs with Weight Sets $\mathcal{W}_n$ and $\mathcal{W}_n(N)$}

In this subsection, we focus on the comparison of the optimal risks of MA estimators with the weights belonging to the weight sets $\mathcal{W}_n$ and $\mathcal{W}_n(N)$, respectively. We present the following theorem on an upper bound of $R_n(\bm{w}_{n,N}^*)- R_n(\bm{w}_n^*)$ and an answer to Question Q4.

\begin{theorem}[\textbf{Answer to Question Q4}]\label{thm:discrete}
	Suppose that Assumptions 1-6 hold. Then, for any sufficiently large $n$,
		\begin{equation}\label{eq:ub}
			R_n(\bm{w}_{n,N}^*)- R_n(\bm{w}_n^*)\leq \frac{1}{2N}R_n(m_n^*).
		\end{equation}
%i.e., the risks using $\bm{w}_{n,N}^*$ and $\bm{w}_n^*$ have the same order. 
Furthermore, under Conditions M2 and A1, we have
$$
R_n(\bm{w}_{n,N}^*)- R_n(\bm{w}_n^*)\asymp R_n(\bm{w}_n^*);
$$
and under Condition M1 or Conditions M2 and A2, we have $$R_n(\bm{w}_{n,N}^*)- R_n(\bm{w}_n^*)=o\{R_n(\bm{w}_n^*)\},$$
i.e., 
$$
R_n(\bm{w}_{n,N}^*)\sim R_n(\bm{w}_n^*). 
$$
\end{theorem}

Observe that $R_n(\bm{w}_{n,1}^*)=R_n(m_n^*)$. Therefore, Theorems \ref{thm:mw} and \ref{thm:delta2}--\ref{thm:delta} are special cases of Theorem \ref{thm:discrete} with $N=1$. The upper bound in \eqref{eq:ub} implies that for a fixed and sufficiently large sample size, $R_n(\bm{w}_{n,N}^*)$ can be made arbitrarily close to $R_n(\bm{w}_n^*)$ as $N$ is large enough, which is expected since $\mathcal{W}_n(N)$ approaches $\mathcal{W}_n$ as close as possible by making $N$ sufficiently large. From Theorem \ref{thm:discrete}, when $M_n$ is large enough and $\theta_{n,m}$ decays slowly in $m$, restricting the weight set $\mathcal{W}_n$ to $\mathcal{W}_n(N)$ can enlarge the optimal risk of MA by a substantial multiple; when $M_n$ is too small or $M_n$ is large enough and $\theta_{n,m}$ decays fast in $m$, MA restricted to the discrete weight set has no real disadvantage over MA with $\mathcal{W}_n$. In the Supplementary Material, we further present a comparison of MAs with nested discrete weight sets.

\section{Two Examples}\label{sec:twoexm}

In this section, we provide two examples to verify the theoretical results of Theorems \ref{thm:delta2}--\ref{thm:delta} and \ref{thm:discrete}, whose detailed derivations are given in Appendix \ref{subsec:examp}. Given $a, b>0$, the incomplete beta function is defined by $B(x;a,b)=\int_0^x t^{a-1}(1-t)^{b-1}\,dt$ for $0\leq x\leq 1$.

 %For the model \eqref{eq:model} of 
 In these two examples, we consider the model setting of \citet{peng2021}, that is model 
 \eqref{eq:model} with orthonormal design assumption \eqref{eq:orthcond}, 
 homoscedastic and uncorrelated error terms 
 %are homoscedastic and uncorrelated with variance $\sigma^2>0$, 
 and $\nu_m=m$ for $m=1,\ldots,p_n$.

\noindent  \textbf{Example \arabic{section}.1 (Slowly Decaying $\theta_{n,m}$).} The true coefficients are set to $\beta_m=m^{-\alpha}$ with $\alpha>1/2$. Then, $\theta_{n,m}=m^{-2\alpha}/\sigma^2$ and Condition A1 is satisfied. 
By some simple calculations, we can get 
%The global optimal model should include the first
 $m_n^{**}\sim (\frac{n}{\sigma^2})^{\frac{1}{2\alpha}}$. %terms. 
 We consider the following two situations on the number of candidate models $M_n$.

(i) When $\lim_{n\to \infty} M_n/m_n^{**}=0$, we have
\begin{equation*}
	\frac{1}{n}R_n(\bm{w}_{n,N}^*)\sim \frac{1}{n}R_n(\bm{w}_n^*)\sim 
	\begin{cases}
		\sum_{m=M+1}^{\infty} m^{-2\alpha}, & \text{if}~M_n\equiv M~\text{is fixed as}~n\to \infty,\\
		\frac{M_n^{-2\alpha+1}}{2\alpha-1}, & \text{if}~\lim_{n\to \infty} M_n= \infty~\text{but}~M_n=o(m_n^{**}).
	\end{cases}
\end{equation*}
This verifies that $R_n(\bm{w}_{n,N}^*)\sim R_n(\bm{w}_n^*)$ under Condition M1, which accords with 
Theorem \ref{thm:delta2} and 
the third conclusion of Theorem \ref{thm:discrete}. 

(ii) When $M_n/m_n^{**}\geq \underline{c}$ for some $\underline{c}>0$, we have
\begin{equation*}
	\frac{1}{n}R_n(\bm{w}_{n,N}^*)\asymp \frac{1}{n}R_n(\bm{w}_n^*)\asymp n^{-\frac{2\alpha-1}{2\alpha}}.
\end{equation*}
By a simple calculation, we know that $R_n(\bm{w}_{n,N}^*)- R_n(\bm{w}_n^*)$ is lower bounded by $\frac{\varpi\sigma^2}{2^{2\alpha+1}\varpi^{-2\alpha}+2} \left(\frac{n}{\sigma^2}\right)^{\frac{1}{2\alpha}}$, where $\varpi=\min\{\underline{c},(2N-1)^{-\frac{1}{2\alpha}}\}$. Moreover, if $\lim_{n\to \infty} M_n/m_n^{**}= \kappa$, $\kappa\in (0,\infty]$ and $M_n=o(p_n)$, we have
\begin{equation*}
	\lim_{n\to \infty} \frac{R_n(\bm{w}_n^*)}{R_n(\bm{w}_{n,N}^*)}=\frac{1}{\psi_N^*+\frac{\kappa^{-2\alpha+1}}{2\alpha}} \left[\frac{2\alpha-1}{4\alpha^2}\left\{ \frac{\pi}{\sin(\frac{\pi}{2\alpha})}-B\left(\frac{1}{1+\kappa^{2\alpha}}; 1-\frac{1}{2\alpha},\frac{1}{2\alpha}\right)\right\}+\frac{\kappa^{-2\alpha+1}}{2\alpha}\right],
\end{equation*}
where $\psi_N^*$ is defined in Appendix \ref{subsec:examp}. Furthermore, we can show that
\begin{equation*}
	\lim_{n\to \infty} \frac{R_n(\bm{w}_n^*)}{R_n(\bm{w}_{n,N}^*)}<1,
\end{equation*}
which verifies that $R_n(\bm{w}_{n,N}^*)- R_n(\bm{w}_n^*)\asymp R_n(\bm{w}_n^*)$, which accords with 
the first conclusion of 
Theorem \ref{thm:delta} 
and the second conclusion of Theorem \ref{thm:discrete}. Figure \ref{fig:ratio}(a) plots $\lim_{n\to \infty} \frac{R_n(\bm{w}_n^*)}{R_n(\bm{w}_{n,N}^*)}$ againsts $N\in \{1,\ldots,10\}$ for $\kappa=0.5, 1, 2$ and Figure \ref{fig:ratio}(b) plots $\lim_{n\to \infty} \frac{R_n(\bm{w}_n^*)}{R_n(\bm{w}_{n,N}^*)}$ againsts $\kappa\in (0,8)$ for $N=1, 2, 4$, where $\alpha=0.8$, which
further verifies that $\lim_{n\to \infty} \frac{R_n(\bm{w}_n^*)}{R_n(\bm{w}_{n,N}^*)}<1$.

\noindent  \textbf{Example \arabic{section}.2 (Fast Decaying $\theta_{n,m}$).} The true coefficients are set to $\beta_m=\exp(-cm)$ with $c>0$. Then, $\theta_{n,m}=\exp(-2cm)/\sigma^2$ and Condition A2 is satisfied. The global optimal model should include the first $m_n^{**}\sim \frac{1}{2c} \log \left(\frac{n}{\sigma^2}\right)$ terms. We consider the following three situations on the number of candidate models $M_n$.

(i) When $\limsup_{n\to \infty} M_n/m_n^{**}<1$, we have
\begin{equation*}
	\frac{1}{n}R_n(\bm{w}_{n,N}^*)\sim \frac{1}{n}R_n(\bm{w}_n^*)\sim \frac{\exp(-2cM_n)}{\exp(2c)-1},
\end{equation*}
which verifies that $R_n(\bm{w}_{n,N}^*)\sim R_n(\bm{w}_n^*)$ under Condition M1 and 
accords with 
Theorem \ref{thm:delta2} and the third conclusion of Theorem \ref{thm:discrete}.

(ii) When $M_n<m_n^{**}$ for any sufficiently large $n$ but $\lim_{n\to \infty} M_n/m_n^{**}=1$, we have
\begin{equation*}
	\frac{1}{n}R_n(\bm{w}_{n,N}^*)\sim \frac{1}{n}R_n(\bm{w}_n^*)\sim \frac{1}{2c} \frac{\sigma^2}{n}\log \left(\frac{n}{\sigma^2}\right)+\frac{\exp(-2cM_n)-\exp(-2cp_n)}{\exp(2c)-1},
\end{equation*}
which
 verifies that $R_n(\bm{w}_{n,N}^*)\sim R_n(\bm{w}_n^*)$ under Conditions M2 and A2, and 
accords with the second conclusion of Theorem \ref{thm:delta} and the third conclusion of Theorem \ref{thm:discrete}.

(iii) When $M_n\geq m_n^{**}$ for any sufficiently large $n$, we have
\begin{equation*}
	\frac{1}{n}R_n(\bm{w}_{n,N}^*)\sim \frac{1}{n}R_n(\bm{w}_n^*)\sim \frac{1}{2c} \frac{\sigma^2}{n}\log \left(\frac{n}{\sigma^2}\right).
\end{equation*}
which also
verifies that $R_n(\bm{w}_{n,N}^*)\sim R_n(\bm{w}_n^*)$ under Conditions M2 and A2, and 
accords with the second conclusion of Theorem \ref{thm:delta} and the third conclusion of Theorem \ref{thm:discrete}.

%{\colorchange As a result, Examples \arabic{section}.1 and \arabic{section}.2 verify Theorem \ref{thm:discrete}. Since $R_n(\bm{w}_{n,1}^*)=R_n(m_n^*)$, Examples \arabic{section}.1 and \arabic{section}.2 with $N=1$ also verify Theorems \ref{thm:delta2}--\ref{thm:delta}.}

\begin{figure}[htpb]
	\centering
	\includegraphics[scale=0.8]{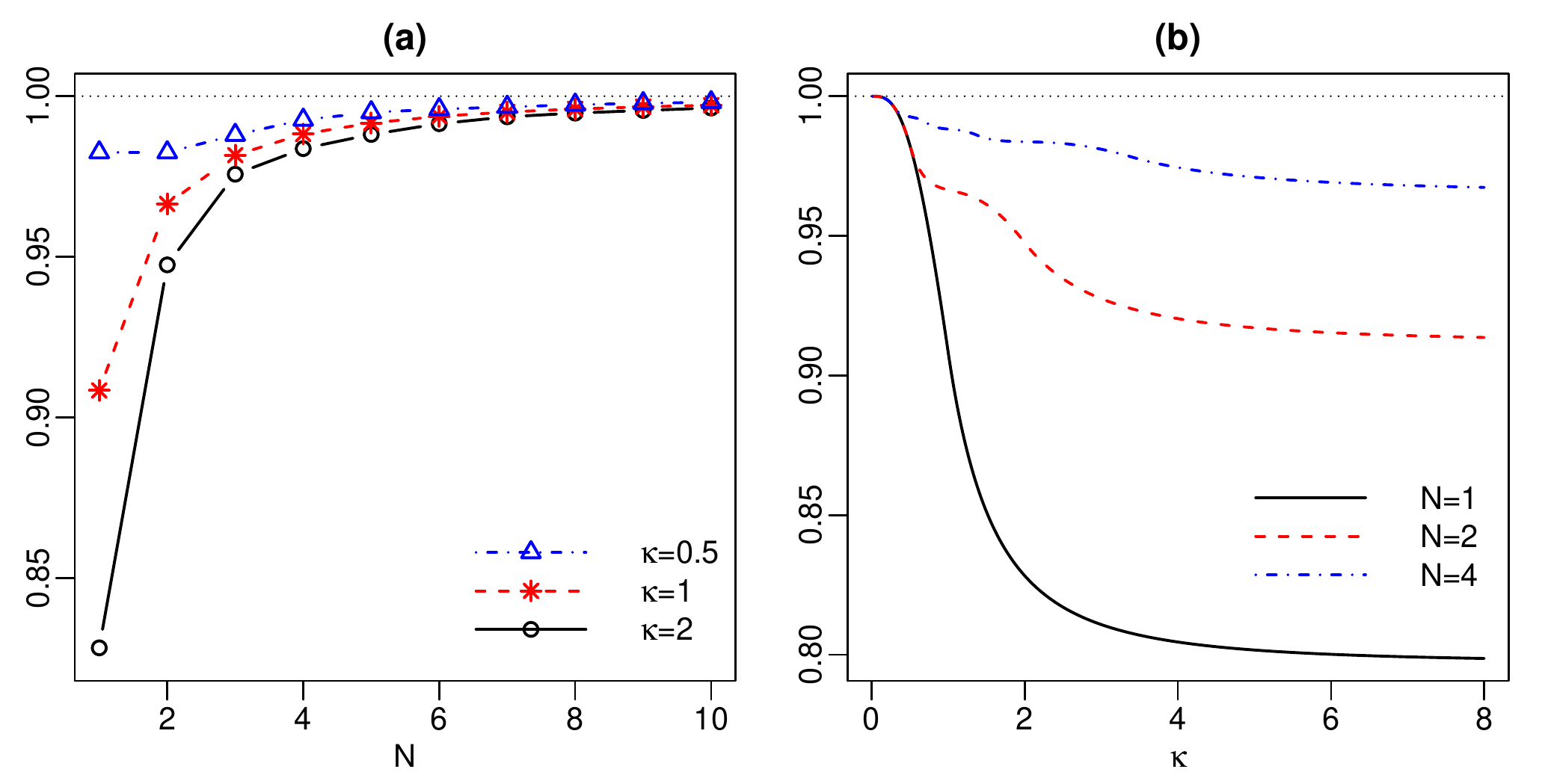}
	\caption{Numerical illustration for Example \arabic{section}.1 with $\alpha=0.8$. Left: plots of $\lim_{n\to \infty} \frac{R_n(\bm{w}_n^*)}{R_n(\bm{w}_{n,N}^*)}$ against $N\in \{1,\ldots,10\}$ for $\kappa=0.5, 1, 2$, respectively. Right: plots of $\lim_{n\to \infty} \frac{R_n(\bm{w}_{n,dN}^*)}{R_n(\bm{w}_{n,d}^*)}$ againsts $\kappa\in (0, 8)$ for $N=1, 2, 4$, respectively.}\label{fig:ratio}
\end{figure}

\section{Simulation Studies}\label{sec:simu}

In this section, we conduct several simulation studies to verify the theoretical results presented in Corollaries \ref{cor:main} and \ref{cor:main2}, where the specific MS and MA methods are compared.
% 
% Here we choose AIC, BIC, and LOO-CV as MS methods, and for MA, we choose MMA  \citep{hansen2007mma} for the Examples 1 and 3 below with the independent and homoscedastic errors and JMA \citep{hansen2012,zhang2013JMA} for the Example 2 below with the autocorrelated and heteroscedastic errors. 
Here we choose AIC, BIC, and LOO-CV as MS methods, and MMA and JMA for MA methods.  Specifically, we use the following three examples:
\begin{itemize}
	\item {\bf Example 1}: {\bf General nested framework} (i.e., $\nu_m\neq m$),
	homoscedastic and uncorrelated errors, and  (approximately) orthonormal design; 
		\item {\bf Example 2}: Typical nested framework (i.e., $\nu_m= m$),
		{\bf heteroscedastic and autocorrelated errors},
		and (approximately) orthonormal design;
			\item {\bf Example 3}: Typical nested framework (i.e., $\nu_m= m$),
			homoscedastic and uncorrelated errors,
			and {\bf non-orthonormal design}.  
\end{itemize}
To evaluate the estimators, we compute the risks of the competing methods by computing averages across 1000 replications. Supplementary materials contain more simulation studies. \\

\noindent  \textbf{Example 1 (General nested framework)} We use the same set-up as that of \citet{peng2021} except the coefficients $\beta_m$'s. Specifically, suppose the data come from the model \eqref{eq:model}, where $p_n=\lfloor 5n^{2/3}\rfloor$, $x_{i1}=1$, the remaining $x_{ij}$ are independent and identically distributed (iid) from $N(0,1)$, and the random errors $\varepsilon_i$ are iid from $N(0,\sigma^2)$ and are independent of $x_{ij}$'s. The population $R^2$ is denoted by $R^2=\mathrm{var}(\mu_i)/\mathrm{var}(y_i)$, which is controlled in $0.25$, $0.5$ or $0.75$ via the parameter $\sigma^2$. We consider a more general nested model setting than that of \citet{peng2021} by setting
\begin{equation*}
	\nu_m=\begin{cases}
		5 \lfloor m/2\rfloor+2, & m~\text{is odd} \\
		5 \lfloor m/2\rfloor, & m~\text{is even}
	\end{cases},
\quad m=1,\ldots,q_n-1
\end{equation*}
and $\nu_{q_n}=p_n$. Thus, the size of the $m$th group of predictors is 2 when $m$ is odd and 3 when $m$ is even, $m=1,\ldots,q_n-1$. We consider two cases with different coefficient decaying orders: 
\begin{itemize}
	\item Case 1. $\beta_j= m^{-\alpha_1}$ when $x_{ij}$ is in the $m$th group and $\alpha_1$ is set to be 1, 1.5 or 2.
	\item Case 2. $\beta_j= \exp(-\alpha_2 m)$ when $x_{ij}$ is in the $m$th group and $\alpha_2$ is set to be 1, 1.5 or 2.
\end{itemize}
For Case 1, we know that $\theta_{n,m}$ converges to $\theta_m^*= m^{-2\alpha_1}/\sigma^2$ in probability and then Condition B1 is satisfied (in our theories, we take the predictors as fixed, but in simulation, they are random. When verifying our conditions, we do not consider this randomness), and for Case 2, $\theta_{n,m}$ converges to $\theta_m^*= \exp(-2\alpha_2 m)/\sigma^2$ in probability and then Condition B2 is satisfied. The sample size $n$ varies at 50, 500, 1000, 2000, 3000, and 4000. The number of candidate models is determined by $M_n=\mathrm{INT}(3n^{1/3})$, where the function $\mathrm{INT}(a)$ returns the nearest integer from $a$. In each simulation setting of the combination of $n$, $R^2$, and $\alpha_1$ (or $\alpha_2$), we normalize the risks of the MS methods by dividing by the risk of MMA. 

Simulation results are summarized in Figures \ref{fig:example1_1}--\ref{fig:example1_2}. In each figure, the simulation results with three different coefficient decaying orders are displayed in rows (a), (b), and (c), respectively. Note that Figure \ref{fig:example1_1} are under Case 1 (slowly decaying coefficients) and Figure \ref{fig:example1_2} are under Case 2 (fast decaying coefficients). Since both AIC and LOO-CV are asymptotically optimal for Example 1, as expected, their performances are very close for large sample size. In the slowly decaying $\theta_m^*$ case, the performance gap between AIC (or LOO-CV) and MMA does not vanish when $n$ increases, while in the fast decaying $\theta_m^*$ case, it becomes very small when $n$ is large, which are consistent with the results of Corollary \ref{cor:main}.

Following \citet{peng2021}, we also include BIC in our simulation, although often, BIC is not asymptotically optimal.
% in nonparametric scenarios. 
In Case 1, the advantage of AIC over BIC becomes increasingly larger as $n$ increases from 50 to 4000, while in Case 2 with fast decaying $\theta_m^*$, BIC is competitive with AIC in some scenarios. This phenomenon was also observed by \citet{peng2021}.\\

\noindent  \textbf{Example 2 (Heteroscedastic and autocorrelated errors)} 
The setting of this example is the same as that of Example 1 except that
the typical nested framework with $\nu_m=m$
and  heteroscedastic and autocorrelated errors are considered. 
% 
%Suppose the data come from the model \eqref{eq:model}, where $p_n=\lfloor 5n^{2/3}\rfloor$, $x_{i1}=1$, and the remaining $x_{ij}$ are independently generated from $N(0,1)$. 
%In specific, 
We utilize
 the same error process of \citet{zhang2013JMA} that is both heteroscedastic and autocorrelated. Specifically, the error process is given by $\varepsilon_i=\varepsilon_{i1}+\varepsilon_{i2}$, $\varepsilon_{i1}$'s are independent observations from the $N(0,x_{i2}^2)$ distribution, and $\varepsilon_{i2}$ follows an AR(1) process with an autocorrelation coefficient $\rho_1=0.5$, where $\varepsilon_{i2}=\rho_1\varepsilon_{i-1,2}+e_i$, $\varepsilon_{12}\sim N(0,1)$, and $e_i$'s are iid from $N(0,1-\rho_1^2)$ and are independent of $\varepsilon_{i2}$'s. Then, the covariance matrix of $\bm{\varepsilon}$ given $x_{ij}$'s is $\bm{\Omega}=\bm{\Omega}_1+\bm{\Omega}_2$, where $\bm{\Omega}_1=\mathrm{diag}\{x_{12}^2,\ldots,x_{n2}^2\}$ and $\bm{\Omega}_2=(\rho_1^{|k-l|})_{k,l=1,\ldots,n}$. By a simple calculation, we have
\begin{equation*}
	\mathrm{tr}( \mathbf{P}_m \bm{\Omega}_1) \xrightarrow{p}
	\begin{cases}
		1, & \nu_m=1 \\
		\nu_m+2, & \nu_m\geq 2
	\end{cases},\quad 
\mathrm{tr}( \mathbf{P}_m \bm{\Omega}_2) \xrightarrow{p} \frac{2\rho_1}{1-\rho_1}+\nu_m.
\end{equation*}
%We consider the typical nested framework with $\nu_m=m$ and the number of the candidate models $M_n=\mathrm{INT}(5n^{1/3})$. 
Therefore, for any fixed $m$, $\theta_{n,m}\to \theta_m^*= \beta_m^2/\zeta_m$ in probability, where 
\begin{equation*}
\zeta_m=\begin{cases}
	\frac{2}{1-\rho_1}, & m=1, \\
	4, & m=2, \\
	2, & m\geq 3.
\end{cases}
\end{equation*}
We consider two cases with different decaying orders of $\theta_m^*$: 
\begin{itemize}
	\item Case 1 (With $\theta_m^*$ satisfying Condition B1). Here, $\beta_m=c\sqrt{\zeta_m} m^{-\alpha_1}$ and $\alpha_1$ is set to be 1, 1.5 or 2.
	\item Case 2 (With $\theta_m^*$ satisfying Condition B2). Here, $\beta_m=c\sqrt{\zeta_m} \exp(-\alpha_2 m)$ and $\alpha_2$ is set to be 1, 1.5 or 2.
\end{itemize}
As \cite{hansen2012},
the parameter $c$ was selected to control the
approximate population $\tilde{R}^2=c^2/(1+c^2)$ to vary on $0.25$, $0.5$, and $0.75$. The sample size is varied among $n=500$, 1000, 2000, 3000, 4000, and 5000. To verify the results in Corollary \ref{cor:main2}, we include JMA without the restriction $\sum_{m=1}^{M_n} w_m=1$, denoted by JMA2 %\citep[JMA2,][]{ando2014,zhao2016} 
as a competing method. In each simulation setting of combination of $n$, $\tilde{R}^2$, and $\alpha_1$ (or $\alpha_2$), we normalize the risks of the MS methods and JMA2 by dividing by the risk of JMA.

Simulation results are shown in Figures \ref{fig:example2_1}--\ref{fig:example2_2}. In each figure, the simulation results with three different decaying orders of $\theta_m^*$ are displayed in rows (a), (b), and (c), respectively. From these results, we see that in the slowly decaying $\theta_m^*$ case (Figure \ref{fig:example2_1}), the performance gap between LOO-CV and JMA does not vanish when sample size increases, while in the fast decaying $\theta_m^*$ case (Figure \ref{fig:example2_2}), it becomes very small when the sample size is large, which are consistent with the results of Corollary \ref{cor:main}. Note that from Figure \ref{fig:example2_2}, the performances of AIC and JMA are not  consistently close since AIC may not be  asymptotically optimal in Example 2 because of heteroscedastisity. Another observation is that the performances of JMA2 and JMA are very close when $n$ is sufficiently large, which verify the results in Corollary \ref{cor:main2}. Moreover, we can observe the same phenomena as that in Example 1 for a comparison of AIC and BIC.\\

\noindent  \textbf{Example 3 (Non-orthonormal design)} 
The setting of this example is the same as that of Example 1 except that
the typical nested framework with $\nu_m=m$
and predictors are non-orthonormal. 
%Suppose the data come from the model \eqref{eq:model}, where $p_n=\lfloor 5n^{2/3}\rfloor$,
Specifically, 
 the predictors $(x_{i1},\ldots,x_{ip_n})^{\top}$, $i=1,\ldots,n$, are iid normal random vectors with zero mean and covariance matrix between $k$th and $l$th elements being $\rho_2^{|k-l|}$, and the random errors $\varepsilon_i$ are iid from $N(0,\sigma^2)$ and are independent of $x_{ij}$'s. Here, $\rho_2$ is set to be $0.5$. %The SNR $R^2=\mathrm{var}(\mu_i)/\mathrm{var}(y_i)$ is controlled in $0.25$, $0.5$ or $0.75$ via the parameter $\sigma^2$. We consider the typical nested framework with $\nu_m=m$ and the number of the candidate models $M_n=\mathrm{INT}(5n^{1/3})$. 
 It is easy to prove that for any fixed $m$, \begin{equation*}
	\frac{1}{n} \bm{\mu}^{\top} \mathbf{P}_m \bm{\mu} \xrightarrow{p} \lim_{n\to \infty} \bm{\beta}_{p_n}^{\top}\bm{\Sigma}_{m\times p_n}^{\top}\bm{\Sigma}_{m\times m}^{-1}\bm{\Sigma}_{m\times p_n} \bm{\beta}_{p_n},
\end{equation*}
where $\bm{\Sigma}_{d_1\times d_2}$ is a $d_1\times d_2$ matrix with $(k,l)$th element being $\rho_2^{|k-l|}$ and $\bm{\beta}_{p_n}=(\beta_1,\ldots,\beta_{p_n})^{\top}$. It follows that $\theta_{n,m}=\bm{\mu}^{\top}( \mathbf{P}_m- \mathbf{P}_{m-1})\bm{\mu}/(n\sigma^2) \to \theta_m^*=\xi_m/\sigma^2$ in probability, where $\xi_1=\lim_{n\to \infty} (\bm{\Sigma}_{1\times p_n}\bm{\beta}_{p_n})^2$ and
\begin{equation*}
	\xi_m=\lim_{n\to \infty} \bm{\beta}_{p_n}^{\top} \left\{\bm{\Sigma}_{m\times p_n}^{\top}\bm{\Sigma}_{m\times m}^{-1}\bm{\Sigma}_{m\times p_n}-\bm{\Sigma}_{(m-1)\times p_n}^{\top}\bm{\Sigma}_{(m-1)\times (m-1)}^{-1}\bm{\Sigma}_{(m-1)\times p_n}\right\} \bm{\beta}_{p_n}
\end{equation*}
for $m=2,\ldots,p_n$. By some calculations, we can obtain simple forms of $\xi_m$ as follows 
\begin{equation}\label{eq:exmp3}
	\xi_1=\lim_{n\to \infty} \left(\sum_{j=1}^{p_n} \beta_j \rho_2^{j-1}\right)^2,\quad \xi_m=\lim_{n\to \infty} (1-\rho_2^2)\left(\sum_{j=m}^{p_n} \beta_j \rho_2^{j-m}\right)^2,~m\geq 2.
\end{equation}
We consider two cases with different decaying orders of $\theta_m^*$: 
\begin{itemize}
	\item Case 1 (With $\theta_m^*$ satisfying Condition B1). Here, $\xi_m=m^{-2\alpha_1}$ and $\alpha_1$ is set to be 1, 1.5 or 2.
	\item Case 2 (With $\theta_m^*$ satisfying Condition B2). Here, $\xi_m=\exp(-2\alpha_2 m)$ and $\alpha_2$ is set to be 1, 1.5 or 2.
\end{itemize}
We can set different coefficient $\beta_j$ via \eqref{eq:exmp3} such that Case 1 and Case 2 hold respectively. Without loss of generality, we assume $\beta_j\geq 0$ for all $j$. Then from \eqref{eq:exmp3}, we have
\begin{equation*}
	\beta_1=\sqrt{\xi_1}-\rho_2\sqrt{\frac{\xi_2}{1-\rho_2^2}}\quad \text{and} \quad \beta_j=\frac{\sqrt{\xi_j}-\rho_2\sqrt{\xi_{j+1}}}{\sqrt{1-\rho_2^2}},\quad j\geq 2.
\end{equation*}
The sample size $n$ varies at 50, 500, 1000, 2000, 3000, and 4000. In each simulation setting of the combination of $n$, $R^2$, and $\alpha_1$ (or $\alpha_2$), we normalize the risks of the MS methods by dividing by the risk of MMA. Figures \ref{fig:example3_1}--\ref{fig:example3_2} display the simulation results for Case 1 and Case 2, respectively. In each figure, the simulation results with three different decaying orders of $\theta_m^*$ are displayed in rows (a), (b), and (c), respectively. From these results, we can see the same observations as that in Example 1, which also verify the previous theoretical findings.

In above examples, we use the decaying orders of GVI to determine the performance of MAs and MSs and
set large enough $M_n$, i.e., (Condition M2 is satisfied). In Section \label{sim:new}
of the Supplementary Material, 
we further design Example 4 to verify Corollary \ref{cor:main} when $M_n$ is too small (Condition M1 is satisfied). 

\section{Conclusion}\label{sec:concl}

This paper extends the work of \citet{peng2021} on comparison of MS and MA to a general model setting, where we allow the predictors are non-orthonormal, the error terms are heteroscedastic and autocorrelated,
and some predictors are totally unimportant. We obtain the results that the number of candidate models $M_n$ and the decaying order of $\{\theta_{n,m}\}_{m=1}^{d_n}$ determine when MA is better than MS. Specifically, when $M_n$ is large enough and $\theta_{n,m}$ decays slowly in $m$, the benefit of MA over MS is real; when $M_n$ is too small or $M_n$ is large enough and $\theta_{n,m}$ decays fast in $m$, the risks of MA and MS are asymptotically equivalent. %These results are the same as that of \citet{peng2021} in their model setting.
 Furthermore, the obtained results are extended to compare MAs with the weights belonged to three different weight sets.

Along with the paper, there are a few open questions. First, an interesting issue is how to order the predictors and prepare nested candidate models such that the risk gain of MA is 
optimal. Although various procedures are proposed to order the predictors in the implementation of MA such as forward selection approach \citep{claeskens2006}, the marginal correlation \citep{ando2014,ando2017,zhang2016}, and solution path algorithm of penalized regression \citep{zhang2020PMA,feng2020nested}, there is still a lack of theoretical study on an optimal way of ordering the predictors. Second, it is interesting to develop a data-driven way to choose the number of the candidate models. Another appealing direction is to compare MS and MA in the non-nested model setting. However, in this case, it is difficult to characterize the unknown optimal model $m_n^*$ and weights $\bm{w}_n^*$. A deeper and detailed investigation of these issues warrants further studies.

\renewcommand{\theequation}{{A.\arabic{equation}}}
\renewcommand{\thelemma}{{\it A.\arabic{lemma}}}
\renewcommand{\thesubsection}{{\it A.\arabic{subsection}}}
\setcounter{equation}{0}

\section*{Appendix}

\subsection{Proof of Theorem \ref{thm:mw}}\label{sec:thm1}
The risk of the $m$th candidate model is
	\begin{eqnarray*}
	R_n(m)\hspace{-6mm}&&=E\|\hat{\bm{\mu}}_m-\bm{\mu}\|^2\nonumber\\
	&&=\mathrm{tr}\left[ E\left\{(\hat{\bm{\mu}}_m-\bm{\mu})(\hat{\bm{\mu}}_m-\bm{\mu})^\top\right\}\right]\nonumber\\
	&&=\mathrm{tr}\left\{ (E\hat{\bm{\mu}}_m-\bm{\mu})(E\hat{\bm{\mu}}_m-\bm{\mu})^\top\right\} + \mathrm{tr}\left\{\mathrm{var}(\hat{\bm{\mu}}_m)\right\}\nonumber\\
	&&=\bm{\mu}^{\top}(\mathbf{I}_n-\mathbf{P}_m)\bm{\mu}+\mathrm{tr}(\mathbf{P}_m\bm{\Omega}).
\end{eqnarray*}
Observe that
\begin{equation}\label{eq:diffrisk}
	R_n(m)-R_n(m-1)=\mathrm{tr}\{( \mathbf{P}_m- \mathbf{P}_{m-1})\bm{\Omega}\}-\bm{\mu}^{\top} ( \mathbf{P}_m- \mathbf{P}_{m-1}) \bm{\mu}=n\mathrm{tr}\{( \mathbf{P}_m- \mathbf{P}_{m-1})\bm{\Omega}\}\biggl(\frac{1}{n}-\theta_{n,m}\biggr),
\end{equation}
where $\theta_{n,m}$ is defined in \eqref{eq:index}. Note that $\mathrm{tr}\{( \mathbf{P}_m- \mathbf{P}_{m-1})\bm{\Omega}\}>0$ and $\{\theta_{n,m}\}_{m=1}^{q_n}$ is non-increasing from Assumption 3. Then, under Assumption 6, it is easy to see that the global optimal model $m_n^{**}$ that minimizes $R_n(m)$ on $\{1,\ldots,d_n\}$ satisfies
\begin{equation}\label{eq:optm2}
	\theta_{n,m_n^{**}}>\frac{1}{n}\geq \theta_{n,m_n^{**}+1}.
\end{equation}
Hence, the risk of the optimal model $m_n^*$ is
\begin{equation}\label{eq:riskms}
	R_n(m_n^*)=\bm{\mu}^{\top}( \mathbf{I}_n- \mathbf{P}_{m_n^*})\bm{\mu}+\mathrm{tr}( \mathbf{P}_{m_n^*}\bm{\Omega}),
\end{equation}
where when $M_n<m_n^{**}$, $m_n^*=M_n$ and $\theta_{n,m_n^*}>1/n$; when $M_n\geq m_n^{**}$, $m_n^*=m_n^{**}$ and
\begin{equation}\label{eq:optm}
	\theta_{n,m_n^*}>\frac{1}{n}\geq \theta_{n,m_n^*+1}.
\end{equation}

The risk of the MA estimator with weights $\bm{w}$ is 
\begin{equation*}
	\begin{split}
		R_n(\bm{w})&=E\|\hat{\bm{\mu}}(\bm{w})-\bm{\mu}\|^2=E\| \mathbf{P}(\bm{w})\bm{y}-\bm{\mu}\|^2 \\
		&=\bm{\mu}^{\top}\{ \mathbf{P}(\bm{w})- \mathbf{I}_n\}^2\bm{\mu}+\mathrm{tr}\{ \mathbf{P}^2(\bm{w})\bm{\Omega}\}.
	\end{split}
\end{equation*}
Rewrite $ \mathbf{P}(\bm{w})=\sum_{m=1}^{M_n} \gamma_m( \mathbf{P}_m- \mathbf{P}_{m-1})$, where $\gamma_m=\sum_{j=m}^{M_n} w_j$ and $ \mathbf{P}_0=\bm{0}$. Since $ \mathbf{P}_m  \mathbf{P}_l= \mathbf{P}_{\min(m,l)}$ for the nested candidate models, it is easy to verify that $\{\mathbf{P}_m- \mathbf{P}_{m-1}\}_{m=1}^{M_n}$ are mutually orthonormal projection matrices, i.e., 
\begin{equation*}
	(\mathbf{P}_{m_1}- \mathbf{P}_{m_1-1})(\mathbf{P}_{m_2}- \mathbf{P}_{m_2-1})=
	\begin{cases}
		\mathbf{P}_{m_1}- \mathbf{P}_{m_1-1}, & \mathrm{if}~m_1=m_2, \\
		\mathbf{0}, & \mathrm{if}~m_1\neq m_2.
	\end{cases}
\end{equation*}
Using the above fact, $R_m(\bm{w})$ is further expanded as
\begin{eqnarray}\label{eq:riskRw}
	R_n(\bm{w})=\hspace{-6mm}&&\sum_{m=1}^{M_n} \biggl(\gamma_m^2\Bigl[ \bm{\mu}^{\top}( \mathbf{P}_m- \mathbf{P}_{m-1})\bm{\mu}+\mathrm{tr}\{( \mathbf{P}_m- \mathbf{P}_{m-1})\bm{\Omega}\}\Bigr] \notag \\
	&& -2\gamma_m \bm{\mu}^{\top}( \mathbf{P}_m- \mathbf{P}_{m-1})\bm{\mu}+\bm{\mu}^{\top}( \mathbf{P}_m- \mathbf{P}_{m-1})\bm{\mu}\biggr) +\bm{\mu}^{\top}( \mathbf{I}_n- \mathbf{P}_{M_n})\bm{\mu}.
\end{eqnarray}
It is straightforward to show that the infeasible optimal weight $\bm{w}_n^*=(w_{n,1}^*,\ldots,w_{n,M_n}^*)^{\top}$ can be obtained by setting $w_{n,m}^*=\gamma_{n,m}^*-\gamma_{n,m+1}^*$
for $m=1,\ldots,M_n-1$ and $w_{n,M_n}^*=\gamma_{n,M_n}^*$, where $\gamma_{n,1}^*=1$ and
\begin{equation}\label{eq:gamma}
	\gamma_{n,m}^*=\frac{\bm{\mu}^{\top}( \mathbf{P}_m- \mathbf{P}_{m-1})\bm{\mu}}{\bm{\mu}^{\top}( \mathbf{P}_m- \mathbf{P}_{m-1})\bm{\mu}+\mathrm{tr}\{( \mathbf{P}_m- \mathbf{P}_{m-1})\bm{\Omega}\}}=\frac{\theta_{n,m}}{\theta_{n,m}+1/n},\quad m=2,\ldots,M_n.
\end{equation}
Hence, the risk of the optimal MA estimator is
\begin{equation}\label{eq:riskma}
	R_n(\bm{w}_n^*)=\mathrm{tr}( \mathbf{P}_1\bm{\Omega})+\sum_{m=2}^{M_n} \frac{\mathrm{tr}\{( \mathbf{P}_m- \mathbf{P}_{m-1})\bm{\Omega}\} \bm{\mu}^{\top}( \mathbf{P}_m- \mathbf{P}_{m-1})\bm{\mu}}{\bm{\mu}^{\top}( \mathbf{P}_m- \mathbf{P}_{m-1})\bm{\mu}+\mathrm{tr}\{( \mathbf{P}_m- \mathbf{P}_{m-1})\bm{\Omega}\}}+\bm{\mu}^{\top}( \mathbf{I}_n- \mathbf{P}_{M_n})\bm{\mu}.
\end{equation}
Combining \eqref{eq:riskms} and \eqref{eq:riskma}, the potential advantage of MA over MS is
\begin{eqnarray}\label{eq:delta}
	\Delta_n\hspace{-6mm}&&=R_n(m_n^*)- R_n(\bm{w}_n^*) \nonumber\\
	&&= \sum_{m=2}^{m_n^*} \biggl[\mathrm{tr}\{( \mathbf{P}_{m}- \mathbf{P}_{m-1})\bm{\Omega}\}- \frac{\mathrm{tr}\{( \mathbf{P}_m- \mathbf{P}_{m-1})\bm{\Omega}\} \bm{\mu}^{\top}( \mathbf{P}_m- \mathbf{P}_{m-1})\bm{\mu}}{\bm{\mu}^{\top}( \mathbf{P}_m- \mathbf{P}_{m-1})\bm{\mu}+\mathrm{tr}\{( \mathbf{P}_m- \mathbf{P}_{m-1})\bm{\Omega}\}}\biggr] \notag \\
	&&\quad +\sum_{m=m_n^*+1}^{M_n} \frac{ \{\bm{\mu}^{\top}( \mathbf{P}_m- \mathbf{P}_{m-1})\bm{\mu}\}^2}{\bm{\mu}^{\top}( \mathbf{P}_m- \mathbf{P}_{m-1})\bm{\mu}+\mathrm{tr}\{( \mathbf{P}_m- \mathbf{P}_{m-1})\bm{\Omega}\}},
\end{eqnarray}
which implies that, as is expected, the optimal risk of MA is not larger than that of MS, i.e., $R_n(\bm{w}_n^*)\leq R_n(m_n^*)$. We consider two scenarios that $M_n<m_n^{**}$ and $M_n\geq m_n^{**}$. 

When $M_n< m_n^{**}$, $m_n^*=M_n$ and it follows from Assumption 3 and $\theta_{n,m_n^*}>1/n$ that $\{\gamma_{n,m}^*\}_{m=1}^{M_n}$ is non-increasing and $\gamma_{n,m_n^*}^*>1/2$. Then, for a sufficiently large $n$,
\begin{eqnarray*}
	R_n(\bm{w}_n^*)\hspace{-6mm}&&\geq \sum_{m=1}^{m_n^*} \frac{\mathrm{tr}\{( \mathbf{P}_m- \mathbf{P}_{m-1})\bm{\Omega}\} \bm{\mu}^{\top}( \mathbf{P}_m- \mathbf{P}_{m-1})\bm{\mu}}{\bm{\mu}^{\top}( \mathbf{P}_m- \mathbf{P}_{m-1})\bm{\mu}+\mathrm{tr}\{( \mathbf{P}_m- \mathbf{P}_{m-1})\bm{\Omega}\}}+\bm{\mu}^{\top}( \mathbf{I}_n- \mathbf{P}_{m_n^*})\bm{\mu} \\
	&&=\sum_{m=1}^{m_n^*} \gamma_{n,m}^* \mathrm{tr}\{( \mathbf{P}_m- \mathbf{P}_{m-1})\bm{\Omega}\}+\bm{\mu}^{\top}( \mathbf{I}_n- \mathbf{P}_{m_n^*})\bm{\mu} \\
	&&\geq \gamma_{n,m_n^*}^*\sum_{m=1}^{m_n^*} \mathrm{tr}\{( \mathbf{P}_m- \mathbf{P}_{m-1})\bm{\Omega}\}+\bm{\mu}^{\top}( \mathbf{I}_n- \mathbf{P}_{m_n^*})\bm{\mu} \\
	&&\geq \frac{1}{2}\mathrm{tr}( \mathbf{P}_{m_n^*}\bm{\Omega})+\bm{\mu}^{\top}( \mathbf{I}_n- \mathbf{P}_{m_n^*})\bm{\mu}\geq \frac{1}{2}R_n(m_n^*).
\end{eqnarray*}

When $M_n\geq m_n^{**}$, it follows from Assumption 3 and \eqref{eq:optm} that $\gamma_{n,m_n^*}^*>1/2\geq \gamma_{n,m_n^*+1}^*$. Then, for a sufficiently large $n$,
\begin{eqnarray*}
		R_n(\bm{w}_n^*)\hspace{-6mm}&&\geq \sum_{m=1}^{M_n} \frac{\mathrm{tr}\{( \mathbf{P}_m- \mathbf{P}_{m-1})\bm{\Omega}\} \bm{\mu}^{\top}( \mathbf{P}_m- \mathbf{P}_{m-1})\bm{\mu}}{\bm{\mu}^{\top}( \mathbf{P}_m- \mathbf{P}_{m-1})\bm{\mu}+\mathrm{tr}\{( \mathbf{P}_m- \mathbf{P}_{m-1})\bm{\Omega}\}}+\bm{\mu}^{\top}( \mathbf{I}_n- \mathbf{P}_{M_n})\bm{\mu} \\
		&&= \sum_{m=1}^{m_n^*} \gamma_{n,m}^* \mathrm{tr}\{( \mathbf{P}_m- \mathbf{P}_{m-1})\bm{\Omega}\}+\sum_{m=m_n^*+1}^{M_n} (1-\gamma_{n,m}^*)\bm{\mu}^{\top}( \mathbf{P}_m- \mathbf{P}_{m-1})\bm{\mu}+\bm{\mu}^{\top}( \mathbf{I}_n- \mathbf{P}_{M_n})\bm{\mu} \\
		&&\geq \gamma_{n,m_n^*}^* \sum_{m=1}^{m_n^*} \mathrm{tr}\{( \mathbf{P}_m- \mathbf{P}_{m-1})\bm{\Omega}\}+(1-\gamma_{n,m_n^*+1}^*)\sum_{m=m_n^*+1}^{M_n} \bm{\mu}^{\top}( \mathbf{P}_m- \mathbf{P}_{m-1})\bm{\mu}+\bm{\mu}^{\top}( \mathbf{I}_n- \mathbf{P}_{M_n})\bm{\mu} \\
		&&\geq \frac{1}{2}\mathrm{tr}( \mathbf{P}_{m_n^*}\bm{\Omega})+\frac{1}{2}\bm{\mu}^{\top}( \mathbf{P}_{M_n}- \mathbf{P}_{m_n^*})\bm{\mu}+\bm{\mu}^{\top}( \mathbf{I}_n- \mathbf{P}_{M_n})\bm{\mu} \\
		&&\geq \frac{1}{2}R_n(m_n^*).
\end{eqnarray*}
Therefore, we have $R_n(m_n^*)\geq R_n(\bm{w}_n^*)\geq R_n(m_n^*)/2$ for any sufficiently large $n$, which yields that $R_n(m_n^*)$ and $R_n(\bm{w}_n^*)$ have the same order. This completes the proof of Theorem \ref{thm:mw}.

\subsection{Proof of Theorem \ref{thm:mn}}
We first prove (i). Let $C>0$ be a given sufficiently large constant. From Assumption 5, there exists a constant $K_C^*=\max\{K_m, \lfloor 2/\bar{\theta}_{\lfloor C\rfloor+1}\rfloor+1\}>0$ such that $\theta_{n,\lfloor C\rfloor+1}-1/n\geq \bar{\theta}_{\lfloor C\rfloor+1}/2>0$ for any $n\geq K_C^*$. Since $m_n^{**}$ satisfies $1/n\geq \theta_{n,m_n^{**}+1}$ from \eqref{eq:optm2}, we have
\begin{equation*}
	\theta_{n,\lfloor C\rfloor+1}-\theta_{n,m_n^{**}+1}\geq \theta_{n,\lfloor C\rfloor+1}-\frac{1}{n}>0,
\end{equation*}
which, along with Assumption 3, leads to $m_n^{**}+1\geq \lfloor C\rfloor+2$. This implies that for any constant $C>0$, there exists a constant $K_C^*>0$ such that $m_n^{**}\geq \lfloor C\rfloor+1>C$ for any $n\geq K_C^*$, i.e., $\lim_{n\to \infty} m_n^{**}=\infty$. This completes the proof of Theorem \ref{thm:mn}(i).

Next, we prove (ii). When $M_n\geq m_n^{**}$, $m_n^*=m_n^{**}$ and thus
\begin{equation*}
	R_n(m_n^*)\geq \mathrm{tr}( \mathbf{P}_{m_n^*}\bm{\Omega})=\mathrm{tr}( \mathbf{P}_{m_n^{**}}\bm{\Omega})\geq c_1\nu_{m_n^{**}}\geq c_1 m_n^{**}\to \infty.
\end{equation*}
When $M_n<m_n^{**}$, $m_n^*=M_n$ and thus by \eqref{eq:optm2} and Assumptions 2--3,
\begin{eqnarray}\label{eq:M12}
	R_n(m_n^*)\hspace{-6mm}&&=R_n(M_n)=\bm{\mu}^{\top}(\mathbf{I}_n- \mathbf{P}_{M_n})\bm{\mu}+\mathrm{tr}( \mathbf{P}_{M_n}\bm{\Omega}) \notag \\
	&& \geq \bm{\mu}^{\top}(\mathbf{P}_{m_n^{**}}- \mathbf{P}_{M_n})\bm{\mu}+\mathrm{tr}( \mathbf{P}_{M_n}\bm{\Omega}) \notag \\
	&& = \sum_{m=M_n+1}^{m_n^{**}} \mathrm{tr}\{( \mathbf{P}_m- \mathbf{P}_{m-1})\bm{\Omega}\}n\theta_{n,m}+\mathrm{tr}( \mathbf{P}_{M_n}\bm{\Omega}) \notag \\
	&&\geq  n\theta_{n,m_n^{**}}\sum_{m=M_n+1}^{m_n^{**}} \mathrm{tr}\{( \mathbf{P}_m- \mathbf{P}_{m-1})\bm{\Omega}\}+\mathrm{tr}( \mathbf{P}_{M_n}\bm{\Omega}) \notag \\
	&&\geq \mathrm{tr}\{(\mathbf{P}_{m_n^{**}}- \mathbf{P}_{M_n})\bm{\Omega}\}+\mathrm{tr}( \mathbf{P}_{M_n}\bm{\Omega}) \notag \\
	&& =\mathrm{tr}(\mathbf{P}_{m_n^{**}}\bm{\Omega})\geq c_1\nu_{m_n^{**}}\geq c_1 m_n^{**}\to \infty.
\end{eqnarray}
Therefore, $R_n(m_n^*)\to \infty$ as $n\to \infty$ for any $M_n$. Combining it with Theorem \ref{thm:mw}, we have $R_n(\bm{w}_n^*)\geq \frac{1}{2} R_n(m_n^*)\to \infty$ as $n\to \infty$. This completes the proof of Theorem \ref{thm:mn}(ii).

\subsection{Proof of Theorem \ref{thm:delta2}}
Under Condition M1, $\lim_{n\to \infty} M_n/m_n^{**}=0$, which implies that $M_n<m_n^{**}$ when $n$ is large enough and thus $m_n^*=M_n$. By \eqref{eq:delta}, for a sufficiently large $n$,
\begin{eqnarray}\label{eq:M11}
	\Delta_n \hspace{-6mm}&&=\sum_{m=2}^{M_n} \frac{[\mathrm{tr}\{( \mathbf{P}_m- \mathbf{P}_{m-1})\bm{\Omega}\}]^2}{\bm{\mu}^{\top}( \mathbf{P}_m- \mathbf{P}_{m-1})\bm{\mu}+\mathrm{tr}\{( \mathbf{P}_m- \mathbf{P}_{m-1})\bm{\Omega}\}} \notag \\
	&&\leq \sum_{m=2}^{M_n} \mathrm{tr}\{( \mathbf{P}_m- \mathbf{P}_{m-1})\bm{\Omega}\} \notag \\
	&&=\mathrm{tr}\{( \mathbf{P}_{M_n}- \mathbf{P}_{1})\bm{\Omega}\} \notag \\
	&&\leq c_2 (\nu_{M_n}-\nu_1) \leq c_2 V (M_n-1),
\end{eqnarray}
where the last two inequalities are due to Assumptions 2 and 4, respectively. Combining \eqref{eq:M12}, \eqref{eq:M11}, $\lim_{n\to \infty} M_n/m_n^{**}=0$, and Theorem \ref{thm:mn}, we have
\begin{equation*}
	\limsup_{n\to \infty} \frac{\Delta_n}{R_n(m_n^*)}\leq \frac{c_2V}{c_1}\lim_{n\to \infty} \frac{M_n-1}{m_n^{**}}=0,
\end{equation*}
which yields that $\Delta_n=o\{R_n(m_n^*)\}$. This completes the proof of Theorem \ref{thm:delta2}.

\subsection{Proof of Theorem \ref{thm:delta}}\label{sec:thm2}
When Condition M2 holds, we consider two scenarios: $M_n\geq m_n^{**}$ and $\underline{c}\leq M_n/m_n^{**}<1$ for any sufficiently large $n$, to prove this theorem.

\underline{(i) $M_n\geq m_n^{**}$ for any sufficiently large $n$}. In this case, $m_n^*=m_n^{**}$ satisfies \eqref{eq:optm}. When Condition A1 holds, we first examine the order of the optimal risk of MS. Let $s_n^*=\max\{s: \lfloor k^s(m_n^*+1)\rfloor\leq d_n, s=0,1,\ldots\}$. The first term in \eqref{eq:riskms} is upper bounded by
\begin{eqnarray*}
	&&\bm{\mu}^{\top}(\mathbf{I}_n- \mathbf{P}_{m_n^*})\bm{\mu} \\
	&&\quad =\sum_{m=m_n^*+1}^{q_n} \bm{\mu}^{\top}( \mathbf{P}_{m}- \mathbf{P}_{m-1})\bm{\mu}=\sum_{m=m_n^*+1}^{d_n} \bm{\mu}^{\top}( \mathbf{P}_{m}- \mathbf{P}_{m-1})\bm{\mu} \\
	&&\quad =\sum_{s=0}^{s_n^*-1} \sum_{m=\lfloor k^s(m_n^*+1)\rfloor}^{\lfloor k^{s+1}(m_n^*+1)\rfloor-1} \bm{\mu}^{\top}( \mathbf{P}_{m}- \mathbf{P}_{m-1})\bm{\mu} +\sum_{m=\lfloor k^{s_n^*}(m_n^*+1)\rfloor}^{d_n} \bm{\mu}^{\top}( \mathbf{P}_{m}- \mathbf{P}_{m-1})\bm{\mu} \\
	&&\quad =\sum_{s=0}^{s_n^*-1} \sum_{m=\lfloor k^s(m_n^*+1)\rfloor}^{\lfloor k^{s+1}(m_n^*+1)\rfloor-1} n\mathrm{tr}\{( \mathbf{P}_m- \mathbf{P}_{m-1})\bm{\Omega}\} \theta_{n,m} +\sum_{m=\lfloor k^{s_n^*}(m_n^*+1)\rfloor}^{d_n} n\mathrm{tr}\{( \mathbf{P}_m- \mathbf{P}_{m-1})\bm{\Omega}\} \theta_{n,m} \\
	&&\quad \leq \sum_{s=0}^{s_n^*-1} \theta_{n,\lfloor k^s(m_n^*+1)\rfloor} \sum_{m=\lfloor k^s(m_n^*+1)\rfloor}^{\lfloor k^{s+1}(m_n^*+1)\rfloor-1} n\mathrm{tr}\{( \mathbf{P}_m- \mathbf{P}_{m-1})\bm{\Omega}\} \\ &&\qquad +\theta_{n,\lfloor k^{s_n^*}(m_n^*+1)\rfloor} \sum_{m=\lfloor k^{s_n^*}(m_n^*+1)\rfloor}^{d_n} n\mathrm{tr}\{( \mathbf{P}_m- \mathbf{P}_{m-1})\bm{\Omega}\} \\
	&&\quad \leq n\theta_{n,m_n^*+1} \sum_{s=0}^{s_n^*-1} \eta^s \mathrm{tr}\{( \mathbf{P}_{\lfloor k^{s+1}(m_n^*+1)\rfloor-1}- \mathbf{P}_{\lfloor k^s(m_n^*+1)\rfloor-1})\bm{\Omega}\} \\
	&&\qquad +n\theta_{n,m_n^*+1} \eta^{s_n^*} \mathrm{tr}\{( \mathbf{P}_{d_n}- \mathbf{P}_{\lfloor k^{s_n^*}(m_n^*+1)\rfloor-1})\bm{\Omega}\} \\
	&&\quad \leq c_2\sum_{s=0}^{s_n^*-1} \eta^s (\nu_{\lfloor k^{s+1}(m_n^*+1)\rfloor}-\nu_{\lfloor k^s(m_n^*+1)\rfloor})+c_2 \eta^{s_n^*}(\nu_{d_n}-\nu_{\lfloor k^{s_n^*}(m_n^*+1)\rfloor-1}) \\
	&&\quad \leq c_2 V \sum_{s=0}^{s_n^*} \eta^s (\lfloor k^{s+1}(m_n^*+1)\rfloor-\lfloor k^s(m_n^*+1)\rfloor) \\
	&&\quad \sim c_2 V (k-1)(m_n^*+1)\sum_{s=0}^{s_n^*} (k\eta)^s \asymp m_n^*\asymp \mathrm{tr}( \mathbf{P}_{m_n^*}\bm{\Omega}),
\end{eqnarray*}
where the first equality follows from the fact that $\bm{\mu}^{\top} ( \mathbf{I}_n- \mathbf{P}_{q_n})\bm{\mu}=0$, the first inequality follows from Assumption 3, the second inequality follows from $\theta_{n,\lfloor k^s(m_n^*+1)\rfloor}/\theta_{n,m_n^*+1}\leq \eta^s$ for a sufficiently large $n$, which can be obtained by Condition A1 and Theorem \ref{thm:mn}, and the last two inequalities follow from \eqref{eq:optm} and Assumption 4 respectively. Thus, the order of the optimal risk of MS satisfies $R_n(m_n^*)\asymp \mathrm{tr}( \mathbf{P}_{m_n^*}\bm{\Omega})$.

Next, we prove that the potential advantage $\Delta_n$ of MA over MS has the same order as $R_n(m_n^*)$ under Condition A1. Define $t_n=\min \{t\in \mathbb{N}: \lfloor kt\rfloor\geq m_n^*+1\}$. Then it follows from Theorem \ref{thm:mn} and \citet{peng2021} that $\lim_{n\to \infty} t_n=\infty$, $\lfloor kt_n\rfloor \sim m_n^*$, and $t_n\sim m_n^*/k$. The first term in \eqref{eq:delta} can be lower bounded by
\begin{eqnarray}\label{eq:condlb}
	&&\sum_{m=2}^{m_n^*} \biggl[\mathrm{tr}\{( \mathbf{P}_{m}- \mathbf{P}_{m-1})\bm{\Omega}\}- \frac{\mathrm{tr}\{( \mathbf{P}_m- \mathbf{P}_{m-1})\bm{\Omega}\} \bm{\mu}^{\top}( \mathbf{P}_m- \mathbf{P}_{m-1})\bm{\mu}}{\bm{\mu}^{\top}( \mathbf{P}_m- \mathbf{P}_{m-1})\bm{\mu}+\mathrm{tr}\{( \mathbf{P}_m- \mathbf{P}_{m-1})\bm{\Omega}\}}\biggr] \notag\\
	&&\quad \geq \mathrm{tr}\{( \mathbf{P}_{m_n^*}- \mathbf{P}_{1})\bm{\Omega}\}-\sum_{m=2}^{\lfloor kt_n\rfloor} \frac{\mathrm{tr}\{( \mathbf{P}_m- \mathbf{P}_{m-1})\bm{\Omega}\} \bm{\mu}^{\top}( \mathbf{P}_m- \mathbf{P}_{m-1})\bm{\mu}}{\bm{\mu}^{\top}( \mathbf{P}_m- \mathbf{P}_{m-1})\bm{\mu}+\mathrm{tr}\{( \mathbf{P}_m- \mathbf{P}_{m-1})\bm{\Omega}\}} \notag\\
	&&\quad \geq \mathrm{tr}\{( \mathbf{P}_{m_n^*}- \mathbf{P}_{1})\bm{\Omega}\}-\sum_{m=2}^{t_n} \mathrm{tr}\{( \mathbf{P}_{m}- \mathbf{P}_{m-1})\bm{\Omega}\}-\sum_{m=t_n+1}^{\lfloor kt_n\rfloor} \frac{\mathrm{tr}\{( \mathbf{P}_m- \mathbf{P}_{m-1})\bm{\Omega}\}}{1+1/(n\theta_{n,m})} \notag\\
	&&\quad \geq \mathrm{tr}\{( \mathbf{P}_{m_n^*}- \mathbf{P}_{t_n})\bm{\Omega}\}-\frac{1}{1+1/(n\theta_{n,t_n})} \mathrm{tr}\{( \mathbf{P}_{\lfloor kt_n\rfloor}- \mathbf{P}_{t_n})\bm{\Omega}\}, \notag \\
	&& \quad \geq \mathrm{tr}\{( \mathbf{P}_{m_n^*}- \mathbf{P}_{t_n})\bm{\Omega}\}-\frac{1}{1+\delta} \mathrm{tr}\{( \mathbf{P}_{\lfloor kt_n\rfloor}- \mathbf{P}_{t_n})\bm{\Omega}\} \notag \\
	&&\quad = \frac{1}{1+\delta}\mathrm{tr}\big[ \{(1+\delta) \mathbf{P}_{m_n^*}-  \mathbf{P}_{\lfloor kt_n\rfloor}-\delta  \mathbf{P}_{t_n}\}\bm{\Omega}\bigr]
\end{eqnarray}
where the third inequality follows from Assumption 3 and the last inequality follows from the following fact
\begin{equation*}
	\frac{1}{1+1/(n\theta_{n,t_n})}\leq \frac{1}{1+\delta/(n\theta_{n, \lfloor kt_n\rfloor})}\leq \frac{1}{1+\delta/(n\theta_{n, m_n^*+1})}\leq \frac{1}{1+\delta},
\end{equation*}
which can be derived by \eqref{eq:optm} and Condition A1. Since $\nu_{m_n^*}\sim \nu_{\lfloor kt_n\rfloor}$, it is easy to show that $(1+\delta) \mathbf{P}_{m_n^*}-  \mathbf{P}_{\lfloor kt_n\rfloor}-\delta  \mathbf{P}_{t_n}$ is positive semi-definite for sufficiently large $n$. By Assumption 2 and the fact that $\mathrm{tr}(\mathbf{A}\mathbf{B})\geq \lambda_{\min}(\mathbf{A})\mathrm{tr}(\mathbf{B})$ for symmetric matrix $\mathbf{A}$ and positive semi-definite matrix $\mathbf{B}$ \citep[Proposition 8.4.13]{matrixbook2005}, we have
\begin{eqnarray}\label{eq:cond1}
	&&\frac{1}{1+\delta}\mathrm{tr}\big[ \{(1+\delta) \mathbf{P}_{m_n^*}-  \mathbf{P}_{\lfloor kt_n\rfloor}-\delta  \mathbf{P}_{t_n}\}\bm{\Omega}\bigr] \notag \\
	&&\quad \geq \frac{c_1}{1+\delta} \{(1+\delta)\nu_{m_n^*}-\nu_{\lfloor kt_n\rfloor}-\delta \nu_{t_n}\} \notag \\
	&&\quad \geq \frac{c_1}{1+\delta} (\nu_{m_n^*}-\nu_{\lfloor kt_n\rfloor})+\frac{c_1\delta}{1+\delta}(m_n^*-t_n) \notag \\
	&&\quad \sim \frac{(k-1)c_1\delta}{k(1+\delta)}m_n^*\asymp \mathrm{tr}( \mathbf{P}_{m_n^*}\bm{\Omega}),
\end{eqnarray}
where the last line is due to $\nu_{m_n^*}\sim \nu_{\lfloor kt_n\rfloor}$ and $t_n\sim m_n^*/k$. From \eqref{eq:delta}, we see
\begin{equation*}
	R_n(m_n^*)\geq \Delta_n\geq \sum_{m=2}^{m_n^*}\biggl[\mathrm{tr}\{( \mathbf{P}_{m}- \mathbf{P}_{m-1})\bm{\Omega}\}- \frac{\mathrm{tr}\{( \mathbf{P}_m- \mathbf{P}_{m-1})\bm{\Omega}\} \bm{\mu}^{\top}( \mathbf{P}_m- \mathbf{P}_{m-1})\bm{\mu}}{\bm{\mu}^{\top}( \mathbf{P}_m- \mathbf{P}_{m-1})\bm{\mu}+\mathrm{tr}\{( \mathbf{P}_m- \mathbf{P}_{m-1})\bm{\Omega}\}}\biggr],
\end{equation*}
which, along with \eqref{eq:condlb} and \eqref{eq:cond1}, implies $\Delta_n\asymp R_n(m_n^*)$. This completes the proof of the result under Condition A1.

When Condition A2 holds, we next examine $\Delta_n=o\{R_n(m_n^*)\}$. Let $2/m_n^*<k'<1$. The first term in \eqref{eq:delta} is upper bounded by
\begin{eqnarray}
	&&\sum_{m=2}^{m_n^*}\biggl[\mathrm{tr}\{( \mathbf{P}_{m}- \mathbf{P}_{m-1})\bm{\Omega}\}- \frac{\mathrm{tr}\{( \mathbf{P}_m- \mathbf{P}_{m-1})\bm{\Omega}\} \bm{\mu}^{\top}( \mathbf{P}_m- \mathbf{P}_{m-1})\bm{\mu}}{\bm{\mu}^{\top}( \mathbf{P}_m- \mathbf{P}_{m-1})\bm{\mu}+\mathrm{tr}\{( \mathbf{P}_m- \mathbf{P}_{m-1})\bm{\Omega}\}}\biggr] \notag \\
	&&\quad \leq \mathrm{tr}\{( \mathbf{P}_{m_n^*}- \mathbf{P}_{1})\bm{\Omega}\}-\sum_{m=2}^{\lfloor k'm_n^*\rfloor} \frac{\mathrm{tr}\{( \mathbf{P}_m- \mathbf{P}_{m-1})\bm{\Omega}\}}{1+1/(n\theta_{n,m})} \notag \\
	&&\quad \leq \mathrm{tr}\{( \mathbf{P}_{m_n^*}- \mathbf{P}_1)\bm{\Omega}\}-\frac{1}{1+1/(n\theta_{n,\lfloor k'm_n^*\rfloor})} \sum_{m=2}^{\lfloor k'm_n^*\rfloor} \mathrm{tr}\{( \mathbf{P}_m- \mathbf{P}_{m-1})\bm{\Omega}\} \label{eq:cond21} \notag \\
	&&\quad = \mathrm{tr}\{( \mathbf{P}_{m_n^*}- \mathbf{P}_1)\bm{\Omega}\}-\frac{1}{1+1/(n\theta_{n,\lfloor k'm_n^*\rfloor})} \mathrm{tr}\{( \mathbf{P}_{\lfloor k'm_n^*\rfloor}- \mathbf{P}_{1})\bm{\Omega}\} \notag \\
	&&\quad =\mathrm{tr} \biggl\{\biggl( \mathbf{P}_{m_n^*}- \mathbf{P}_1-\frac{ \mathbf{P}_{\lfloor k'm_n^*\rfloor}- \mathbf{P}_{1}}{1+1/(n\theta_{n,\lfloor k'm_n^*\rfloor})}\biggr) \bm{\Omega}\biggr\}.
\end{eqnarray}
Observe that
\begin{equation*}
	 \mathbf{P}_{m_n^*}- \mathbf{P}_1-\frac{ \mathbf{P}_{\lfloor k'm_n^*\rfloor}- \mathbf{P}_{1}}{1+1/(n\theta_{n,\lfloor k'm_n^*\rfloor})}=\frac{ \mathbf{P}_{m_n^*}- \mathbf{P}_{\lfloor k'm_n^*\rfloor}+( \mathbf{P}_{m_n^*}- \mathbf{P}_1)/(n\theta_{n,\lfloor k'm_n^*\rfloor})}{1+1/(n\theta_{n,\lfloor k'm_n^*\rfloor})}
\end{equation*}
is a positive semi-definite matrix. By the fact that $\mathrm{tr}(\mathbf{A}\mathbf{B})\leq \lambda_{\max}(\mathbf{A})\mathrm{tr}(\mathbf{B})$ for symmetric matrix $\mathbf{A}$ and positive semi-definite matrix $\mathbf{B}$ \citep[Proposition 8.4.13]{matrixbook2005}, we have
\begin{eqnarray}
	&& \mathrm{tr} \biggl\{\biggl( \mathbf{P}_{m_n^*}- \mathbf{P}_1-\frac{ \mathbf{P}_{\lfloor k'm_n^*\rfloor}- \mathbf{P}_{1}}{1+1/(n\theta_{n,\lfloor k'm_n^*\rfloor})}\biggr) \bm{\Omega}\biggr\} \notag \\
	&&\quad \leq \frac{c_2}{1+1/(n\theta_{n,\lfloor k'm_n^*\rfloor})} \mathrm{tr} \biggl( \mathbf{P}_{m_n^*}- \mathbf{P}_{\lfloor k'm_n^*\rfloor}+\frac{ \mathbf{P}_{m_n^*}- \mathbf{P}_1}{n\theta_{n,\lfloor k'm_n^*\rfloor}}\biggr) \notag \\
	&&\quad = \frac{c_2}{1+1/(n\theta_{n,\lfloor k'm_n^*\rfloor})} \biggl(\nu_{m_n^*}-\nu_{\lfloor k'm_n^*\rfloor}+\frac{\nu_{m_n^*}-\nu_1}{n\theta_{n,\lfloor k'm_n^*\rfloor}}\biggr) \label{eq:cond22} \notag \\
	&&\quad \leq \frac{c_2 V}{1+1/(n\theta_{n,\lfloor k'm_n^*\rfloor})} \biggl(m_n^*-\lfloor k'm_n^*\rfloor+\frac{m_n^*-1}{n\theta_{n,\lfloor k'm_n^*\rfloor}}\biggr) \notag \\
	&&\quad \leq c_2 V \biggl\{m_n^*-\lfloor k'm_n^*\rfloor+\frac{\theta_{n,m_n^*}}{\theta_{n,\lfloor k'm_n^*\rfloor}}(m_n^*-1)\biggr\},
\end{eqnarray}
where the second inequality follows from Assumption 4. Since $\lim_{n\to \infty} \theta_{n,m_n^*}/\theta_{n,\lfloor k'm_n^*\rfloor}=0$ for any $k'<1$ under Condition A2 and Theorem \ref{thm:mn}, we have 
\begin{equation*}
	\biggl\{m_n^*-\lfloor k'm_n^*\rfloor+\frac{\theta_{n,m_n^*}}{\theta_{n,\lfloor k'm_n^*\rfloor}}(m_n^*-1)\biggr\}\Big/m_n^*=1-\frac{\lfloor k'm_n^*\rfloor}{m_n^*}+\frac{\theta_{n,m_n^*}}{\theta_{n,\lfloor k'm_n^*\rfloor}}\left(1-\frac{1}{m_n^*}\right)\to 1-k',
\end{equation*}
which, along with \eqref{eq:cond21} and \eqref{eq:cond22}, yields that
\begin{equation*}
	\sum_{m=2}^{m_n^*}\biggl[\mathrm{tr}\{( \mathbf{P}_{m}- \mathbf{P}_{m-1})\bm{\Omega}\}- \frac{\mathrm{tr}\{( \mathbf{P}_m- \mathbf{P}_{m-1})\bm{\Omega}\} \bm{\mu}^{\top}( \mathbf{P}_m- \mathbf{P}_{m-1})\bm{\mu}}{\bm{\mu}^{\top}( \mathbf{P}_m- \mathbf{P}_{m-1})\bm{\mu}+\mathrm{tr}\{( \mathbf{P}_m- \mathbf{P}_{m-1})\bm{\Omega}\}}\biggr]=O\{(1-k')m_n^*\}.
\end{equation*}
Due to the arbitrariness of $k'$ and the fact $\mathrm{tr}( \mathbf{P}_{m_n^*}\bm{\Omega})\asymp m_n^*$, letting $k'\to 1$, we can obtain the first term of \eqref{eq:delta}:
\begin{equation}\label{eq:fterm}
	\sum_{m=2}^{m_n^*}\biggl[\mathrm{tr}\{( \mathbf{P}_{m}- \mathbf{P}_{m-1})\bm{\Omega}\}- \frac{\mathrm{tr}\{( \mathbf{P}_m- \mathbf{P}_{m-1})\bm{\Omega}\} \bm{\mu}^{\top}( \mathbf{P}_m- \mathbf{P}_{m-1})\bm{\mu}}{\bm{\mu}^{\top}( \mathbf{P}_m- \mathbf{P}_{m-1})\bm{\mu}+\mathrm{tr}\{( \mathbf{P}_m- \mathbf{P}_{m-1})\bm{\Omega}\}}\biggr]=o\{\mathrm{tr}( \mathbf{P}_{m_n^*}\bm{\Omega})\}.
\end{equation}

Next, we consider the order of the second term of \eqref{eq:delta}. Choose $k>1$. We have
\begin{equation}\label{eq:tail}
	\begin{split}
	& \sum_{m=m_n^*+1}^{M_n} \frac{ \{\bm{\mu}^{\top}( \mathbf{P}_m- \mathbf{P}_{m-1})\bm{\mu}\}^2}{\bm{\mu}^{\top}( \mathbf{P}_m- \mathbf{P}_{m-1})\bm{\mu}+\mathrm{tr}\{( \mathbf{P}_m- \mathbf{P}_{m-1})\bm{\Omega}\}} \\
	&\quad = \sum_{m=m_n^*+1}^{\lfloor k(m_n^*+1)\rfloor} \frac{n\theta_{n,m}}{1+1/(n\theta_{n,m})} \mathrm{tr}\{( \mathbf{P}_m- \mathbf{P}_{m-1})\bm{\Omega}\}+\sum_{m=\lfloor k(m_n^*+1)\rfloor+1}^{\min\{M_n,d_n\}} \frac{\bm{\mu}^{\top}( \mathbf{P}_m- \mathbf{P}_{m-1})\bm{\mu}}{1+1/(n\theta_{n,m})}.
	\end{split}
\end{equation}
The first term of \eqref{eq:tail} is upper bounded by
\begin{eqnarray*}
	&& \sum_{m=m_n^*+1}^{\lfloor k(m_n^*+1)\rfloor} \frac{n\theta_{n,m}}{1+1/(n\theta_{n,m})} \mathrm{tr}\{( \mathbf{P}_m- \mathbf{P}_{m-1})\bm{\Omega}\} \\
	&&\quad \leq \frac{n\theta_{n,m_n*+1}}{1+1/(n\theta_{n,m_n^*+1})}  \sum_{m=m_n^*+1}^{\lfloor k(m_n^*+1)\rfloor} \mathrm{tr}\{( \mathbf{P}_m- \mathbf{P}_{m-1})\bm{\Omega}\} \\
	&&\quad \leq \frac{1}{2} \mathrm{tr}\{( \mathbf{P}_{\lfloor k(m_n^*+1)\rfloor}- \mathbf{P}_{m_n^*})\bm{\Omega}\} \\
	&&\quad \leq \frac{c_2}{2} (\nu_{\lfloor k(m_n^*+1)\rfloor}-\nu_{m_n^*}) \\
	&&\quad \leq \frac{c_2}{2} V (\lfloor k(m_n^*+1)\rfloor-m_n^*),
\end{eqnarray*}
where the first two inequalities follow from Assumption 3 and \eqref{eq:optm}, respectively, and the last inequality follows from Assumption 4. Using Theorem \ref{thm:mn}, as $n\to \infty$, 
\begin{equation*}
	\frac{\lfloor k(m_n^*+1)\rfloor-m_n^*}{m_n^*}=\frac{\lfloor k(m_n^*+1)\rfloor}{m_n^*}-1\to k-1.
\end{equation*}
Therefore, 
\begin{equation}\label{eq:cond23}
	\sum_{m=m_n^*+1}^{\lfloor k(m_n^*+1)\rfloor} \frac{n\theta_{n,m}}{1+1/(n\theta_{n,m})} \mathrm{tr}\{( \mathbf{P}_m- \mathbf{P}_{m-1})\bm{\Omega}\}=O\{(k-1)m_n^*\}=O\{(k-1)\mathrm{tr}( \mathbf{P}_{m_n^*}\bm{\Omega})\}.
\end{equation}
The second term of \eqref{eq:tail} can be upper bounded by
\begin{eqnarray}
	&& \sum_{m=\lfloor k(m_n^*+1)\rfloor+1}^{\min\{M_n,d_n\}} \frac{\bm{\mu}^{\top}( \mathbf{P}_m- \mathbf{P}_{m-1})\bm{\mu}}{1+1/(n\theta_{n,m})} \notag \\
	&&\quad \leq \frac{1}{1+1/(n\theta_{n,\lfloor k(m_n^*+1)\rfloor})} \sum_{m=\lfloor k(m_n^*+1)\rfloor+1}^{\min\{M_n,d_n\}} \bm{\mu}^{\top}( \mathbf{P}_m- \mathbf{P}_{m-1})\bm{\mu} \notag \\
	&&\quad \leq \frac{1}{1+\theta_{n, m_n^*+1}/\theta_{n,\lfloor k(m_n^*+1)\rfloor}} \bm{\mu}^{\top}( \mathbf{P}_{\min\{M_n,d_n\}}- \mathbf{P}_{\lfloor k(m_n^*+1)\rfloor})\bm{\mu} \notag \\
	&&\quad \leq \frac{1}{1+\theta_{n, m_n^*+1}/\theta_{n,\lfloor k(m_n^*+1)\rfloor}} \bm{\mu}^{\top}( \mathbf{I}_n- \mathbf{P}_{m_n^*})\bm{\mu} \notag \\
	&&\quad =o\{\bm{\mu}^{\top}( \mathbf{I}_n- \mathbf{P}_{m_n^*})\bm{\mu}\},\label{eq:cond24}
\end{eqnarray}
where the first two inequalities follow from Assumption 3 and \eqref{eq:optm}, respectively, the last inequality follows from the following fact
\begin{eqnarray*}
	\bm{\mu}^{\top}( \mathbf{I}_n- \mathbf{P}_{m_n^*})\bm{\mu}=\hspace{-6mm}&&\bm{\mu}^{\top}( \mathbf{I}_n- \mathbf{P}_{\min\{M_n,d_n\}})\bm{\mu}+\bm{\mu}^{\top}( \mathbf{P}_{\min\{M_n,d_n\}}- \mathbf{P}_{\lfloor k(m_n^*+1)\rfloor})\bm{\mu} \\
	 &&+\bm{\mu}^{\top}( \mathbf{P}_{\lfloor k(m_n^*+1)\rfloor}- \mathbf{P}_{m_n^*})\bm{\mu} \\
	 \geq\hspace{-6mm}&& \bm{\mu}^{\top}( \mathbf{P}_{\min\{M_n,d_n\}}- \mathbf{P}_{\lfloor k(m_n^*+1)\rfloor})\bm{\mu},
\end{eqnarray*}
and the last equality follows from the fact that $\lim_{n\to \infty} \theta_{n,\lfloor k(m_n^*+1)\rfloor}/\theta_{n, m_n^*+1}=0$ for any $k>1$ under Condition A2. Combining \eqref{eq:tail}, \eqref{eq:cond23}, and \eqref{eq:cond24}, we have
\begin{equation*}
	\sum_{m=m_n^*+1}^{M_n} \frac{ \{\bm{\mu}^{\top}( \mathbf{P}_m- \mathbf{P}_{m-1})\bm{\mu}\}^2}{\bm{\mu}^{\top}( \mathbf{P}_m- \mathbf{P}_{m-1})\bm{\mu}+\mathrm{tr}\{( \mathbf{P}_m- \mathbf{P}_{m-1})\bm{\Omega}\}}=O\{(k-1)\mathrm{tr}( \mathbf{P}_{m_n^*}\bm{\Omega})\}+o\{\bm{\mu}^{\top}( \mathbf{I}_n- \mathbf{P}_{m_n^*})\bm{\mu}\}.
\end{equation*}
Duo to the arbitrariness of $k$, letting $k\to 1$, we have
\begin{equation*}
	\sum_{m=m_n^*+1}^{M_n} \frac{ \{\bm{\mu}^{\top}( \mathbf{P}_m- \mathbf{P}_{m-1})\bm{\mu}\}^2}{\bm{\mu}^{\top}( \mathbf{P}_m- \mathbf{P}_{m-1})\bm{\mu}+\mathrm{tr}\{( \mathbf{P}_m- \mathbf{P}_{m-1})\bm{\Omega}\}}=o\{R_n(m_n^*)\},
\end{equation*}
which, along with \eqref{eq:delta} and \eqref{eq:fterm}, leads to $\Delta_n=o\{R_n(m_n^*)\}$. This completes the proof of the result under Condition A2.

\underline{(ii) $\underline{c}\leq M_n/m_n^{**}<1$ for any sufficiently large $n$}. In this case, $m_n^*=M_n \asymp m_n^{**}$. When Condition A1 holds, there exists a finite positive integer $\tau_1$ such that $k^{-\tau_1}\leq \underline{c}$. Therefore,
\begin{equation}\label{eq:case2}
	\theta_{n,m_n^*+1}=\theta_{n,M_n+1}\leq \theta_{n,\lfloor\underline{c}m_n^{**}\rfloor+1}\leq \theta_{n,\lfloor k^{-\tau_1}(m_n^{**}+1)\rfloor}\leq \delta^{-\tau_1}\theta_{n,m_n^{**}+1}\leq \delta^{-\tau_1}\frac{1}{n},
\end{equation}
where the last inequality is due to \eqref{eq:optm2}. Then, using the same arguments in (i) and \eqref{eq:case2}, it is easy to prove the result under Condition A1. When Condition A2 holds, we can also obtain \eqref{eq:fterm}, which, along with the fact that the second term of \eqref{eq:delta} equals 0, yields the result under Condition A2. This completes the proof of Theorem \ref{thm:delta} under (ii).

\subsection{Proof of Lemma \ref{lem:cond}}

From Assumption 7, we know that for any small $0<\epsilon<1$, there exists a constant $K_{\epsilon}>0$ which does not depend on $m$, such that $0<1-\epsilon\leq \theta_{n,m}/\theta_m^*\leq 1+\epsilon$ holds uniformly in $m=1,\ldots,d_n$ and $n\geq K_{\epsilon}$.

(i) When Condition B1 holds, there exist constants $k>1$ and $0<\delta^* \leq \eta^*<1$ with $k\eta^*<1$ such that for a sufficiently large $n$,
\begin{equation*}
	\frac{1-\epsilon}{1+\epsilon}\delta^* \leq \frac{\theta_{n, \lfloor k l_n\rfloor}}{\theta_{n, l_n}}=\frac{\theta_{n, \lfloor k l_n\rfloor}}{\theta_{\lfloor k l_n\rfloor}^*} \times \frac{\theta_{\lfloor k l_n\rfloor}^*}{\theta_{l_n}^*}\times \frac{\theta_{l_n}^*}{\theta_{n, l_n}}\leq \frac{1+\epsilon}{1-\epsilon}\eta^*.
\end{equation*}
Let $\delta=\frac{1-\epsilon}{1+\epsilon}\delta^*$ and $\eta=\frac{1+\epsilon}{1-\epsilon}\eta^*$. Since $\lim_{\epsilon\to 0} \frac{1+\epsilon}{1-\epsilon}=1$, we can choose a small enough $\epsilon>0$ such that $0<\delta\leq \eta<1$ and $k\eta<1$. Therefore, Condition B1 implies Condition A1.

(ii) When Condition B2 holds, for every constant $k>1$ and every integer sequence $\{l_n\}$ satisfied $\lim_{n\to \infty} l_n=\infty$,
\begin{equation*}
	\lim_{n\to \infty} \frac{\theta_{n, \lfloor k l_n\rfloor}}{\theta_{n, l_n}}=\lim_{n\to \infty} \biggl\{\frac{\theta_{n, \lfloor k l_n\rfloor}}{\theta_{\lfloor k l_n\rfloor}^*} \times \frac{\theta_{\lfloor k l_n\rfloor}^*}{\theta_{l_n}^*}\times \frac{\theta_{l_n}^*}{\theta_{n, l_n}} \biggr\}\leq \frac{1+\epsilon}{1-\epsilon} \lim_{n\to \infty} \frac{\theta_{\lfloor k l_n\rfloor}^*}{\theta_{l_n}^*}=0.
\end{equation*}
Therefore, Condition B2 implies Condition A2.

\subsection{Proof of Corollary \ref{cor:main}}\label{subsec:appcor}

From Theorem \ref{thm:mw}, $\frac{1}{2}\leq R_n(\bm{w}_n^*)/R_n(m_n^*)\leq 1$ for any sufficiently large $n$. Since $R_n(\hat{m})/R_n(m_n^*)=1+o_p(1)$ and $R_n(\hat{\bm{w}})/R_n(\bm{w}_n^*)=1+o_p(1)$, then when $n$ is large enough,
\begin{equation*}
	\frac{1}{2}\{1+o_p(1)\}\leq \frac{R_n(\hat{\bm{w}})}{R_n(\hat{m})}=\frac{R_n(\hat{\bm{w}})}{R_n(\bm{w}_n^*)}\frac{R_n(\bm{w}_n^*)}{R_n(m_n^*)}\frac{R_n(m_n^*)}{R_n(\hat{m})}\leq 1+o_p(1),
\end{equation*}
which yields that $R_n(\hat{\bm{w}})\asymp_p R_n(\hat{m})$. Observe that
\begin{equation}\label{eq:corol}
	\frac{R_n(\hat{m})-R_n(\hat{\bm{w}})}{R_n(\hat{m})}=1-\frac{R_n(\hat{\bm{w}})}{R_n(\bm{w}_n^*)}\frac{R_n(m_n^*)}{R_n(\hat{m})}+\frac{R_n(\hat{\bm{w}})}{R_n(\bm{w}_n^*)}\frac{\Delta_n}{R_n(m_n^*)}\frac{R_n(m_n^*)}{R_n(\hat{m})}.
\end{equation}
Under Conditions M2 and A1, from Theorem \ref{thm:delta}, $\Delta_n/R_n(m_n^*)\geq c^*$ for some $c^*\in (0,1/2]$ and any sufficiently large $n$. Therefore, when $n$ is large enough,
\begin{eqnarray*}
	1 \geq \left|\frac{R_n(\hat{m})-R_n(\hat{\bm{w}})}{R_n(\hat{m})}\right| \hspace{-6mm}&&\geq \frac{R_n(\hat{\bm{w}})}{R_n(\bm{w}_n^*)}\frac{\Delta_n}{R_n(m_n^*)}\frac{R_n(m_n^*)}{R_n(\hat{m})}-\left|1-\frac{R_n(\hat{\bm{w}})}{R_n(\bm{w}_n^*)}\frac{R_n(m_n^*)}{R_n(\hat{m})}\right| \\
	&&\geq c^*\{1+o_p(1)\}-o_p(1)=c^*\{1+o_p(1)\},
\end{eqnarray*}
which leads to $R_n(\hat{m})-R_n(\hat{\bm{w}})\asymp_p R_n(\hat{m})$. Under Condition M1 or Conditions M2 and A2, $\lim_{n\to \infty} \Delta_n/R_n(m_n^*)=0$ from Theorems \ref{thm:delta2} and \ref{thm:delta}. Therefore, by \eqref{eq:corol}, we have
\begin{equation*}
	\frac{R_n(\hat{m})-R_n(\hat{\bm{w}})}{R_n(\hat{m})}\xrightarrow{p} 0,
\end{equation*}
which implies that $R_n(\hat{m})-R_n(\hat{\bm{w}})=o_p\{R_n(\hat{m})\}$ or $R_n(\hat{m})\sim_p R_n(\hat{\bm{w}})$. This completes the proof of Corollary \ref{cor:main}.

\subsection{Proof of Theorem \ref{thm:relax}}

From \eqref{eq:riskRw} and Assumption 3, it is easy to see that the risk of the optimal MA estimator without the total weight constraint is
\begin{equation*}
	R_n(\tilde{\bm{w}}_n^*)=\sum_{m=1}^{M_n} \frac{\mathrm{tr}\{(\mathbf{P}_m-\mathbf{P}_{m-1})\bm{\Omega}\} \bm{\mu}^{\top}(\mathbf{P}_m-\mathbf{P}_{m-1})\bm{\mu}}{\bm{\mu}^{\top}(\mathbf{P}_m-\mathbf{P}_{m-1})\bm{\mu}+\mathrm{tr}\{(\mathbf{P}_m-\mathbf{P}_{m-1})\bm{\Omega}\}}+\bm{\mu}^{\top}(\mathbf{I}_n-\mathbf{P}_{M_n})\bm{\mu},
\end{equation*}
which, along with \eqref{eq:riskma} and Assumption 2, yields that
\begin{equation*}
	R_n(\bm{w}_n^*)-R_n(\tilde{\bm{w}}_n^*)=\frac{\{\mathrm{tr}(\mathbf{P}_1\bm{\Omega})\}^2}{\bm{\mu}^{\top}\mathbf{P}_1\bm{\mu}+\mathrm{tr}(\mathbf{P}_1\bm{\Omega})}\leq \mathrm{tr}(\mathbf{P}_1\bm{\Omega})< c_2\nu_1.
\end{equation*}
Furthermore, if Assumptions 4--6 hold, we have $R_n(\bm{w}_n^*)\to \infty$ from Theorem \ref{thm:mn}(ii). Therefore, $R_n(\bm{w}_n^*)- R_n(\tilde{\bm{w}}_n^*)=o\{R_n(\bm{w}_n^*)\}$, which completes the proof of Theorem \ref{thm:relax}.

\subsection{Two Lemmas and Their Proofs}

Before giving the proof of Theorems \ref{thm:discrete}, we need some lemmas. Let $\lceil a\rceil$ denote the least integer greater than or equal to $a$. We first present the following lemma on an expression of $R_n(\bm{w}_{n,N}^*)$.

\begin{lemma}\label{lem:RwnN}
	Suppose that Assumptions 3 and 6 hold. For any sufficiently large $n$, the optimal risk of MA restricted to $\mathcal{W}_n(N)$ is
	\begin{eqnarray*}
		R_n(\bm{w}_{n,N}^*)=\hspace{-6mm}&&\mathrm{tr}(\mathbf{P}_1\mathbf{\Omega})+\bm{\mu}^{\top}(\mathbf{I}_n-\mathbf{P}_{M_n})\bm{\mu} \\
		&&+\sum_{i=i_{n,N}+1}^{N}\sum_{m=m_n(\frac{2i+1}{2N})+1}^{m_n(\frac{2i-1}{2N})} \Biggl[\left(\frac{i}{N}\right)^2\mathrm{tr}\{(\mathbf{P}_m-\mathbf{P}_{m-1})\bm{\Omega}\} +\left(1-\frac{i}{N}\right)^2\bm{\mu}^{\top}(\mathbf{P}_m-\mathbf{P}_{m-1})\bm{\mu} \Bigg] \\
		&&+\sum_{m=m_n(\frac{2i_{n,N}+1}{2N})+1}^{M_n} \Biggl[\left(\frac{i_{n,N}}{N}\right)^2\mathrm{tr}\{(\mathbf{P}_m-\mathbf{P}_{m-1})\bm{\Omega}\}+\left(1-\frac{i_{n,N}}{N}\right)^2\bm{\mu}^{\top}(\mathbf{P}_m-\mathbf{P}_{m-1})\bm{\mu} \Bigg],
	\end{eqnarray*}
	where $i_{n,N}=\left\lceil N\gamma_{n,M_n}^*-\frac{1}{2}\right\rceil$ and $m_n(z)$
	for $z\in (\gamma_{n,q_n}^*,1)$ is an integer in $\{1,\ldots,q_n\}$ satisfying
	\begin{equation}\label{eq:mndef}
		\theta_{n,m_n(z)}> \frac{z}{(1-z)n}\geq \theta_{n,m_n(z)+1},
	\end{equation}
	and $m_n(z_0)=1$ for any $z_0\geq 1$.
\end{lemma}

\begin{proof}
	Since $\bm{w}\in \mathcal{W}_n(N)$, we have $\gamma_m=\sum_{j=m}^{M_n} w_j \in \{0,1/N,2/N,\ldots,1\}$. Observe that
	\begin{eqnarray*}
		f_m(\gamma_m)\equiv\hspace{-6mm}&& \gamma_m^2\Bigl[ \bm{\mu}^{\top}( \mathbf{P}_m- \mathbf{P}_{m-1})\bm{\mu}+\mathrm{tr}\{( \mathbf{P}_m- \mathbf{P}_{m-1})\bm{\Omega}\}\Bigr] -2\gamma_m \bm{\mu}^{\top}( \mathbf{P}_m- \mathbf{P}_{m-1})\bm{\mu}+\bm{\mu}^{\top}( \mathbf{P}_m- \mathbf{P}_{m-1})\bm{\mu} \\
		=\hspace{-6mm}&&\Bigl[ \bm{\mu}^{\top}( \mathbf{P}_m- \mathbf{P}_{m-1})\bm{\mu}+\mathrm{tr}\{( \mathbf{P}_m- \mathbf{P}_{m-1})\bm{\Omega}\}\Bigr](\gamma_m-\gamma_{n,m}^*)^2+\gamma_{n,m}^* \mathrm{tr}\{( \mathbf{P}_m- \mathbf{P}_{m-1})\bm{\Omega}\},
	\end{eqnarray*}
	where $\gamma_{n,m}^*$ is defined in \eqref{eq:gamma}. Since $\{\gamma_{n,m}^*\}_{m=1}^{M_n}$ is non-increasing, it is easy to see that
	\begin{equation*}
		\min_{\gamma_m\in \{0,1/N,2/N,\ldots,1\}} f_m(\gamma_m)=f_m\left(\frac{i}{N}\right),\quad \text{when~}\frac{2i-1}{2N}<\gamma_{n,m}^*\leq \frac{2i+1}{2N},\quad i=0,\ldots,N.
	\end{equation*}
	Therefore, from \eqref{eq:riskRw}, we have
	\begin{eqnarray}\label{eq:diff}
		R_n(\bm{w}_{n,N}^*)\hspace{-6mm}&&=\mathrm{tr}( \mathbf{P}_1\bm{\Omega})+\sum_{m=2}^{M_n} \min_{\gamma_m\in \{0,1/N,2/N,\ldots,1\}} f_m(\gamma_m)+\bm{\mu}^{\top}( \mathbf{I}_n- \mathbf{P}_{M_n})\bm{\mu} \nonumber\\
		&&=\mathrm{tr}( \mathbf{P}_1\bm{\Omega})+\sum_{m=2}^{M_n}\sum_{i=0}^N  f_m\left(\frac{i}{N}\right) \mathbf{1}\biggl\{\frac{2i-1}{2N}<\gamma_{n,m}^*\leq \frac{2i+1}{2N}\biggr\}+\bm{\mu}^{\top}( \mathbf{I}_n- \mathbf{P}_{M_n})\bm{\mu} \nonumber\\
		&&=\mathrm{tr}( \mathbf{P}_1\bm{\Omega})+\sum_{m=2}^{M_n}\sum_{i=0}^N \left[\left(\frac{i}{N}\right)^2 \mathrm{tr}\{( \mathbf{P}_m- \mathbf{P}_{m-1})\bm{\Omega}\}+\left(1-\frac{i}{N}\right)^2 \bm{\mu}^{\top}( \mathbf{P}_m- \mathbf{P}_{m-1})\bm{\mu}\right] \nonumber \\
		&&\quad \times \mathbf{1}\biggl\{\frac{2i-1}{2N}<\gamma_{n,m}^*\leq \frac{2i+1}{2N}\biggr\}+\bm{\mu}^{\top}( \mathbf{I}_n- \mathbf{P}_{M_n})\bm{\mu},
	\end{eqnarray}
	where $\mathbf{1}\{\cdot\}$ denotes the usual indicator function. By the definition of $m_n(z)$ in \eqref{eq:mndef}, we have $\frac{2i-1}{2N}<\gamma_{n,m}^*\leq \frac{2i+1}{2N}$ if and only if $m_n(\frac{2i+1}{2N})+1\leq m\leq m_n(\frac{2i-1}{2N})$ for $i=i_{n,N}+1,\ldots,N$ and $\frac{2i_{n,N}-1}{2N}<\gamma_{n,m}^*\leq \frac{2i_{n,N}+1}{2N}$ if and only if $m_n(\frac{2i_{n,N}+1}{2N})+1\leq m\leq M_n$, where
	$$
	i_{n,N}=\min\left\{i=0,1,\ldots,N: \gamma_{n,M_n}^*\leq \frac{2i+1}{2N}\right\}=\left\lceil N\gamma_{n,M_n}^*-\frac{1}{2}\right\rceil
	.$$
	Combining the above fact with \eqref{eq:diff}, it is easy to obtain the expression of $R_n(\bm{w}_{n,N}^*)$ in Lemma \ref{lem:RwnN}. Moreover, we can obtain another expression of $R_n(\bm{w}_{n,N}^*)$ as follows:
	\begin{eqnarray}\label{eq:newRwn}
		R_n(\bm{w}_{n,N}^*)\hspace{-6mm}&&=R_n(\bm{w}_n^*)+\sum_{m=2}^{M_n} \biggl(\Bigl[ \bm{\mu}^{\top}( \mathbf{P}_m- \mathbf{P}_{m-1})\bm{\mu}+\mathrm{tr}\{( \mathbf{P}_m- \mathbf{P}_{m-1})\bm{\Omega}\}\Bigr] \nonumber\\
		&& \quad \times \sum_{i=0}^N \biggl(\frac{i}{N}-\gamma_{n,m}^*\biggr)^2 \mathbf{1}\biggl\{\frac{2i-1}{2N}<\gamma_{n,m}^*\leq \frac{2i+1}{2N}\biggr\}\biggr).
	\end{eqnarray}
This completes the proof of Lemma \ref{lem:RwnN}.
\end{proof}

Note that $m_n(1/2)=m_n^{**}$. Next, we present some elementary properties of $m_n(z)$ in the following lemma.

\begin{lemma}\label{lem:mxi}
	Suppose that Assumptions 3 and 5 hold. Then, $m_n(z)$ for $z\in (\gamma_{n,q_n}^*,1)$ defined in Lemma \ref{lem:RwnN} satisfies the following properties.
	\begin{itemize}
		\item[(i)] $m_n(z)$ is a non-increasing function in $z$; $\lim_{n\to \infty} m_n(z)=\infty$ for any fixed $z\in (\gamma_{n,q_n}^*,1)$.
		\item[(ii)] If there exist constants $k>1$, $\eta<1$, and $K>1$ such that $\theta_{n, \lfloor k l_n\rfloor}/\theta_{n, l_n}\leq \eta$ for any $n\geq K$ and any integer sequence $\{l_n\}$ satisfied $\lim_{n\to \infty} l_n=\infty$, then $m_n(z_1)\asymp m_n(z_2)$ for any $\gamma_{n,q_n}^*<z_1\neq z_2< 1$.
	\end{itemize}
\end{lemma}

\begin{proof}
	The results of (i) are easily shown by Assumption 1 and the similar arguments in the proof of Lemma \ref{thm:mn}. Next, we shall prove (ii). Without loss of generality, we assume $z_1<z_2$, which follows that $m_n(z_1)\geq m_n(z_2)$. Observe there exists an integer $s>0$ such that $\frac{z_1}{1-z_1}\geq \frac{z_2}{1-z_2} \eta^s$. Then, by the defination of $m_n(\xi)$, we have
	\begin{equation}\label{eq:lema1}
		\theta_{n,m_n(z_1)}> \frac{z_1}{(1-z_1)n}\geq \frac{z_2}{(1-z_2)n} \eta^s \geq \eta^s \theta_{n,m_n(z_2)+1}\geq \theta_{n,\lfloor k^s(m_n(z_2)+1)\rfloor}.
	\end{equation}
	Thus, $m_n(z_1)< \lfloor k^s(m_n(z_2)+1)\rfloor$, which, along with $m_n(z_1)\geq m_n(z_2)$, yields that $m_n(z_1)\asymp m_n(z_2)$. This completes the proof of Lemma \ref{lem:mxi}.
\end{proof}

\subsection{Proof of Theorem \ref{thm:discrete}}\label{sec:thm3}

Observe that
\begin{equation}\label{eq:thm3id}
	\bm{\mu}^{\top}( \mathbf{P}_m- \mathbf{P}_{m-1})\bm{\mu}+\mathrm{tr}\{( \mathbf{P}_m- \mathbf{P}_{m-1})\bm{\Omega}\}=\frac{\bm{\mu}^{\top}( \mathbf{P}_m- \mathbf{P}_{m-1})\bm{\mu}}{\gamma_{n,m}^*}=\frac{\mathrm{tr}\{( \mathbf{P}_m- \mathbf{P}_{m-1})\bm{\Omega}\}}{1-\gamma_{n,m}^*},
\end{equation}
which, along with \eqref{eq:newRwn}, yields that
\begin{eqnarray*}
	&&R_n(\bm{w}_{n,N}^*)-R_n(\bm{w}_n^*) \\
	&&\quad =\sum_{m=2}^{m_n^*} \mathrm{tr}\{( \mathbf{P}_m- \mathbf{P}_{m-1})\bm{\Omega}\} \sum_{i=0}^N \frac{\left(\frac{i}{N}-\gamma_{n,m}^*\right)^2}{1-\gamma_{n,m}^*} \mathbf{1}\biggl\{\frac{2i-1}{2N}<\gamma_{n,m}^*\leq \frac{2i+1}{2N}\biggr\} \\
	&&\qquad+\sum_{m=m_n^*+1}^{M_n} \bm{\mu}^{\top}( \mathbf{P}_m- \mathbf{P}_{m-1})\bm{\mu} \sum_{i=0}^N \frac{\left(\frac{i}{N}-\gamma_{n,m}^*\right)^2}{\gamma_{n,m}^*} \mathbf{1}\biggl\{\frac{2i-1}{2N}<\gamma_{n,m}^*\leq \frac{2i+1}{2N}\biggr\} \\
	&&\quad\leq \frac{1}{2N} \sum_{m=2}^{m_n^*} \mathrm{tr}\{( \mathbf{P}_m- \mathbf{P}_{m-1})\bm{\Omega}\}+\frac{1}{2N} \sum_{m=m_n^*+1}^{M_n} \bm{\mu}^{\top}( \mathbf{P}_m- \mathbf{P}_{m-1})\bm{\mu} \\
	&&\quad = \frac{1}{2N} \Bigl[\mathrm{tr}\{( \mathbf{P}_{m_n^*}- \mathbf{P}_{1})\bm{\Omega}\}+\bm{\mu}^{\top}( \mathbf{P}_{M_n}- \mathbf{P}_{m_n^*})\bm{\mu}\Bigr] \\
	&&\quad \leq  \frac{1}{2N} R_n(m_n^*),
\end{eqnarray*}
where the first inequality is derived by the fact that when $\frac{2i-1}{2N}<\gamma_{n,m}^*\leq \frac{2i+1}{2N}$
for $i=0,\ldots,N$, 
\begin{equation*}
	\frac{\left(\frac{i}{N}-\gamma_{n,m}^*\right)^2}{1-\gamma_{n,m}^*}\leq \frac{1}{2N}\quad\text{and}\quad  \frac{\left(\frac{i}{N}-\gamma_{n,m}^*\right)^2}{\gamma_{n,m}^*}\leq \frac{1}{2N},
\end{equation*}
which can be easily verified. Therefore, $R_n(\bm{w}_{n,N}^*)- R_n(\bm{w}_n^*)\leq \frac{1}{2N}R_n(m_n^*)$.

When Conditions M2 and A1 hold, our task is to examine $R_n(\bm{w}_{n,N}^*)-R_n(\bm{w}_n^*)$ has the same order as $R_n(\bm{w}_n^*)$. We consider two scenarios: $M_n\geq m_n(\frac{2N-1}{2N})$ and $M_n< m_n(\frac{2N-1}{2N})$ but $M_n/m_n^{**} \geq \underline{c}$ for any sufficiently large $n$.

\underline{First, consider $M_n\geq m_n(\frac{2N-1}{2N})$}. Define $t_n^N=\min \{t\in \mathbb{N}: \lfloor kt\rfloor\geq m_n(\frac{2N-1}{2N})+1\}$. Then it follows from Theorem \ref{thm:mn} and \citet{peng2021} that $\lim_{n\to \infty} t_n^N=\infty$ and $\lfloor kt_n^N\rfloor \sim m_n(\frac{2N-1}{2N})$, respectively. Using the same arguments as that in \eqref{eq:condlb} and \eqref{eq:cond1}, we have
\begin{eqnarray}\label{eq:thm3up}
	&&\sum_{m=2}^{m_n^*} \mathrm{tr}\{( \mathbf{P}_m- \mathbf{P}_{m-1})\bm{\Omega}\}  (1-\gamma_{n,m}^*) \mathbf{1}\bigl\{\gamma_{n,m}^*>1-1/(2N)\bigr\}\notag \\
	&&\quad=\sum_{m=2}^{m_n(\frac{2N-1}{2N})} \mathrm{tr}\{( \mathbf{P}_m- \mathbf{P}_{m-1})\bm{\Omega}\} -\sum_{m=2}^{m_n(\frac{2N-1}{2N})} \frac{\mathrm{tr}\{( \mathbf{P}_m- \mathbf{P}_{m-1})\bm{\Omega}\}}{1+1/(n\theta_{n,m})}\notag \\
	&&\quad \geq \mathrm{tr}\{( \mathbf{P}_{m_n(\frac{2N-1}{2N})}- \mathbf{P}_{t_n^N})\bm{\Omega}\}-\frac{1}{1+1/(n\theta_{n,t_n^N})} \mathrm{tr}\{( \mathbf{P}_{\lfloor kt_n^N \rfloor}- \mathbf{P}_{t_n^N})\bm{\Omega}\}\notag \\
	&&\quad \geq \frac{1}{1+\frac{\delta}{2N-1}}\mathrm{tr}\left[ \left\{\left(1+\frac{\delta}{2N-1}\right) \mathbf{P}_{m_n(\frac{2N-1}{2N})}-  \mathbf{P}_{\lfloor kt_n^N \rfloor}-\frac{\delta}{2N-1}  \mathbf{P}_{t_n}\right\}\bm{\Omega}\right] \notag \\
	&&\quad \geq \frac{c_1}{1+\frac{\delta}{2N-1}} \left\{ \left(1+\frac{\delta}{2N-1}\right)\nu_{m_n(\frac{2N-1}{2N})}- \nu_{\lfloor kt_n^N \rfloor}-\frac{\delta}{2N-1} \nu_{t_n^N}\right\}\notag \\
	&&\quad \geq \frac{c_1}{1+\frac{\delta}{2N-1}} (\nu_{m_n(\frac{2N-1}{2N})}- \nu_{\lfloor kt_n^N \rfloor})+\frac{c_1\delta}{2N-1+\delta} \left\{m_n\left(\frac{2N-1}{2N}\right)-t_n^N\right\}\notag \\
	&&\quad \sim \frac{c_1(k-1)\delta}{k(2N-1+\delta)} m_n\biggl(\frac{2N-1}{2N}\biggr) \asymp m_n^*,
\end{eqnarray}
where the second inequality is derived by the fact
\begin{equation*}
	\frac{1}{1+1/(n\theta_{n,t_n^N})}\leq \frac{1}{1+\delta/(n\theta_{n, \lfloor kt_n^N\rfloor})}\leq \frac{1}{1+\delta/(n\theta_{n, m_n(\frac{2N-1}{2N})+1})}\leq \frac{1}{1+\delta/(2N-1)},
\end{equation*}
and the last line is due to $\nu_{m_n(\frac{2N-1}{2N})}\sim \nu_{\lfloor kt_n^N \rfloor}$, $t_n^N\sim m_n(\frac{2N-1}{2N})/k$, and Lemma \ref{lem:mxi}(ii). Since
\begin{equation*}
	\frac{1}{2N} R_n(m_n^*)\geq R_n(\bm{w}_{n,N}^*)-R_n(\bm{w}_n^*)\geq \sum_{m=2}^{m_n^*} \mathrm{tr}\{( \mathbf{P}_m- \mathbf{P}_{m-1})\bm{\Omega}\}  (1-\gamma_{n,m}^*) \mathbf{1}\bigl\{\gamma_{n,m}^*>1-1/(2N)\bigr\}, 
\end{equation*}
using \eqref{eq:thm3up} and $\mathrm{tr}( \mathbf{P}_{m_n^*}\bm{\Omega}) \asymp R_n(m_n^*) \asymp R_n(\bm{w}_n^*)$, we have $R_n(\bm{w}_{n,N}^*)-R_n(\bm{w}_n^*)\asymp R_n(\bm{w}_n^*)$.

\underline{Next, consider $M_n< m_n(\frac{2N-1}{2N})$ but $M_n/m_n^{**} \geq \underline{c}$}. Using \eqref{eq:case2} and the similar arguments in \eqref{eq:thm3up}, we can also prove $R_n(\bm{w}_{n,N}^*)-R_n(\bm{w}_n^*)\asymp R_n(\bm{w}_n^*)$. This completes the proof of Theorem \ref{thm:discrete} under Conditions M2 and A1.

When Condition M1 or Conditions M2 and A2 hold, $R_n(\bm{w}_{n,N}^*)- R_n(\bm{w}_n^*)=o\{R_n(\bm{w}_n^*)\}$ is directly followed by Theorems \ref{thm:delta2}--\ref{thm:delta} and the fact $R_n(m_n^*)\geq R_n(\bm{w}_{n,N}^*)\geq R_n(\bm{w}_n^*)$.

\subsection{Proof of the Results in Examples 6.1--6.2}\label{subsec:examp}

Using the expression of $R_n(\bm{w}_{n,N}^*)$ in Lemma \ref{lem:RwnN},
% and in the model setting of \citet{peng2021}, 
we have that for any sufficiently large $n$,
\begin{eqnarray*}
	\frac{1}{n}R_n(\bm{w}_{n,N}^*)=\hspace{-6mm}&&\frac{\sigma^2}{n}+\sum_{i=i_{n,N}+1}^{N}\sum_{m=m_n(\frac{2i+1}{2N})+1}^{m_n(\frac{2i-1}{2N})} \left\{ \frac{\sigma^2}{n}\left(\frac{i}{N}\right)^2+\left(1-\frac{i}{N}\right)^2\beta_m^2\right\} \\
	&&+\sum_{m=m_n(\frac{2i_{n,N}+1}{2N})+1}^{M_n} \left\{ \frac{\sigma^2}{n}\left(\frac{i_{n,N}}{N}\right)^2+\left(1-\frac{i_{n,N}}{N}\right)^2\beta_m^2\right\}+\sum_{m=M_n+1}^{p_n} \beta_m^2.
\end{eqnarray*}

\noindent\underline{Proof of the results in Example 6.1}: When $\beta_m=m^{-\alpha}$ for $\alpha>1/2$, we have $m_n(\frac{2i+1}{2N})\sim (\frac{2N}{2i+1}-1)^{\frac{1}{2\alpha}}(\frac{n}{\sigma^2})^{\frac{1}{2\alpha}}$ for $i=i_{n,N},\ldots,N-1$ and $m_n^{**}\sim (\frac{n}{\sigma^2})^{\frac{1}{2\alpha}}$. When $M_n\equiv M$ is fixed as $n\to \infty$, $i_{n,N}=N$ for any sufficiently large $n$. Therefore,
\begin{equation*}
	\frac{1}{n}R_n(\bm{w}_{n,N}^*)=\frac{M\sigma^2}{n}+\sum_{m=M+1}^{p_n} m^{-2\alpha}\sim \sum_{m=M+1}^{\infty} m^{-2\alpha}.
\end{equation*}
When $M_n\to \infty$ as $n\to \infty$, the optimal risk of MA restricted to $\mathcal{W}_n(N)$ satisfies
\begin{eqnarray}\label{eq:exm51}
	\frac{1}{n}R_n(\bm{w}_{n,N}^*)\hspace{-6mm}&&\sim \frac{\sigma^2}{n}+\sum_{i=i_{n,N}+1}^{N} \int_{m_n(\frac{2i+1}{2N})}^{m_n(\frac{2i-1}{2N})} \left\{ \frac{\sigma^2}{n}\left(\frac{i}{N}\right)^2+\left(1-\frac{i}{N}\right)^2 x^{-2\alpha}\right\}\,dx \notag \\
	&&\quad +\int_{m_n(\frac{2i_{n,N}+1}{2N})}^{M_n} \left\{ \frac{\sigma^2}{n}\left(\frac{i_{n,N}}{N}\right)^2+\left(1-\frac{i_{n,N}}{N}\right)^2 x^{-2\alpha}\right\}\,dx+\int_{M_n}^{p_n} x^{-2\alpha}\,dx \notag \\
	&&\equiv \frac{\sigma^2}{n}+\Pi_{n1}+\Pi_{n2}+\frac{1}{2\alpha-1}(M_n^{-2\alpha+1}-p_n^{-2\alpha+1}).
\end{eqnarray}
Since $m_n^{**}\sim (\frac{n}{\sigma^2})^{\frac{1}{2\alpha}}$, it is easy to see that $i_{n,N}\sim i_{n,N}^*\equiv\Big\lceil \frac{N}{1+(\frac{M_n}{m_n^{**}})^{2\alpha}}-\frac{1}{2}\Big\rceil$. We first simplify $\Pi_{n1}$ as follows:
\begin{eqnarray}\label{eq:exmpi1}
	\Pi_{n1}\hspace{-6mm}&&=\frac{\sigma^2}{n}\sum_{i=i_{n,N}+1}^{N} \left(\frac{i}{N}\right)^2 \left\{m_n\left(\frac{2i-1}{2N}\right)-m_n\left(\frac{2i+1}{2N}\right)\right\} \notag \\
	&& \quad -\frac{1}{2\alpha-1}\sum_{i=i_{n,N}+1}^{N} \left(1-\frac{i}{N}\right)^2 \left\{m_n\left(\frac{2i-1}{2N}\right)^{1-2\alpha}-m_n\left(\frac{2i+1}{2N}\right)^{1-2\alpha}\right\} \notag \\
	&&=\frac{\sigma^2}{n}\frac{2}{N}\sum_{i=i_{n,N}}^{N-1} \left(\frac{2i+1}{2N}\right)m_n\left(\frac{2i+1}{2N}\right)+\frac{1}{2\alpha-1}\frac{2}{N}\sum_{i=i_{n,N}}^{N-1} \left(1-\frac{2i+1}{2N}\right)m_n\left(\frac{2i+1}{2N}\right)^{1-2\alpha} \notag \\
	&&\quad -\frac{\sigma^2}{n}+\frac{\sigma^2}{n} \left(\frac{i_{n,N}}{N}\right)^2 m_n\left(\frac{2i_{n,N}+1}{2N}\right)-\frac{1}{2\alpha-1}\left(1-\frac{i_{n,N}}{2N}\right)^2 m_n\left(\frac{2i_{n,N}+1}{2N}\right)^{1-2\alpha} \notag \\
	&&\sim \frac{2\alpha}{2\alpha-1}\left(\frac{n}{\sigma^2}\right)^{\frac{1}{2\alpha}-1} \frac{2}{N}\sum_{i=i_{n,N}^*}^{N-1} \left(\frac{2i+1}{2N}\right)^{1-\frac{1}{2\alpha}} \left(1-\frac{2i+1}{2N}\right)^{\frac{1}{2\alpha}}-\frac{\sigma^2}{n} \notag \\
	&&\quad +\frac{\sigma^2}{n} \left(\frac{i_{n,N}}{N}\right)^2 m_n\left(\frac{2i_{n,N}+1}{2N}\right)-\frac{1}{2\alpha-1}\left(1-\frac{i_{n,N}}{2N}\right)^2 m_n\left(\frac{2i_{n,N}+1}{2N}\right)^{1-2\alpha}.
\end{eqnarray}
Next, we simplify $\Pi_{n2}$ as follows:
\begin{eqnarray}\label{eq:exmpi2}
	\Pi_{n2}\hspace{-6mm}&&=\frac{\sigma^2}{n}\left(\frac{i_{n,N}}{N}\right)^2 \left\{M_n-m_n\left(\frac{2i_{n,N}+1}{2N}\right)\right\} \notag \\
	&&\quad -\frac{1}{2\alpha-1} \left(1-\frac{i_{n,N}}{N}\right)^2 \left\{M_n^{1-2\alpha}-m_n\left(\frac{2i_{n,N}+1}{2N}\right)^{1-2\alpha}\right\} \notag \\
	&& \sim \left(\frac{n}{\sigma^2}\right)^{\frac{1}{2\alpha}-1} \left(\frac{i_{n,N}^*}{N}\right)^2 \frac{M_n}{m_n^{**}}-\frac{1}{2\alpha-1}\left(\frac{n}{\sigma^2}\right)^{\frac{1}{2\alpha}-1} \left(1-\frac{i_{n,N}^*}{N}\right)^2 \left(\frac{M_n}{m_n^{**}}\right)^{1-2\alpha} \notag \\
	&& \quad -\frac{\sigma^2}{n} \left(\frac{i_{n,N}}{N}\right)^2 m_n\left(\frac{2i_{n,N}+1}{2N}\right)+\frac{1}{2\alpha-1}\left(1-\frac{i_{n,N}}{2N}\right)^2 m_n\left(\frac{2i_{n,N}+1}{2N}\right)^{1-2\alpha}.
\end{eqnarray}
Combining \eqref{eq:exm51}, \eqref{eq:exmpi1}, and \eqref{eq:exmpi2}, we have that when $M_n\to \infty$ as $n\to \infty$,
\begin{equation}\label{eq:ex5.1riskdis}
	\frac{1}{n}R_n(\bm{w}_{n,N}^*)\sim \frac{2\alpha}{2\alpha-1}\left(\frac{n}{\sigma^2}\right)^{\frac{1}{2\alpha}-1} \psi_{n,N}+\frac{1}{2\alpha-1}(M_n^{-2\alpha+1}-p_n^{-2\alpha+1}),
\end{equation}
where
\begin{equation*}
	\psi_{n,N}=\frac{2}{N}\sum_{i=i_{n,N}^*}^{N-1} \left(\frac{2i+1}{2N}\right)^{1-\frac{1}{2\alpha}} \left(1-\frac{2i+1}{2N}\right)^{\frac{1}{2\alpha}}+\frac{2\alpha-1}{2\alpha} \left(\frac{i_{n,N}^*}{N}\right)^2 \frac{M_n}{m_n^{**}}-\frac{1}{2\alpha}\left(1-\frac{i_{n,N}^*}{N}\right)^2 \left(\frac{M_n}{m_n^{**}}\right)^{1-2\alpha}.
\end{equation*}

When $M_n\to \infty$ as $n\to \infty$, it is shown in \citet{peng2021} that the optimal risk of MA with the weight set $\mathcal{W}_n$ satisfies
\begin{equation}\label{eq:ex5.1w}
	\frac{1}{n}R_n(\bm{w}_n^*)\sim \frac{1}{2\alpha}\left(\frac{n}{\sigma^2}\right)^{\frac{1}{2\alpha}-1} \left\{\frac{\pi}{\sin(\frac{\pi}{2\alpha})}-B\left(\frac{1}{1+\left(\frac{M_n}{m_n^{**}}\right)^{2\alpha}}; 1-\frac{1}{2\alpha},\frac{1}{2\alpha}\right)\right\}+\frac{1}{2\alpha-1}(M_n^{-2\alpha+1}-p_n^{-2\alpha+1}).
\end{equation}
When $M_n\equiv M$ is fixed as $n\to \infty$,
\begin{equation*}
	\frac{1}{n}R_n(\bm{w}_n^*)=\frac{\sigma^2}{n}+\sum_{m=2}^M \frac{1}{\frac{n}{\sigma^2}+m^{2\alpha}}+\sum_{m=M+1}^{p_n} m^{-2\alpha}\sim \sum_{m=M+1}^{\infty} m^{-2\alpha}.
\end{equation*}
Therefore, we consider different conditions on $M_n$ as follows.

(i) When $M_n\equiv M$ is fixed as $n\to \infty$,
\begin{equation*}
	\frac{1}{n}R_n(\bm{w}_{n,N}^*)\sim \frac{1}{n}R_n(\bm{w}_n^*)\sim \sum_{m=M+1}^{\infty} m^{-2\alpha}.
\end{equation*}

(ii) When $M_n\to \infty$ but $M_n/m_n^{**}\to 0$ as $n\to \infty$, $i_{n,N}^*=N$ for any sufficiently large $n$ and thus $\psi_{n,N}=o(1)$, which, along with the fact that $B(1;1-\frac{1}{2\alpha},\frac{1}{2\alpha})=\frac{\pi}{\sin(\frac{\pi}{2\alpha})}$, yields that
\begin{equation*}
	\frac{1}{n}R_n(\bm{w}_{n,N}^*)\sim \frac{1}{n}R_n(\bm{w}_n^*)\sim \frac{1}{2\alpha-1}(M_n^{-2\alpha+1}-p_n^{-2\alpha+1})\sim \frac{M_n^{-2\alpha+1}}{2\alpha-1}.
\end{equation*}

(iii) When $M_n/m_n^{**}\geq \underline{c}$ for some $\underline{c}>0$, let us find the lower bound of $R_n(\bm{w}_{n,N}^*)- R_n(\bm{w}_n^*)$. If $M_n\geq m_n(\frac{2N-1}{2N})$, note
\begin{eqnarray*}
	&& \sum_{m=2}^{m_n^*} \mathrm{tr}\{( \mathbf{P}_m- \mathbf{P}_{m-1})\bm{\Omega}\}  (1-\gamma_{n,m}^*) \mathbf{1}\bigl\{\gamma_{n,m}^*>1-1/(2N)\bigr\} \\
	&&\quad = \sum_{m=2}^{m_n(\frac{2N-1}{2N})} \sigma^2\left(1-\frac{\beta_m^2}{\beta_m^2+\sigma^2/n}\right)\geq \sum_{m=\lfloor m_n(\frac{2N-1}{2N})/2\rfloor}^{m_n(\frac{2N-1}{2N})} \frac{\sigma^4/n}{m^{-2\alpha}+\sigma^2/n} \\
	&&\quad \geq \left\lfloor \frac{1}{2} m_n\left(\frac{2N-1}{2N}\right)\right\rfloor \frac{\sigma^4/n}{\lfloor m_n(\frac{2N-1}{2N})/2\rfloor^{-2\alpha}+\sigma^2/n}\sim \frac{(2N-1)^{-\frac{1}{2\alpha}}\sigma^2}{2^{2\alpha+1}(2N-1)+2} \left(\frac{n}{\sigma^2}\right)^{\frac{1}{2\alpha}}.
\end{eqnarray*}
If $m_n(\frac{2N-1}{2N})>M_n\geq \underline{c}m_n^{**}$, we also have
\begin{eqnarray*}
		&& \sum_{m=2}^{m_n^*} \mathrm{tr}\{( \mathbf{P}_m- \mathbf{P}_{m-1})\bm{\Omega}\}  (1-\gamma_{n,m}^*) \mathbf{1}\bigl\{\gamma_{n,m}^*>1-1/(2N)\bigr\} \\
		&&\quad \geq \left\lfloor \underline{c} m_n^{**}/2\right\rfloor \frac{\sigma^4/n}{\lfloor \underline{c} m_n^{**}/2\rfloor^{-2\alpha}+\sigma^2/n}\sim \frac{\underline{c}\sigma^2}{2^{2\alpha+1}\underline{c}^{-2\alpha}+2} \left(\frac{n}{\sigma^2}\right)^{\frac{1}{2\alpha}}.
\end{eqnarray*}
As a result, $R_n(\bm{w}_{n,N}^*)- R_n(\bm{w}_n^*)$ can be lower bounded by $\frac{\varpi\sigma^2}{2^{2\alpha+1}\varpi^{-2\alpha}+2} \left(\frac{n}{\sigma^2}\right)^{\frac{1}{2\alpha}}$, where $\varpi=\min\{\underline{c},(2N-1)^{-\frac{1}{2\alpha}}\}$. Moreover, if $\lim_{n\to \infty} M_n/m_n^{**}= \kappa\in (0,\infty]$ and $M_n=o(p_n)$ are satisfied, it follows from \eqref{eq:ex5.1riskdis} and \eqref{eq:ex5.1w} that
\begin{equation*}
	\lim_{n\to \infty} \frac{R_n(\bm{w}_n^*)}{R_n(\bm{w}_{n,N}^*)}=\frac{1}{\psi_N^*+\frac{\kappa^{-2\alpha+1}}{2\alpha}} \left[\frac{2\alpha-1}{4\alpha^2}\left\{ \frac{\pi}{\sin(\frac{\pi}{2\alpha})}-B\left(\frac{1}{1+\kappa^{2\alpha}}; 1-\frac{1}{2\alpha},\frac{1}{2\alpha}\right)\right\}+\frac{\kappa^{-2\alpha+1}}{2\alpha}\right],
\end{equation*}
where
\begin{equation*}
	\psi_N^*=\frac{2}{N}\sum_{i=i_N^*}^{N-1} \left(\frac{2i+1}{2N}\right)^{1-\frac{1}{2\alpha}} \left(1-\frac{2i+1}{2N}\right)^{\frac{1}{2\alpha}}+\frac{2\alpha-1}{2\alpha} \left(\frac{i_N^*}{N}\right)^2 \kappa-\frac{1}{2\alpha}\left(1-\frac{i_N^*}{N}\right)^2 \kappa^{1-2\alpha}
\end{equation*}
and $i_N^*=\left\lceil \frac{N}{1+\kappa^{2\alpha}}-\frac{1}{2}\right\rceil$. It is easy to see that $\{\psi_N^*\}_{N=1}^{\infty}$ is a strictly decreasing sequence with $\psi_1^*=1-\frac{\kappa^{-2\alpha+1}}{2\alpha}$.
Moreover, we can prove that
\begin{eqnarray*}
	\lim_{N\to \infty} \psi_N^*\hspace{-6mm}&&=2\int_{\frac{1}{1+\kappa^{2\alpha}}}^1 t^{1-\frac{1}{2\alpha}}(1-t)^{\frac{1}{2\alpha}}\,dt+\frac{2\alpha-1}{2\alpha}\frac{\kappa}{(1+\kappa^{2\alpha})^2}-\frac{1}{2\alpha}\frac{\kappa^{1+2\alpha}}{(1+\kappa^{2\alpha})^2} \\
	&&=\frac{2\alpha-1}{4\alpha^2} \int_{\frac{1}{1+\kappa^{2\alpha}}}^1 t^{-\frac{1}{2\alpha}}(1-t)^{\frac{1}{2\alpha}-1}\,dt \\
	&&=\frac{2\alpha-1}{4\alpha^2}\left\{ \frac{\pi}{\sin(\frac{\pi}{2\alpha})}-B\left(\frac{1}{1+\kappa^{2\alpha}}; 1-\frac{1}{2\alpha},\frac{1}{2\alpha}\right)\right\},
\end{eqnarray*}
where the last equality follows from the fact that $B(1;1-\frac{1}{2\alpha},\frac{1}{2\alpha})=\frac{\pi}{\sin(\frac{\pi}{2\alpha})}$. Therefore, for any fixed $N\geq 1$, 
\begin{equation*}
	\lim_{n\to \infty} \frac{R_n(\bm{w}_n^*)}{R_n(\bm{w}_{n,N}^*)}<1.
\end{equation*}

\noindent\underline{Proof of the results in Example 6.2}: When $\beta_m=\exp(-cm)$ for $c>0$, we have $m_n(\frac{2i+1}{2N})\sim \frac{1}{2c} \log \left(\frac{n}{\sigma^2}\right)$ for $i=i_{n,N},\ldots,N-1$ and $m_n^{**}\sim \frac{1}{2c} \log \left(\frac{n}{\sigma^2}\right)$. The optimal risk of MA satisfies
\begin{eqnarray}\label{eq:ex5.2Rw}
	\frac{1}{n}R_n(\bm{w}_n^*)\hspace{-6mm}&&=\frac{\sigma^2}{n}+\sum_{m=2}^{M_n} \frac{1}{\frac{n}{\sigma^2}+\exp(2cm)}+\sum_{m=M_n+1}^{p_n} \exp(-2cm) \notag \\
	&&\sim \sum_{m=1}^{M_n} \frac{1}{\frac{n}{\sigma^2}+\exp(2cm)}+\frac{\exp(-2cM_n)-\exp(-2cp_n)}{\exp(2c)-1}.
\end{eqnarray}
We consider different conditions on $M_n$ as follows.

(i) When $\limsup_{n\to \infty} M_n/m_n^{**}<1$, $M_n<m_n^{**}$ for any sufficiently large $n$. Thus,
\begin{equation}\label{eq:ex5.2Rm}
	\frac{1}{n}R_n(m_n^*)=\frac{M_n\sigma^2}{n}+\sum_{m=M_n+1}^{p_n} \exp(-2cm)=\frac{M_n\sigma^2}{n}+\frac{\exp(-2cM_n)-\exp(-2cp_n)}{\exp(2c)-1}.
\end{equation}
By $2cm_n^{**}\sim \log \left(\frac{n}{\sigma^2}\right)$ and $\lim_{n\to \infty} \log(M_n)/m^{**}=0$, we observe that
\begin{equation*}
	\limsup_{n\to \infty} \log\left\{\frac{M_n\sigma^2/n}{\exp(-2cM_n)}\right\} \Big/(2cm_n^{**})=\limsup_{n\to \infty} \frac{\log M_n-\log(\frac{n}{\sigma^2})+2cM_n}{2cm_n^{**}}\leq -1+\limsup_{n\to \infty} \frac{M_n}{m_n^{**}}<0,
\end{equation*}
which implies that $\frac{M_n\sigma^2/n}{\exp(-2cM_n)}\to 0$ as $n\to \infty$. Moreover, as $n\to \infty$,
\begin{equation*}
\frac{\exp(-2cp_n)}{\exp(-2cM_n)}=\exp\{-2c(p_n-M_n)\}\leq \exp\{-2c(m_n^{**}-M_n)\}\to 0.
\end{equation*}
Therefore, we have $\frac{1}{n}R_n(m_n^*)\sim \frac{\exp(-2cM_n)}{\exp(2c)-1}$. Since $\sum_{m=1}^{M_n} \{\frac{n}{\sigma^2}+\exp(2cm)\}^{-1}\leq \frac{\sigma^2}{n}M_n$, from \eqref{eq:ex5.2Rw},  we have $\frac{1}{n}R_n(\bm{w}_n^*)\sim \frac{\exp(-2cM_n)}{\exp(2c)-1}$. Therefore,
\begin{equation*}
	\frac{1}{n}R_n(m_n^*)\sim \frac{1}{n}R_n(\bm{w}_n^*)\sim \frac{\exp(-2cM_n)}{\exp(2c)-1}.
\end{equation*}

(ii) When $M_n\geq m_n^{**}$ for any sufficiently large $n$, note that as $n\to \infty$,
\begin{equation}\label{eq:ex5.2large}
	\frac{\exp(-2cm_n^{**})}{m_n^{**}\sigma^2/n}\leq \frac{\exp(-2cm_n^{**})}{m_n^{**}\exp\{-2c(m_n^{**}+1)\}}=\frac{\exp(2c)}{m_n^{**}} \to 0,
\end{equation}
where the inequality is due to $\sigma^2/n\geq \exp\{-2c(m_n^{**}+1)\}$ derived from \eqref{eq:optm2}. Therefore, we have
\begin{equation*}
\frac{1}{n}R_n(m_n^*)=\frac{m_n^{**}\sigma^2}{n}+\frac{\exp(-2cm_n^{**})-\exp(-2cp_n)}{\exp(2c)-1}\sim \frac{m_n^{**}\sigma^2}{n}.
\end{equation*}
Next, we investigate $R_n(\bm{w}_n^*)$. From \eqref{eq:ex5.2Rw},
\begin{eqnarray}\label{eq:ex5.2lg1}
	\frac{1}{n}R_n(\bm{w}_n^*)\hspace{-6mm}&& \sim \sum_{m=1}^{M_n} \frac{1}{\frac{n}{\sigma^2}+\exp(2cm)}+\sum_{m=M_n+1}^{p_n} \exp(-2cm) \notag \\
	&&=\sum_{m=1}^{m_n^{**}} \frac{1}{\frac{n}{\sigma^2}+\exp(2cm)}+\sum_{m=m_n^{**}+1}^{M_n} \frac{1}{\frac{n}{\sigma^2}+\exp(2cm)}+\sum_{m=M_n+1}^{p_n} \exp(-2cm).
\end{eqnarray}
For the first term of \eqref{eq:ex5.2lg1}, it is easy to obtain
\begin{eqnarray}\label{eq:ex5.2lg2}
	\sum_{m=1}^{m_n^{**}} \frac{1}{\frac{n}{\sigma^2}+\exp(2cm)}\hspace{-6mm}&& \sim \int_0^{m_n^{**}} \frac{1}{\frac{n}{\sigma^2}+\exp(2cx)}\,dx \notag \\
	&&=\frac{m_n^{**}\sigma^2}{n}-\frac{1}{2c}\frac{\sigma^2}{n} \log\left\{ \frac{1+\frac{\sigma^2}{n}\exp(2cm_n^{**})}{1+\frac{\sigma^2}{n}}\right\} \sim \frac{m_n^{**}\sigma^2}{n},
\end{eqnarray}
where the last ``$\sim$" is due to $\frac{\sigma^2}{n}\exp(2cm_n^{**})<1$ derived from \eqref{eq:optm2}. For the last two terms of \eqref{eq:ex5.2lg1}, using \eqref{eq:ex5.2large}, we have
\begin{eqnarray*}
	&&\sum_{m=m_n^{**}+1}^{M_n} \frac{1}{\frac{n}{\sigma^2}+\exp(2cm)}+\sum_{m=M_n+1}^{p_n} \exp(-2cm) \\
	&&\quad \leq \sum_{m=m_n^{**}+1}^{p_n} \exp(-2cm)=\frac{\exp(-2cm_n^{**})-\exp(-2cp_n)}{\exp(2c)-1}=o\left(\frac{m_n^{**}\sigma^2}{n}\right),
\end{eqnarray*}
which, along with \eqref{eq:ex5.2lg1} and \eqref{eq:ex5.2lg2}, yields that $\frac{1}{n}R_n(\bm{w}_n^*)\sim \frac{m_n^{**}\sigma^2}{n}$. Therefore,
\begin{equation*}
	\frac{1}{n}R_n(m_n^*)\sim \frac{1}{n}R_n(\bm{w}_n^*)\sim \frac{m_n^{**}\sigma^2}{n}\sim \frac{1}{2c} \frac{\sigma^2}{n}\log \left(\frac{n}{\sigma^2}\right).
\end{equation*}

(iii) When $M_n<m_n^{**}$ for any sufficiently large $n$ but $\lim_{n\to \infty} M_n/m_n^{**}=1$, by using the same arguments in \eqref{eq:ex5.2lg2}, we can show that
$\sum_{m=1}^{M_n} \{\frac{n}{\sigma^2}+\exp(2cm)\}^{-1}\sim \frac{M_n\sigma^2}{n}$, which, along with \eqref{eq:ex5.2Rw} and \eqref{eq:ex5.2Rm}, yields that
\begin{equation*}
	\frac{1}{n}R_n(m_n^*)\sim \frac{1}{n}R_n(\bm{w}_n^*)\sim \frac{1}{2c} \frac{\sigma^2}{n}\log \left(\frac{n}{\sigma^2}\right)+\frac{\exp(-2cM_n)-\exp(-2cp_n)}{\exp(2c)-1}.
\end{equation*}

Combining the results of (i)--(iii) and the fact $R_n(m_n^*)\geq R_n(\bm{w}_{n,N}^*)\geq R_n(\bm{w}_n^*)$, we obtain the results of Example 6.2.

\newpage

\baselineskip=14pt
\bibliographystyle{apalike}
\bibliography{refers}

\clearpage\pagebreak\newpage

\pagestyle{empty}

\begin{figure}[htpb]
	\centering
	\begin{subfigure}{\textwidth}\centering
	\includegraphics[scale=0.65]{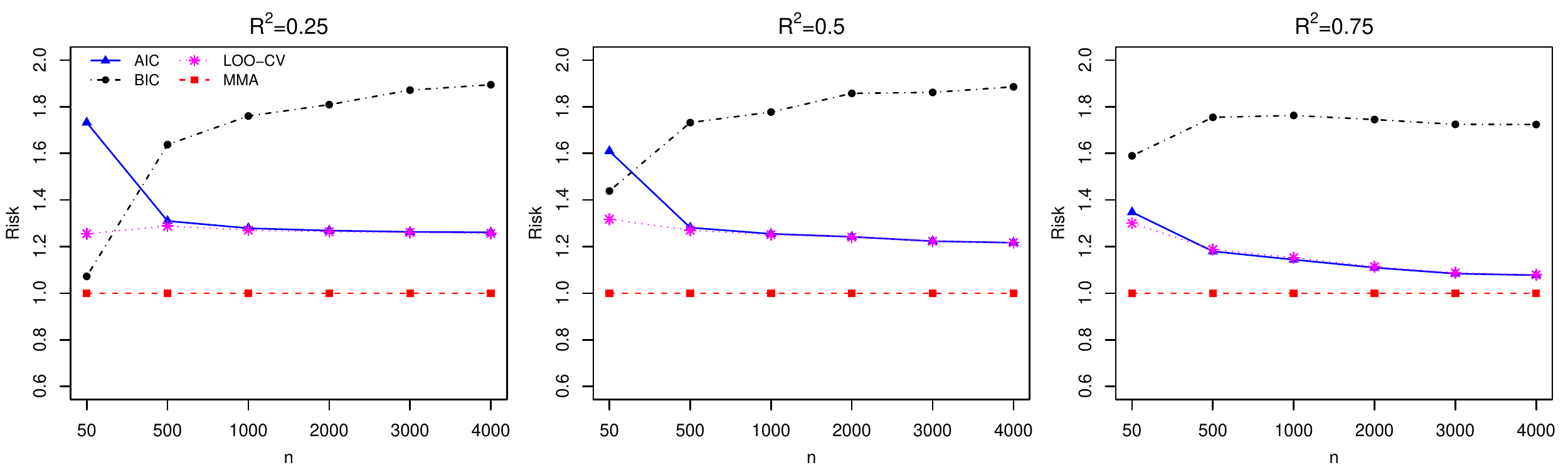}
	\caption{}
	\end{subfigure}
	\begin{subfigure}{\textwidth}\centering
	\includegraphics[scale=0.65]{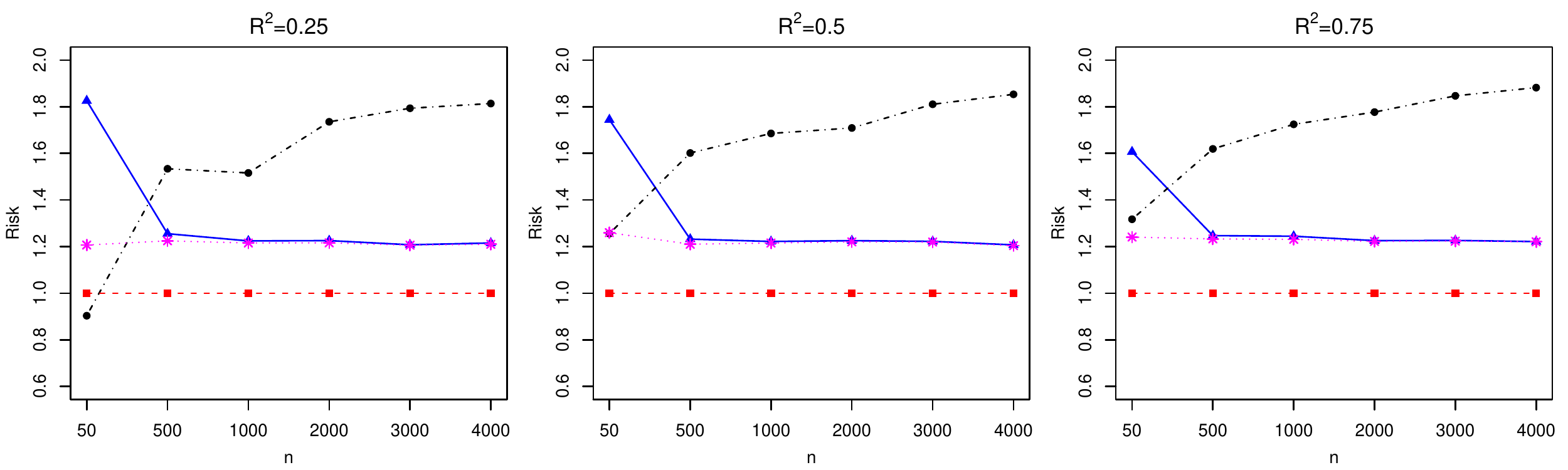}
	\caption{}
\end{subfigure}
	\begin{subfigure}{\textwidth}\centering
	\includegraphics[scale=0.65]{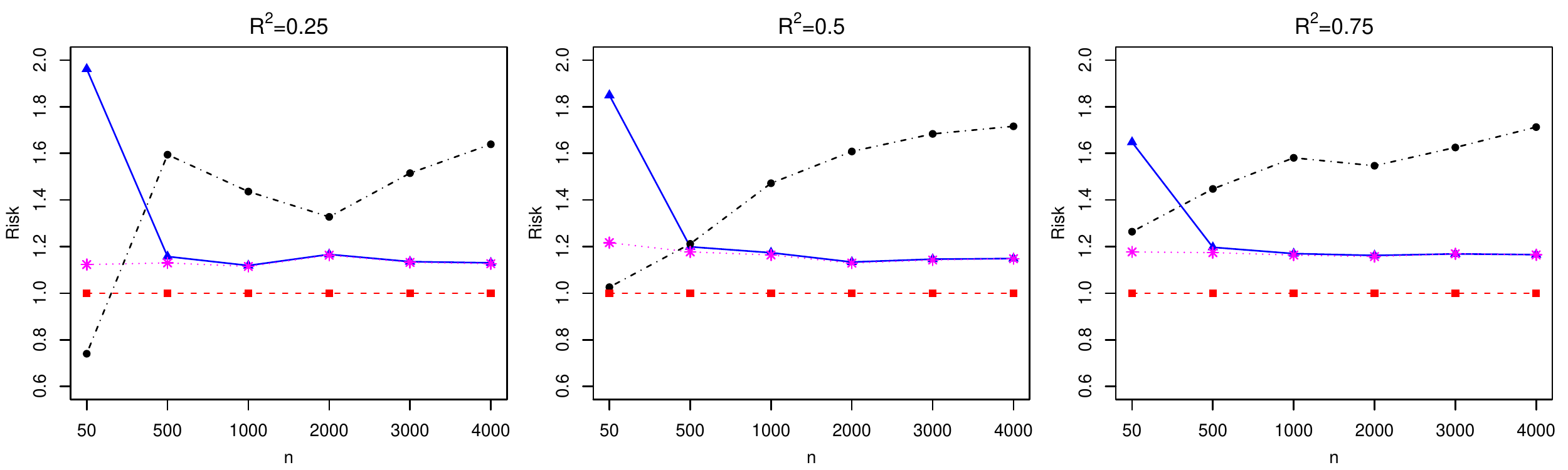}
	\caption{}
\end{subfigure}
	\caption{Simulation results for Example 1 with the case of slowly decaying $\theta_m^*$. Normalized risk functions for AIC, BIC, LOO-CV, and MMA when $\theta_m^*=m^{-2\alpha_1}/\sigma^2$ with $\alpha_1=1$ in row (a), $\alpha_1=1.5$ in row (b), and $\alpha_1=2$ in row (c).}\label{fig:example1_1}
\end{figure}

\clearpage\pagebreak\newpage

\pagestyle{empty}

\begin{figure}[htpb]
	\centering
	\begin{subfigure}{\textwidth}\centering
		\includegraphics[scale=0.65]{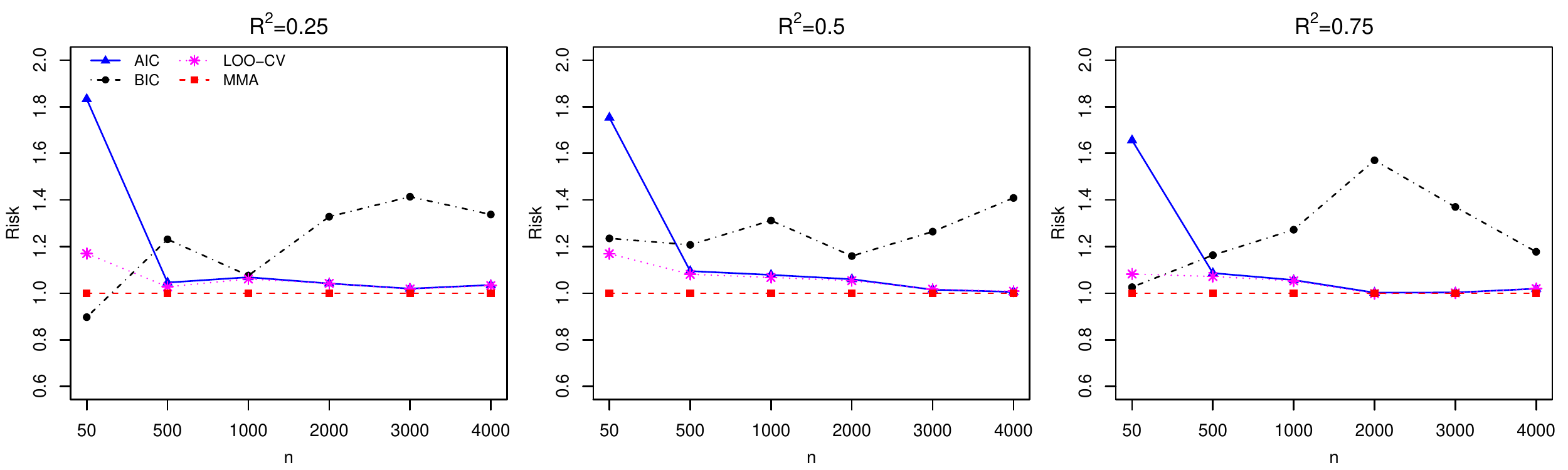}
		\caption{}
	\end{subfigure}
	\begin{subfigure}{\textwidth}\centering
		\includegraphics[scale=0.65]{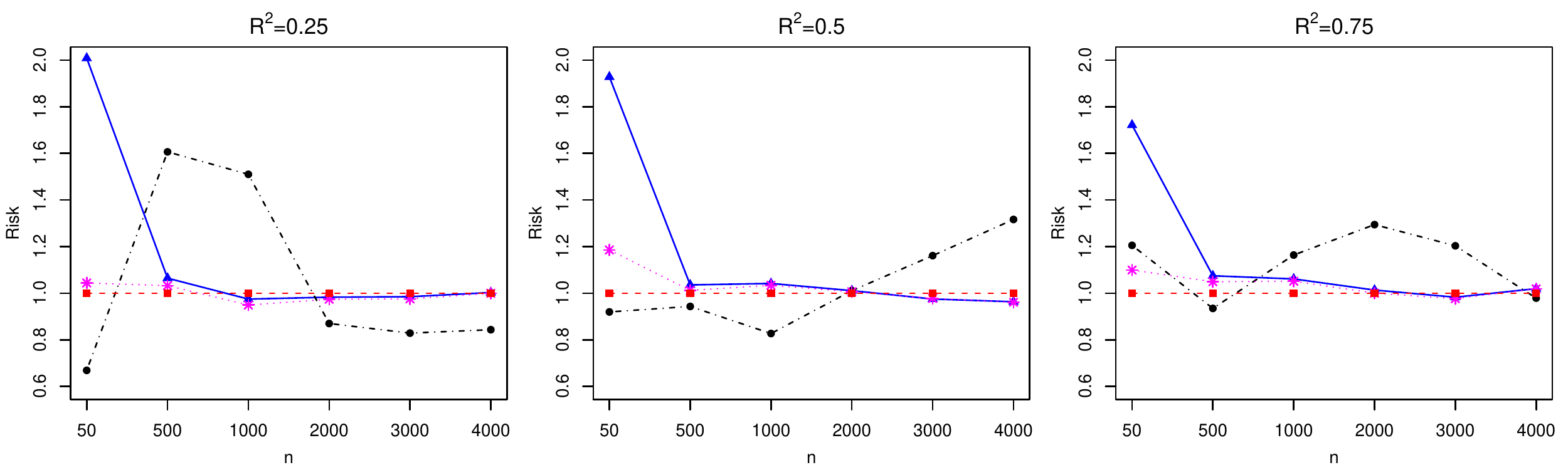}
		\caption{}
	\end{subfigure}
	\begin{subfigure}{\textwidth}\centering
		\includegraphics[scale=0.65]{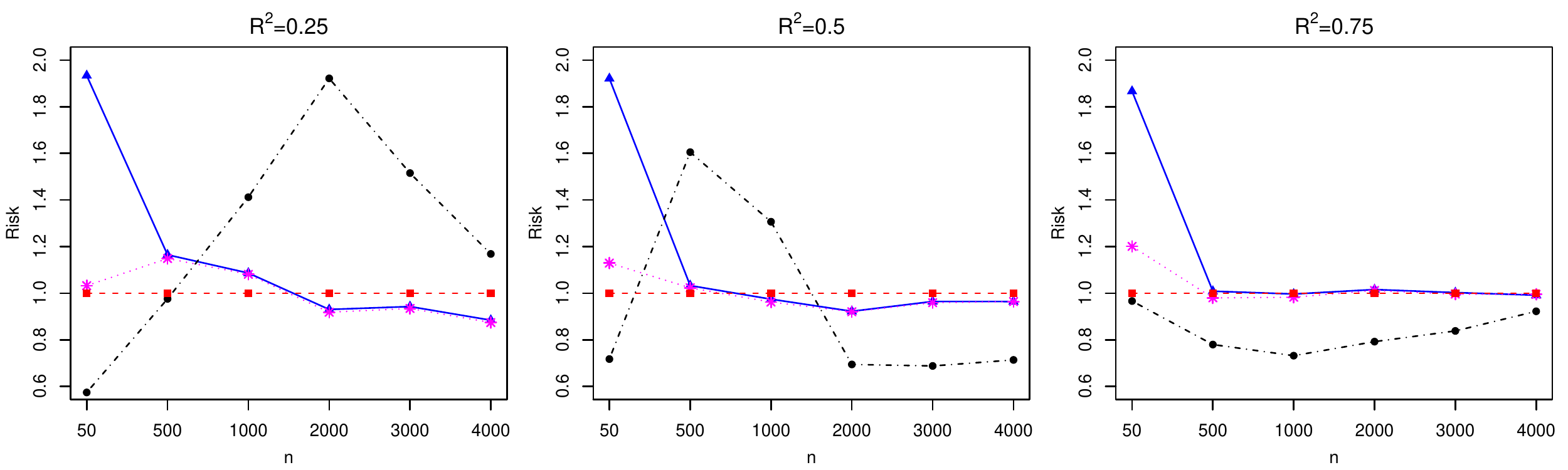}
		\caption{}
	\end{subfigure}
	\caption{Simulation results for Example 1 with the case of fast decaying $\theta_m^*$. Normalized risk functions for AIC, BIC, LOO-CV, and MMA when $\theta_m^*=\exp(-2\alpha_2 m)/\sigma^2$ with $\alpha_2=1$ in row (a), $\alpha_2=1.5$ in row (b), and $\alpha_2=2$ in row (c).}\label{fig:example1_2}
\end{figure}

\clearpage\pagebreak\newpage

\pagestyle{empty}

\begin{figure}[htpb]
	\centering
	\begin{subfigure}{\textwidth}\centering
		\includegraphics[scale=0.65]{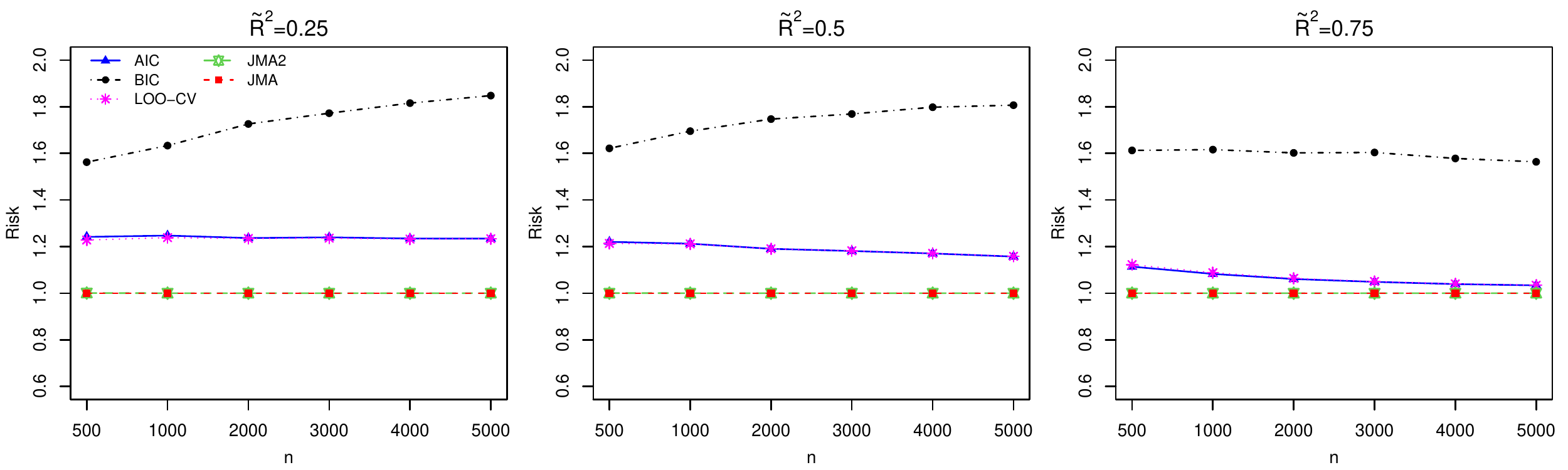}
		\caption{}
	\end{subfigure}
	\begin{subfigure}{\textwidth}\centering
		\includegraphics[scale=0.65]{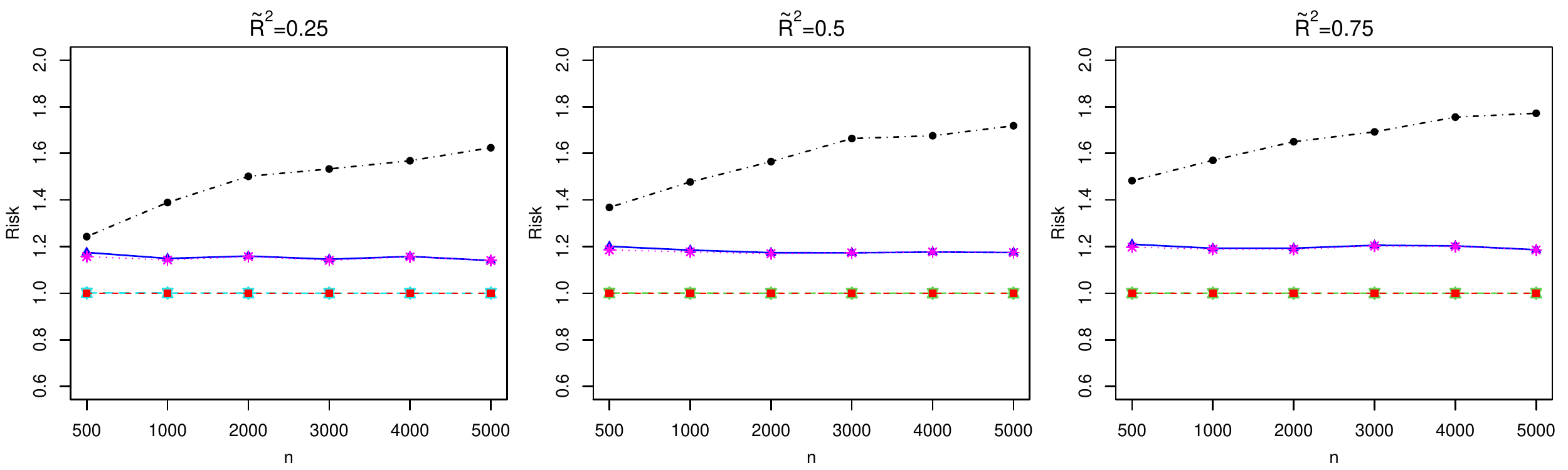}
		\caption{}
	\end{subfigure}
	\begin{subfigure}{\textwidth}\centering
		\includegraphics[scale=0.65]{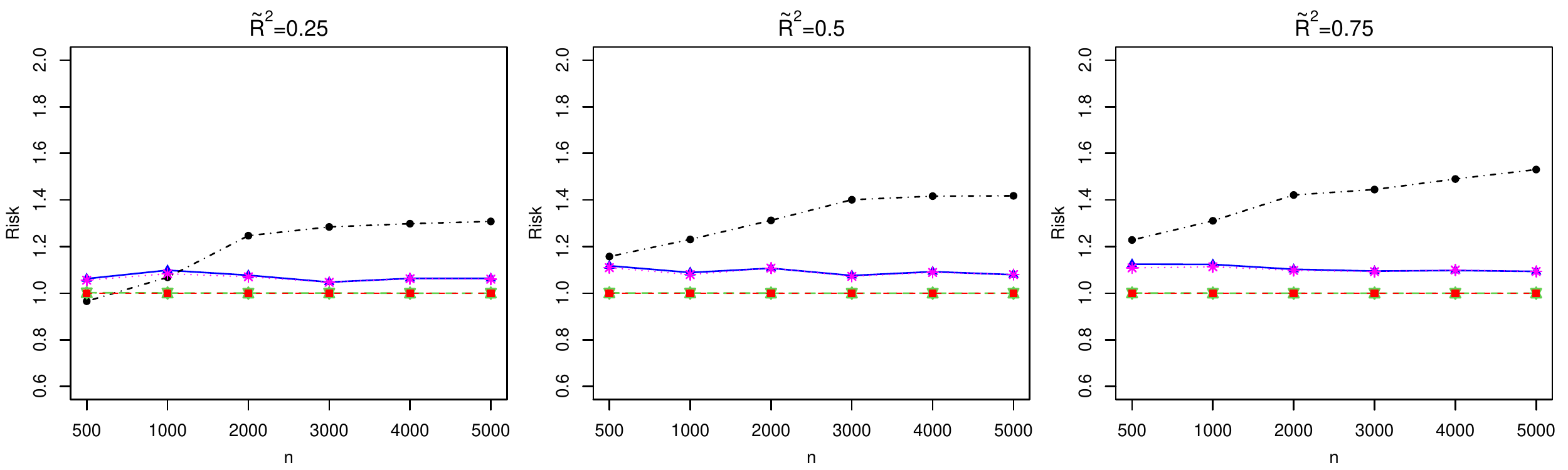}
		\caption{}
	\end{subfigure}
	\caption{Simulation results for Example 2 with the case of slowly decaying $\theta_m^*$. Normalized risk functions for AIC, BIC, LOO-CV, JMA2, and JMA when $\theta_m^*=c^2 m^{-2\alpha_1}$ with $\alpha_1=1$ in row (a), $\alpha_1=1.5$ in row (b), and $\alpha_1=2$ in row (c).}\label{fig:example2_1}
\end{figure}

\clearpage\pagebreak\newpage

\pagestyle{empty}

\begin{figure}[htpb]
	\centering
	\begin{subfigure}{\textwidth}\centering
		\includegraphics[scale=0.65]{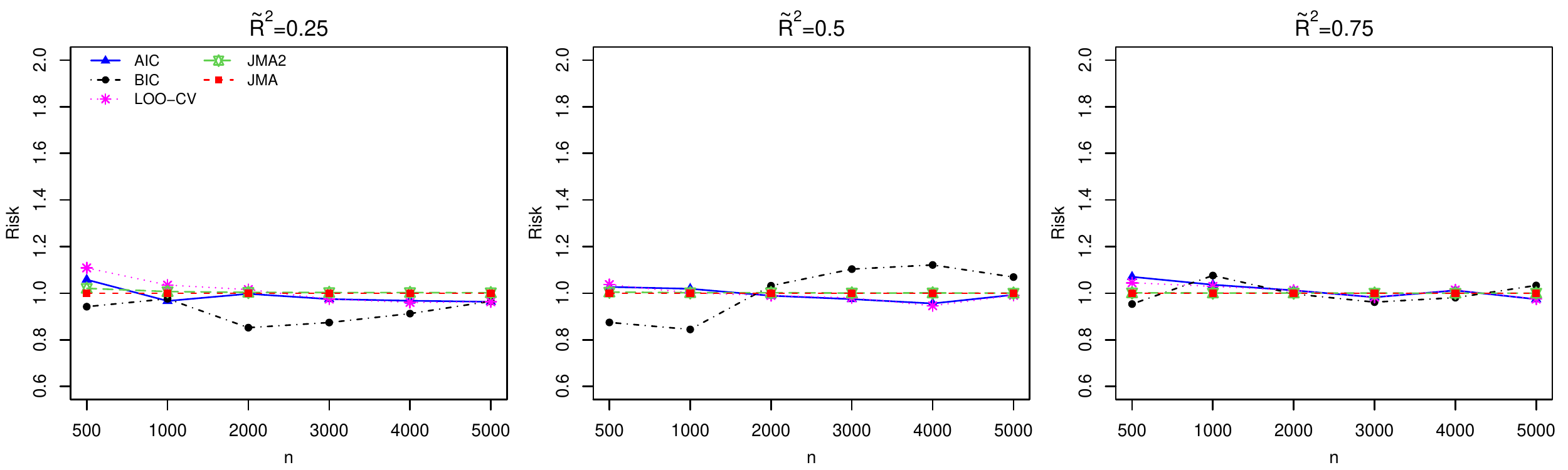}
		\caption{}
	\end{subfigure}
	\begin{subfigure}{\textwidth}\centering
		\includegraphics[scale=0.65]{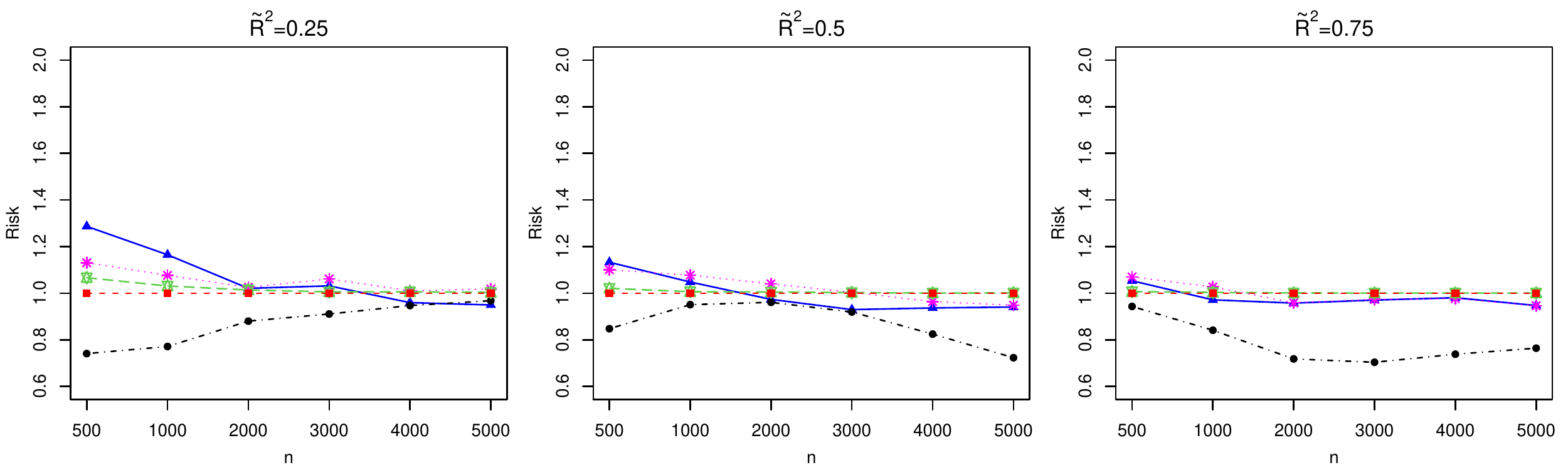}
		\caption{}
	\end{subfigure}
	\begin{subfigure}{\textwidth}\centering
		\includegraphics[scale=0.65]{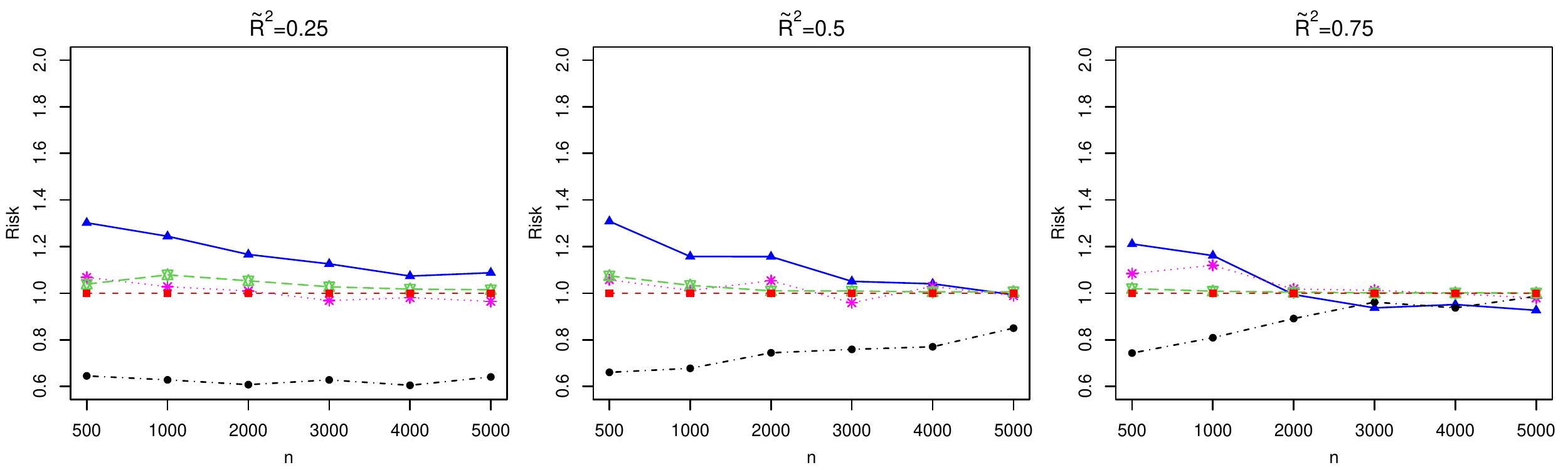}
		\caption{}
	\end{subfigure}
	\caption{Simulation results for Example 2 with the case of fast decaying $\theta_m^*$. Normalized risk functions for AIC, BIC, LOO-CV, JMA2, and JMA when $\theta_m^*=c^2 \exp(-2\alpha_2 m)$ with $\alpha_2=1$ in row (a), $\alpha_2=1.5$ in row (b), and $\alpha_2=2$ in row (c).}\label{fig:example2_2}
\end{figure}

\clearpage\pagebreak\newpage

\pagestyle{empty}

\begin{figure}[htpb]
	\centering
	\begin{subfigure}{\textwidth}\centering
		\includegraphics[scale=0.65]{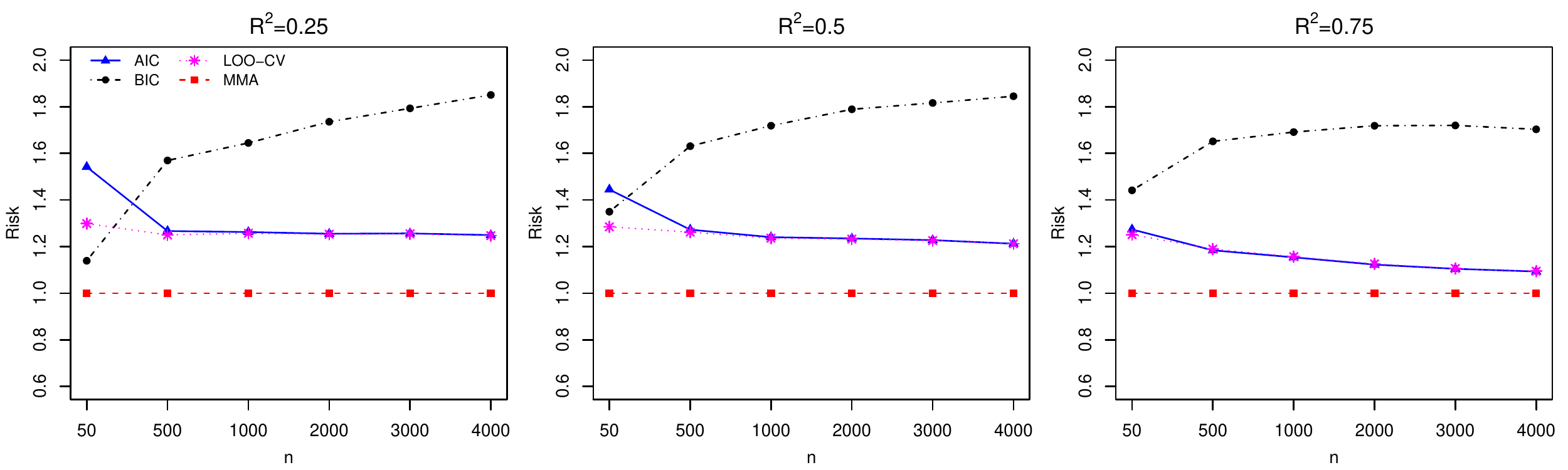}
		\caption{}
	\end{subfigure}
	\begin{subfigure}{\textwidth}\centering
		\includegraphics[scale=0.65]{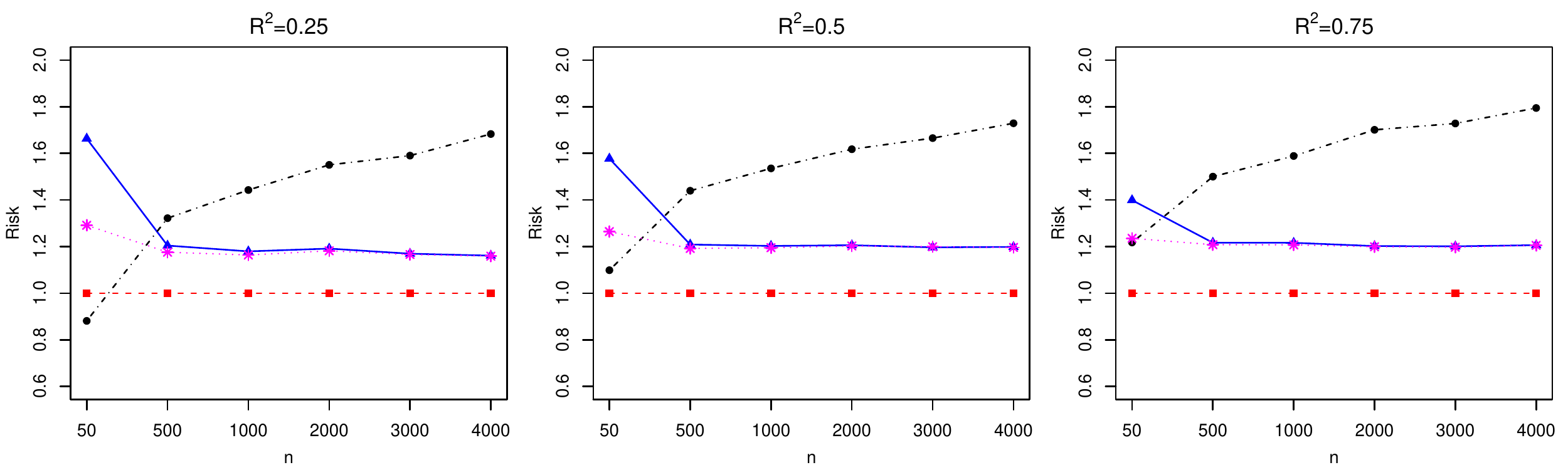}
		\caption{}
	\end{subfigure}
	\begin{subfigure}{\textwidth}\centering
		\includegraphics[scale=0.65]{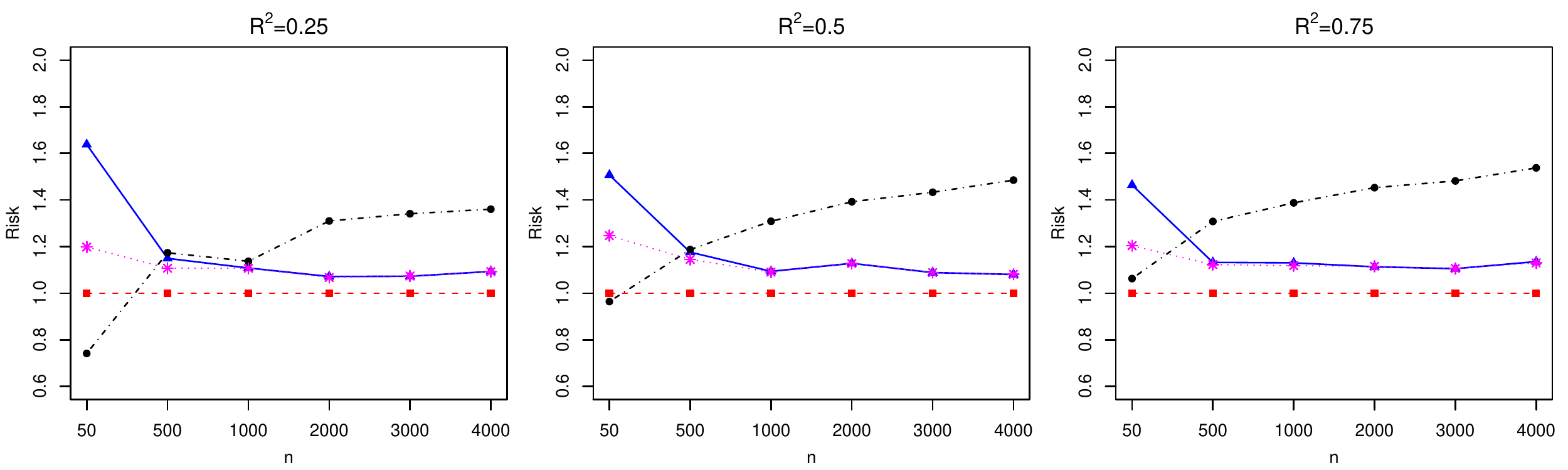}
		\caption{}
	\end{subfigure}
	\caption{Simulation results for Example 3 with the case of slowly decaying $\theta_m^*$. Normalized risk functions for AIC, BIC, LOO-CV, and MMA when $\theta_m^*=m^{-2\alpha_1}/\sigma^2$ with $\alpha_1=1$ in row (a), $\alpha_1=1.5$ in row (b), and $\alpha_1=2$ in row (c).}\label{fig:example3_1}
\end{figure}

\clearpage\pagebreak\newpage

\pagestyle{empty}

\begin{figure}[htpb]
	\centering
	\begin{subfigure}{\textwidth}\centering
		\includegraphics[scale=0.65]{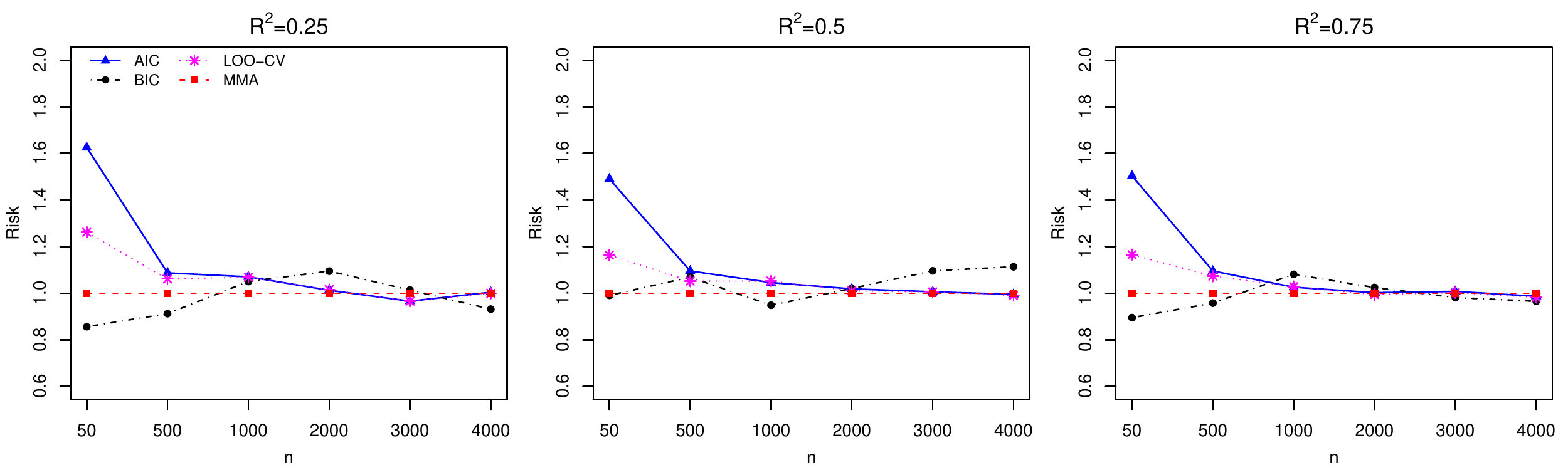}
		\caption{}
	\end{subfigure}
	\begin{subfigure}{\textwidth}\centering
		\includegraphics[scale=0.65]{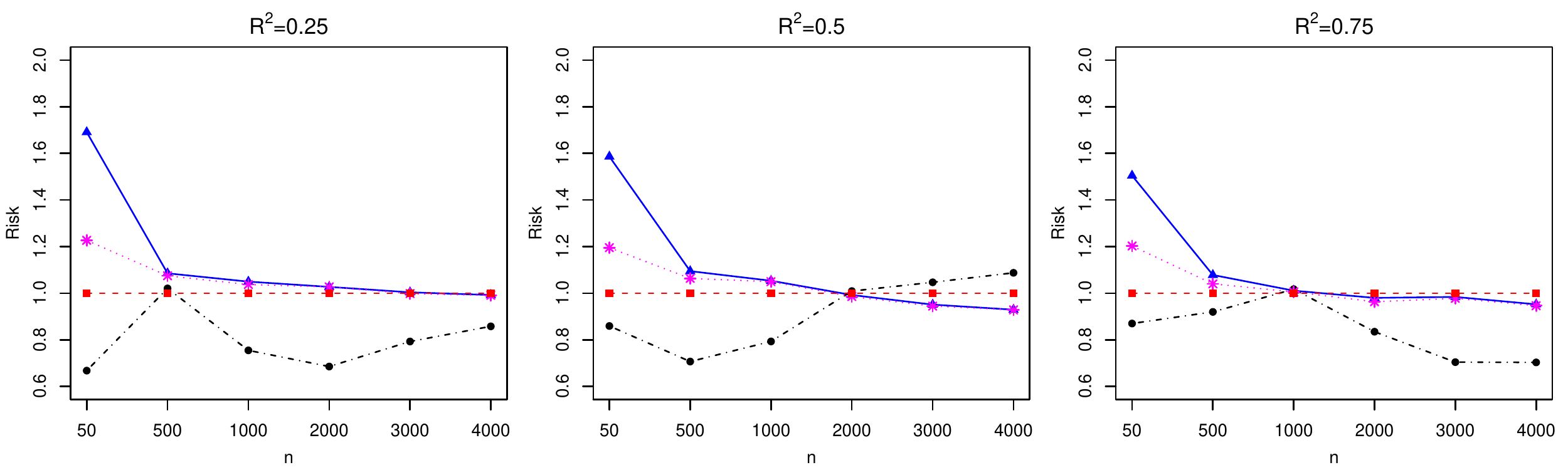}
		\caption{}
	\end{subfigure}
	\begin{subfigure}{\textwidth}\centering
		\includegraphics[scale=0.65]{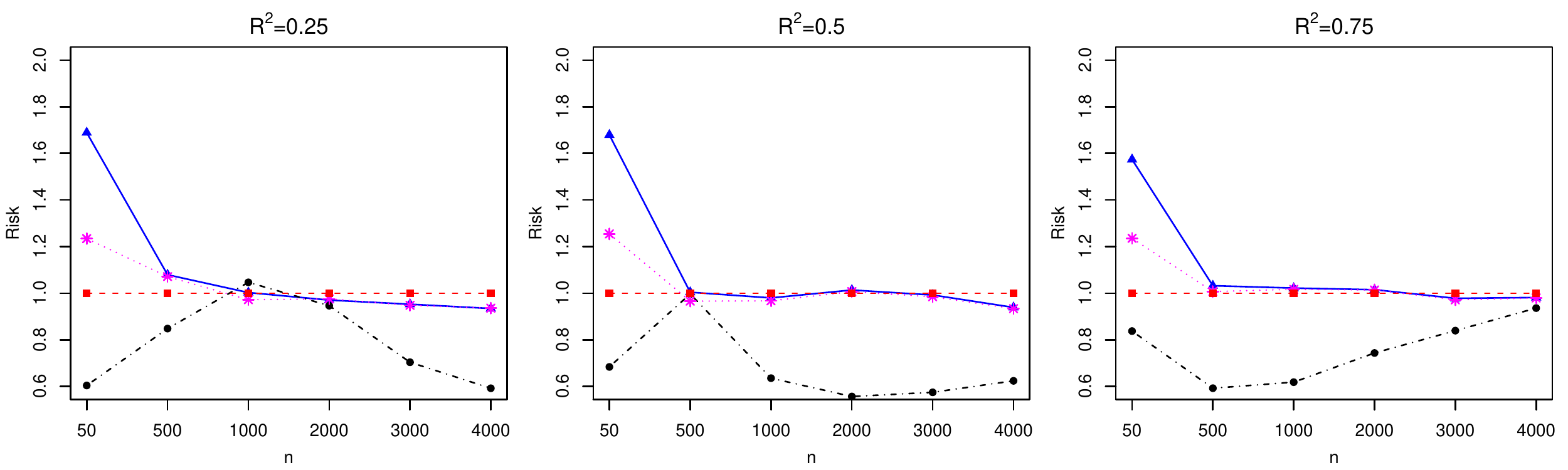}
		\caption{}
	\end{subfigure}
	\caption{Simulation results for Example 3 with the case of fast decaying $\theta_m^*$. Normalized risk functions for AIC, BIC, LOO-CV, and MMA when $\theta_m^*=\exp(-2\alpha_2 m)/\sigma^2$ with $\alpha_2=1$ in row (a), $\alpha_2=1.5$ in row (b), and $\alpha_2=2$ in row (c).}\label{fig:example3_2}
\end{figure}
\end{document}